\documentclass[11pt]{amsart}
\usepackage{a4wide,pdfsync,hyperref}
\usepackage{amsmath,amssymb,esint}
\usepackage{enumerate}
\usepackage{color}
\usepackage{enumitem}

\numberwithin{equation}{section}

\allowdisplaybreaks

\newtheorem{theorem}{Theorem}[section]
\newtheorem{lemma}[theorem]{Lemma}

\newtheorem{corollary}[theorem]{Corollary}
\newtheorem{proposition}[theorem]{Proposition}
\newtheorem{remark}[theorem]{Remark}

\newcommand{\eps}{\varepsilon}

\newcommand{\beqq}{\begin{eqnarray}}
\newcommand{\enqq}{\end{eqnarray}}

\newcommand{\enn}{\end{equation}}
\newcommand{\bef}{\begin{proof}}
\newcommand{\enf}{\end{proof}}

\let\e=\varepsilon

\let\f=\frac

\let\om=\omega

\let\na=\nabla

\let\pa=\partial

\def\dv{\mbox{div}}

\def\curl{\mathop{\rm curl}\nolimits}

\def\ef{\hphantom{MM}\hfill\llap{$\square$}\goodbreak}

\newcommand{\beq}{\begin{equation}}
\newcommand{\eeq}{\end{equation}}
\newcommand{\ben}{\begin{eqnarray}}
\newcommand{\een}{\end{eqnarray}}
\newcommand{\beno}{\begin{eqnarray*}}
\newcommand{\eeno}{\end{eqnarray*}}

\begin{document}
\title[Point Vortex Existence and Uniqueness]{
The Navier-Stokes equations in $\mathbb R^2_+$ with point vortex initial data: construction of the solution
}

\author[C. Wang]{Chao Wang}
\address{School of Mathematical Sciences\\ Peking University\\ Beijing 100871, China}
\email{wangchao@math.pku.edu.cn}

\author[J. Yue]{Jingchao Yue}
\address{School of Mathematical Sciences\\ Peking University\\ Beijing 100871, China}
\email{wasakarumi@163.com}

\author[Z. Zhang]{Zhifei Zhang}
\address{School of Mathematical Sciences\\ Peking University\\ Beijing 100871, China}
\email{zfzhang@math.pku.edu.cn}

\date{\today}

\maketitle

\bigskip

\begin{abstract}
This is the first of two papers concerning the asymptotic behavior of the incompressible Navier-Stokes equations in a half-space at high Reynolds numbers, with initial data given by a point vortex. In the present work, we establish the existence and uniqueness of solutions subject to the non-slip boundary condition. This result was established in \cite{Ken} under the condition that the total mass is sufficiently small. Here, we eliminate the smallness assumption by analyzing the linearized operator near the point vortex and constructing a tailored functional framework-one designed to capture the distinct behaviors of the solution in the vicinity of the point vortex and the boundary, respectively.
\end{abstract}

\section{Introduction}

\subsection{Presentation of the problem and main result}

In this paper and its companion Part II, we investigate the asymptotic behavior of the incompressible Navier-Stokes equations in a half-plane at high Reynolds numbers, with point vortex initial data. This study is particularly challenging due to the singular initial data and the boundary layer, both of which can introduce singularities. To the best of our knowledge, the existence and uniqueness of solutions to the Navier-Stokes equations remain an open problem-one that constitutes the main objective of the present work. Moreover, this work serves as a prelude to the second paper in this series.

To simplify the notations, we consider the incompressible Navier-Stokes system in $\mathbb R^2_+$, where the viscosity coefficient equals to $1$, namely,
\begin{align}\label{eq: NS}
	\left\{
	\begin{aligned}
		\pa_t U-\Delta U+U\cdot\nabla U
		+\nabla p&=0,\quad (x,y)\in\mathbb R^2_+\\
		\operatorname{div}U&=0,\quad (x,y)\in\mathbb R^2_+\\
		U|_{y=0}&=0,
	\end{aligned}
	\right.
\end{align}
where $U=(u,v)$ and $p$ denote the fluid velocity and pressure, respectively.  Here we consider the following initial data $U_0$ satisfying
\begin{align}\label{initial data}
	\omega_0=\curl U_0=\pa_xv_0-\pa_y u_0=\alpha\delta_{(x_0,y_0)},\quad\text{with}\quad \alpha\in\mathbb R,\ y_0>0.
\end{align}
where $\delta_{(x_0,y_0)}$ is the Dirac measure concentrated at the point $(x_0,y_0)$.  Without loss of generality, we assume that 
\[
(x_0, y_0)=(0,20).
\] 
The initial velocity $U_0$ is given by the Biot-Savart law:
\ben\label{def: U_0}
U_0= BS_{\mathbb R^2_+}[\om_0] 
  	=\frac{1}{2\pi}\int_{\mathbb R^2_+}\Big(\frac{(X-Y)^\perp}{|X-Y|^2}-\frac{(X-Y^*)^\perp}{|X-Y^*|^2}\Big)\om_0(Y)dY,
\een
  where $X=(x, y), Y=(\tilde x,\tilde y)\in\mathbb R^2$ and $X^*=(x,-y)$.

We introduce  the vorticity $\om=\pa_x v-\pa_y u$, which satisfies 
\begin{align}\label{eq: NS vorticity}
\pa_t\omega+U\cdot\nabla\omega=&\Delta\omega.
\end{align}
In light of the non-slip boundary condition, the boundary condition of the vorticity  is established in \cite{Maekawa}:
 \begin{align}\label{eq: BC of NS}
	(\pa_y+|D_x|)\omega \big|_{y=0}=\pa_y\Delta_D^{-1}(U\cdot\nabla\omega) \big|_{y=0}.
\end{align}

The well-posedness problem with rough initial data has been extensively studied. Ben-Artzi \cite{BA} studied the Cauchy problem for $\omega_0\in L^1(\mathbb R^2)$. For more singular initial vorticity $\om_0\in \mathcal{M}(\mathbb R^2)$, the space of all finite measures on $\mathbb R^2$, Cottet \cite{Cottet} and Giga, Miyakawa, and Osada \cite{GMO} independently established the existence result via pointwise estimates of the fundamental solution to the advection-diffusion equation. They further proved the uniqueness result under the condition that the purely atomic component of the Radon measure is sufficiently small. Subsequently, Gallay and Wayne \cite{Gallay 1} established uniqueness for initial vorticity composed of arbitrarily large Dirac masses. By synthesizing techniques from \cite{GMO} and \cite{Gallay 1}, Gallagher and Gallay \cite{Gallagher} proved uniqueness for initial vorticity in $\mathcal{M}(\mathbb R^2)$. The core insight in \cite{Gallagher, Gallay 1} resides in the self-similar transformation, which reduces the vorticity equation to
\begin{align*}
	\pa_\tau\mathcal W_R
  	+\alpha\mathcal V_R\cdot\nabla_\eta G
  	+\alpha\mathcal V^G\cdot\nabla_\eta\mathcal W_R
  	-\mathcal L \mathcal W_R=\textrm{small terms},
\end{align*}
where $\om=G+\mathcal W_R$, $\mathcal V_R=BS_{\mathbb R^2}[\mathcal W_R]$ and $\mathcal V^G=BS_{\mathbb R^2}[G]$ with $G(\eta)=\frac{1}{4\pi} e^{-|\eta|^2/4}$. Let $T_\alpha(\tau)$ denote the semigroup generated by the linearized operator above.  Based on the  properties of $T_\alpha(\tau)$, the authors in  \cite{Gallagher, Gallay 1} eliminated the smallness assumption.


In a half space, Abe \cite{Ken} employed a solution formula associated with Stokes flow to derive the following expression:
\beno
\om(t) = T(t) \om_0+\int_0^t \nabla^\bot \cdot S(t-s) \mathbb{P} (\om u^\bot) ds,
\eeno
where $T(f)=-\na^\bot S(t) BS_{\mathbb R^2_+}[f] $, $S(t)$ denotes the Stokes semigroup and $\mathbb{P}$ denotes the Helmholtz projection. The author further established existence and uniqueness for initial vorticities given by Radon measures with a small purely atomic component. That is to say,

\begin{theorem}\label{thm: Ken}$($\cite{Ken}$)$
	There exists $\delta>0$ such that for initial data $\omega_0\in \mathcal{M}(\overline{\mathbb R^2_+})$ with $\|\omega_{0,pp}\|_M \leq \delta$, where $\omega_{0,pp}$ denotes the purely atomic part, there exists a unique solution $(\omega, U)$ satisfying \eqref{eq: NS vorticity} and 
	\begin{align*}
		\omega\in BC_w([0,+\infty),\mathcal{M}(\overline{\mathbb R^2_+})),
		\qquad U\in BC_w([0,+\infty),L^{2,\infty}(\mathbb R^2_+) ).
	\end{align*}
\end{theorem}

The proof of the above theorem relies on an $L^1-$estimate for the solution operator associated with the vorticity equations of Stokes flow. Following the arguments in \cite{Ken}, the Stokes operator cannot control the nonlinear terms without smallness assumptions. This motivates us to develop a novel functional framework to remove such smallness constraints.

\medskip

In this paper, we investigate point vortex solutions without any smallness condition. The main result is stated as follows.

\begin{theorem}\label{main theorem}
Let the initial vorticity $\omega_0$ be given by  \eqref{initial data}. Then the Navier-Stokes equations \eqref{eq: NS vorticity} admit a unique global solution $(U, \omega)\in L^{2,\infty}(\mathbb R^2_+)\times L^1(\mathbb R^2_+)$. Moreover, the vorticity $\omega$ satisfies 
	\begin{align}\label{initial limit}
		\omega(t)\rightharpoonup \alpha\delta_{(x_0,y_0)}-u_0\delta_{\partial \mathbb R_+^2},\quad \text{vaguely in}\  \mathcal{M}(\overline{\mathbb R^2_+})\ \text{as}\ t\rightarrow0^+,
	\end{align}
	where $u_0$ denotes the first component of the initial velocity $U_0=(u_0,v_0)$ defined in \eqref{def: U_0}.

\end{theorem}

\begin{remark}
	The limit of the solution as $t\rightarrow0^+$ given in \eqref{initial limit} does not conflict the initial data \eqref{initial data}, since the initial velocity $U_0$ does not satisfy the non-slip boundary condition, which leads to the formation of an initial layer. This behavior is rigorously justified in Section \ref{sec: Convergence of the integral equations}. Consequently, the vorticity exhibits a discontinuity at $t=0$. To address this discontinuity, we introduce an initial layer corrector, as shown in \eqref{omega c} or \eqref{def: omega c without delta}.
\end{remark}

\begin{remark}

Regarding the asymptotic behavior of the solution constructed in this paper, we have the following more precise description as $t\rightarrow0$:
	\begin{align*}
		\lim_{t\rightarrow0}\|\chi_{vp}\big(\omega(t)-\frac{\alpha}{t}G(\frac{\cdot-X_0}{\sqrt{t}})\|_{H^1_m}
		+\|\chi_m\omega(t)\|_{L^2_m}
		+\|\chi_b(\omega-\omega_c)(t)\|_{Y(t)}=0,
	\end{align*}
	where $\chi_{vp},\chi_m,\chi_b$ are smooth cut-off functions defined in Section \ref{Notations}, $H^1_m$ is a weighted Sobolev space defined in \eqref{def: Evp(t)} via self-similar variables, $L^2_m$ is a weighted $L^2$ space defined in \eqref{def: Em(t)}, $Y(t)$ is a weighted $L^1_y$ space defined in \eqref{def: Yk(t)}, $\omega_c$ is the initial layer corrector defined in \eqref{def: omega c without delta}.   Moreover, uniqueness holds in the sense that the solution remains sufficiently close to the Lamb-Oseen vortex near the point vortex and to the initial layer corrector near the boundary within the functional space described above.
	
\end{remark}

\begin{remark}
	It is worth mentioning that Dalibard and Gallay recently established the same result in \cite{Dalibard} using a different approach. One of the key distinctions lies in the function spaces we employ. We have fully utilized the conditions of being curl-free and divergence-free near the boundary, which implies that the solution is analytic in the vicinity of the boundary.  Our ultimate goal is to prove the inviscid limit, and therefore our framework is designed to be applicable to the inviscid limit as well. In our next paper, we will address the inviscid limit problem, and both articles will employ the same framework for handling boundary layer problem.  
\end{remark}

\begin{remark}
	For general Radon measure data $\nu$ supported away from the boundary $y=0$, we can decompose $\nu$ into two parts:
	\begin{align*}
		\nu=\nu_0+\sum_{i=1}^N \alpha_i\delta_{X_i},
	\end{align*}
	here $\nu_0$ possesses small purely atomic part and $X_i$ denotes the position of point vortex. We decompose the solution as $\omega=\omega^{(0)}+\sum_{i=1}^N\omega^{(i)}, U=U^{(0)}+\sum_{i=1}^N U^{(i)}$ where
	\begin{align*}
		\left\{
		\begin{aligned}
			&\pa_t\omega^{(i)}
			+U\cdot\nabla\omega^{(i)}
			=\Delta\omega^{(i)},\quad \text{for}\ 0\leq i\leq N,\\
			&\omega^{(0)}|_{t=0}=\nu_0,\quad \omega^{(i)}|_{t=0}=\alpha_i \delta_{X_i},\quad \text{for}\ 1\leq i\leq N.
		\end{aligned}
		\right.
	\end{align*}
By employing the methods in \cite{Gallagher} to treat $(\omega^{(0)}, U^{(0)})$ and the methods developed in this paper to handle $(\omega^{(i)}, U^{(i)})$,  we are able to construct solutions to \eqref{eq: NS vorticity} with general Radon measure data supported away from the boundary.	
\end{remark}

\subsection{Main strategies}
In this subsection, we outline the proof strategy for Theorem \ref{main theorem}. First, we regularize the initial data as $\omega_0^\delta$, which generates a global solution. We aim to establish uniform estimates for this global solution with respect to $\delta$.

Since the point vortex lies in $\mathcal M(\mathbb R^2_+)$-a critical function space-we can only expect the following uniform estimates over $\mathbb R^2_+$ via the Biot-Savart law:
\begin{align}\label{uniform in whole R2+}
	\sup_{0<t<T}\|\omega(t)\|_{L^1}+t^{1/2}\sup_{0<t<T}\|U(t)\|_{L^\infty}\leq C.
\end{align}
Thus, if we treat the nonlinear term $U\cdot\nabla\omega$ in \eqref{eq: NS vorticity} as a perturbation to the heat equation, we obtain
\begin{align*}
	\|\omega(t)\|_{L^1}
	&\leq \left\|e^{t\Delta}\omega_0\right\|_{L^1}
	+\int_0^t \left\| e^{(t-s)\Delta}(U\cdot\nabla\omega)(s)\right\|_{L^1}ds\\
	&\leq \left\|e^{t\Delta}\omega_0\right\|_{L^1}
	+C\int_0^t(t-s)^{-1/2}\|U(s)\|_{L^\infty}\|\omega(s)\|_{L^1}ds\\
	&\leq \left\|e^{t\Delta}\omega_0\right\|_{L^1}
	+C\sup_{0<t<T}t^{1/2}\|U(t)\|_{L^\infty}\sup_{0<t<T}\|\omega(t)\|_{L^1}.
\end{align*}
Consistent with the approach in \cite{Ken}, we impose a smallness assumption on the initial data to close the estimate. 

To obtain a precise characterization of the vorticity, we seek a solution $\omega$ that adheres to the following structure:
\begin{align}\label{decomp of om intro}
	\omega=\omega_{vp}\chi_{vp}+\omega_b\chi_b+\omega_R,
\end{align}
where $\chi_{vp}$ and $\chi_b$ denote cut-off functions supported in the vicinity of the point vortex and the boundary, respectively. In \eqref{decomp of om intro}, $\omega_{vp}$ and $\omega_b$, which will be constructed below, denote the approximate solutions near the point vortex and the boundary, respectively. Our objective is to derive uniform estimates for $\omega_R$.

To remove the smallness condition, we adopt some ideas from \cite{Gallagher} and expect the dominate term near the point vortex to be the Oseen vortex. Namely, we set $\omega_{vp}=\frac{\alpha}{\delta}G(\frac{\cdot-X_0}{\delta^{1/2}})$ with $G(\eta)=\frac{1}{4\pi}e^{-\frac{|\eta|^2}{4}}$. We introduce self-similar coordinates and  reformulate the Navier-Stokes system as
\begin{align}\label{eq of WR in sec 1}
	\pa_\tau\mathcal W_R
  	+\alpha\mathcal V_R\cdot\nabla_\eta G
  	+\alpha\mathcal V^G\cdot\nabla_\eta\mathcal W_R
  	-\mathcal L \mathcal W_R= -U^b\cdot\na \mathcal W_R+\cdots ,
\end{align}
where $U^b:=BS_{\mathbb R^2_+}[\chi_b\omega]$ arises from the boundary interactions, see Section 2.2 for the details. Here, we treat the right-hand side of \eqref{eq of WR in sec 1} as a perturbation of the linearized operator on the left,  
which generates a semigroup denoted by $T_\alpha(\tau)$.

 Due to \eqref{eq of WR in sec 1}, $\mathcal W_R$ admits the Duhamel-type representation
\beno
\mathcal W_R =\int_{\tau_0}^\tau T_\alpha(\tau-\tau') (U^b\cdot\na \mathcal W_R) +\cdots,
\eeno
which consequently yields a regularity estimate by means of semigroup estimates:
\beno
\| \mathcal W_R\|_{H^1} \leq C  t^{\f12} \|U^b\|_{L^\infty}  \cdot \| \mathcal W_R\|_{H^1}+\cdots.
\eeno

{\bf The key remaining task is to establish the estimates for $\|U^b\|_{L^\infty}$.} The direct use of the estimates in \eqref{uniform in whole R2+}, such as $\|U^b(t)\|_{L^\infty} \leq t^{-\frac12}$, would still entail a smallness assumption to close the energy argument. Hence, a sharper estimate for $U^b$ is necessary. Based on the Biot-Savart law, we have (see Lemma \ref{velocity estimates 2})
\begin{align}\label{est of Ub intro}
	\|U^b(s)\|_{L^\infty}=\|BS_{\mathbb R^2_+}[\chi_b\omega] \|_{L^\infty}
	\leq C\|\chi_b\omega(s)\|_{Y(s)},
\end{align}
where $\|\cdot\|_{Y(s)}$ is a weighted $L^1_y$ norm to describe the vorticity near the boundary.

All that remains is to establish the uniform boundedness of $\|\chi_b\omega(s)\|_{Y(s)}$, which implies $\|U^b(s)\|_{L^\infty} \leq C$, a sharper estimate for $U^b$.
When estimating the vorticity near the boundary, we encounter an additional difficulty: {\bf the emergence of an initial layer}. Since the initial velocity $U_0$ is derived from the point vortex $\omega_0$ via the Biot-Savart law, we have $u_0(x, 0)\neq 0$, which violates the non-slip boundary condition $u(t, x, 0)=0$. The initial layer results in the discontinuity of $\omega$ at  $t=0$. Specifically,
\begin{align*}
	\lim_{t\rightarrow0^+}\omega\neq \omega_0.
\end{align*}
 To resolve this inconsistency, we introduce a corrector $u_c$ with initial data $u_c(0, x, y)=u_0(x, 0)$ subject to the Dirichlet boundary condition. The corresponding vorticity corrector is defined as $\omega_c=\curl (u_c,0)$, which has the same initial layer as $\omega$. Thus,
 \begin{align*}
 	\lim_{t\rightarrow0^+}(\omega-\omega_c)=(\omega-\omega_c)|_{t=0}.
 \end{align*}
  We take $\omega_b=\omega_c$ in \eqref{decomp of om intro}. After introducing $\omega_c$, we rewrite the vorticity equation near the boundary as
	\begin{align}\label{intro:eq of o-oc diff}
	\left\{
	\begin{aligned}
&\pa_t\big(\chi_b(\omega- \omega_c)\big)
		-\Delta\big(\chi_b(\omega- \omega_c)\big)=U\cdot\nabla\big(\chi_b(\omega- \omega_c)\big)+\cdots ,\\
		&(\pa_y+|D_x|)(\omega-\omega_c)|_{y=0}
		=B.
	\end{aligned}
	\right.
\end{align}
The estimate of $U$ in \eqref{uniform in whole R2+} can not close the estimate as before. In order to obtain a sharper estimate of $U$ near the boundary, we introduce the decomposition:
\begin{align*}
	U=BS_{\mathbb R^2_+}[\omega]=BS_{\mathbb R^2_+}[\chi_b\omega]+BS_{\mathbb R^2_+}[\chi_{vp}\omega]+small\ terms.
\end{align*}
The first term is uniformly bounded by \eqref{est of Ub intro}. The fact $dist(\chi_b,\chi_{vp})>0$ yields the uniform boundedness of the second term. Thus, we obtain $\|U(t)\|_{L^\infty}\leq C$ near the boundary, which is sharper than \eqref{uniform in whole R2+}.

By taking Fourier transformation in $x$-direction, we transfer \eqref{intro:eq of o-oc diff} into the following integral equation:
\begin{align*}
	&\chi_b(\omega-\omega_c)_\xi(t,y)\\
	&=\int_0^t\int_0^{+\infty}\Big(H_\xi(t-s,y,z)+R_\xi(t-s,y,z)\Big)\Big(U\cdot\nabla\big(\chi_b(\omega-\omega_c)\big)\Big)_\xi(t,z)dzds+\cdots,
\end{align*}
where $\xi$ denotes the frequency variable of $x$ direction, the convolution kernel $H_\xi(t-s,y,z)$ arises from the $\pa_y$ part in the boundary condition of \eqref{intro:eq of o-oc diff}, and $R_\xi(t-s,y,z)$ arises from the $|D_x|$ part which loses one derivative. Since the initial data vanishes near the boundary, we introduce the weight function $e^{\eps_0(1+\mu-y)_+|\xi|}$ into $Y(s)$ to overcome the loss of one derivative. Subsequently, by means of the following estimate (see Lemma \ref{analytic recovery}), we can convert the derivative loss into a small divisor problem:
\begin{align*}
	e^{\eps_0(1+\mu-y)_+|\xi|}|(\pa_x f)_\xi(y)|
		\leq \frac{C}{\widetilde\mu-\mu}
		e^{\eps_0(1+\widetilde\mu-y)_+|\xi|}|f_\xi(y)|.
\end{align*}
These techniques are commonly employed in inviscid limit problems, as adopted in \cite{Maekawa, HWYZ}.\smallskip

Finally, it is worth noting that if we employ the velocity equations \eqref{eq: NS} to estimate $U$ near the boundary, pressure estimates are indispensable. However, the pressure is a nonlocal term whose estimation requires the behavior of $U$ over the entire  $\mathbb R^2_+$. The lack of regularity of the point vortex gives rise to difficulties in estimating $p$. 
For this reason, we perform our estimates using the vorticity equation rather than the velocity equation.


\subsection{Notations}\label{Notations}
This subsection outlines frequently used notations.

\begin{enumerate}[label=(\arabic*)]
	\item In this paper, velocity is recovered from the vorticity. We frequently utilize the following notations to denote the Biot-Savart law in $\mathbb R^2$ or $\mathbb R^2_+$ for convenience:
  \begin{align}\label{BS law formulation 1}
  	BS_{\mathbb R^2}[f]:&=\nabla^\perp\Delta^{-1} f
  	=\frac{1}{2\pi}\int_{\mathbb R^2}\frac{(X-Y)^\perp}{|X-Y|^2}f(Y)dY,
  \end{align}
  \begin{align}\label{BS law formulation 2}
  	BS_{\mathbb R^2_+}[f]:&=\nabla^\perp\Delta_D^{-1} f
  	=\frac{1}{2\pi}\int_{\mathbb R^2_+}\big(\frac{(X-Y)^\perp}{|X-Y|^2}-\frac{(X-Y^*)^\perp}{|X-Y^*|^2}\big)f(Y)dY,
  \end{align}
  where $X=(x_1,x_2), Y=(y_1,y_2)\in\mathbb R^2$ and we denote $Y^*:=(y_1,-y_2)$.
  
  \item Let $F=(f_1,f_2)$. Define $\widetilde F(x,y)=(-f_1(x,-y),f_2(x,-y)).$  
  
	\item We use $f_\xi(y)$ to denote the Fourier transform about $x$ variable of function $f(x,y)$.
	
	\item $C_0$ denotes a constant independent of $\eps_0,\eps_1,\gamma,t,\delta$ and $C$ independent of $\gamma, t, \delta$.
	
	\item $\nabla$ denotes derivative for variable $(x,y)$ and $\nabla_\eta$ for self-similar variable $\eta$ defined later.
	
	\item For norm $\|\cdot\|_X$, $\|\cdot\|_Y$, let $\|(1,x)f\|_X:=\|f\|_X+\|xf\|_X$ and $\|f\|_{X\cap Y}:=\|f\|_X+\|f\|_Y$.
	
	\item Smooth cut-off function $\chi_{vp}$ is used to describe the vorticity behavior near the point vortex and is defined as
  \begin{align}\label{def chi vp}
  	\chi_{vp}(x,y)=
  	\left\{
  	\begin{aligned}
  		&1,\quad |(x,y)-(0,20)|\leq5,\\
  		&0,\quad |(x,y)-(0,20)|\geq6.
  	\end{aligned}
  	\right.
  \end{align}
  
  For the middle region between the point vortex and boundary, it is necessary to define $\chi_m$ as
  \begin{align}\label{def: chi m}
  	\chi_m(x,y)=
  	\left\{
  	\begin{aligned}
  		&1,\quad|(x,y)-(0,20)|\geq4 \ \text{and}\ y\geq\frac{3}{8},\\
  		&0,\quad|(x,y)-(0,20)|\leq3 \ \text{or}\ y\leq\frac{1}{4}.
  	\end{aligned}
  	\right.
  \end{align}
Near the boundary, it is beneficial to define $\chi_b$ as
  \begin{align}\label{def of chi b}
	\chi_b(y)=
	\left\{
	\begin{aligned}
		&1,\quad y\leq2,\\
		&0,\quad y\geq3.
	\end{aligned}
	\right.
\end{align}
\end{enumerate}

\section{New formulations  and energy functionals} 

\subsection{Approximate solutions}
To obtain a solution to \eqref{eq: NS}, we begin by constructing an approximate solution via mollification of the initial data as follows:
   \begin{align}\label{initial data mollify}
  	\omega_0^\delta(x,y)
  	=\frac{\alpha}{\delta} G\Big(\frac{(x,y)-(0,20)}{\delta^{1/2}}\Big)\chi_{\{|(x,y)-(0,20)|\leq6\}},\quad \forall \delta>0,
  \end{align}
  where $G(\eta)=\frac{1}{4\pi} e^{-|\eta|^2/4}$. We denote $U^\delta_0=(u^\delta_0, v^\delta_0)=BS_{\mathbb R^2_+}[\omega^\delta_0]$. Note that
\beno
\omega_0^\delta \rightharpoonup \omega_0,\quad \textrm{in}\quad \mathcal{M}(\overline{\mathbb R^2_+}).
 \eeno
Furthermore, taking $\omega_0^\delta$ as the initial data for system \eqref{eq: NS} yields smooth solutions $U^\delta$ and $\omega^\delta$, whose existence is guaranteed by Theorem \ref{thm: Ken}. In what follows, we aim to derive uniform estimates for $U^\delta$ that ensure $U^\delta$ converges to a solution $U$ of \eqref{eq: NS} with initial data $U_0$.

The primary difficulty of this problem resides in the interaction between the initial data singularity and the boundary. Accordingly, we decompose the problem into two components: one near the point vortex and the other near the boundary, and apply distinct methodologies to treat each component.

\subsection{Reformulation near the point vortex}

 Multiplying $\chi_{vp}$ on both sides of \eqref{eq: NS vorticity} gives
  \begin{align}\label{eq: NS near vortex point}
  	&\pa_t(\chi_{vp}\omega^\delta)
  	+BS_{\mathbb R^2_+}[\chi_{vp}\omega^\delta]\cdot\nabla(\chi_{vp}\omega^\delta)-\Delta(\chi_{vp}\omega^\delta)\\
  	\nonumber
  	&=-BS_{\mathbb R^2_+}[(1-\chi_{vp})\omega^\delta]\cdot\nabla(\chi_{vp}\omega^\delta)
  	+U\cdot\nabla\chi_{vp}\omega^\delta
  	-2\nabla\chi_{vp}\cdot\nabla\omega^\delta
  	-\Delta\chi_{vp}\omega^\delta.
  \end{align}
 Here, we aim to show that the equation of $\chi_{vp}\omega^\delta$ behaves analogously to that of the point vortex in the whole space. Thus, we adopt some ideas from the full space setting. 
 First of all, we introduce the self-similar coordinate 
 \beno
 \eta:=\frac{(x,y)-(0,20)}{(t+\delta)^{1/2}},\quad \tau:=\log (t+\delta),
 \eeno
 and
  \begin{align}\label{def: self-similar}
  \frac{\alpha}{t+\delta}\mathcal W^\delta \big(\eta,\tau\big):=	\chi_{vp}(x,y)\omega^\delta(t,x,y). 
  \end{align}
 According to the definition above, we obtain
  \begin{align}\label{supp of w}
  	\operatorname{supp}\mathcal W^\delta\subseteq
  \Big\{|\eta|\leq\frac{6}{(t+\delta)^{1/2}}\Big\},\quad
  	\mathcal W^\delta|_{\tau_0}=\chi_{vp} G,
  \end{align}
  where  $ \tau_0=\log\delta.$ 
  
%

  Based on the new notations introduced above, we rewrite the equation of \eqref{eq: NS near vortex point}. To begin with, by virtue of Lemma \ref{derivation of velocity formula}, we have
  \begin{align}\label{velocity self-similar}
  	BS_{\mathbb R^2_+}[\chi_{vp}\omega^\delta](t,x,y)
  	=\frac{\alpha}{(t+\delta)^{1/2}}\Big\{\mathcal V^\delta\big(\eta,\tau)-
  	\widetilde{\mathcal V^\delta}(\eta+\frac{(0,40)}{(t+\delta)^{1/2}},\tau)\Big\},
  \end{align}
 where 
  $\mathcal V^\delta:=BS_{\mathbb R^2}[\mathcal W^\delta]$ and $\widetilde F(x,y)$ is defined by
   \begin{align}
  	\widetilde F(x,y)=(-f_1(x,-y),f_2(x,-y))
  \end{align}   
  with $F=(f_1,f_2):\mathbb R^2\rightarrow\mathbb R^2$.  

Thus, we deduce the equation of $\mathcal W^\delta$:
  \begin{align}\label{eq: mathcal W}
  	&\quad\pa_\tau \mathcal W^\delta
  	+\alpha\big\{ \mathcal V^\delta(\eta,\tau)-\widetilde{\mathcal V^\delta}(\eta+e^{-\tau/2}(0,40),\tau)\big\}\cdot\nabla_\eta\mathcal W^\delta(\eta,\tau)
  	-\mathcal L \mathcal W^\delta\\
  	\nonumber
  	&=-e^{\tau/2}BS_{\mathbb R^2_+}[(1-\chi_{vp})\omega^\delta]\cdot\nabla_\eta\mathcal W^\delta(\eta,\tau)\\
  	\nonumber
  	&\quad+\frac{(t+\delta)^2}{\alpha}U^\delta\cdot\nabla\chi_{vp}\omega^\delta
  	-\frac{2(t+\delta)^2}{\alpha}\nabla\chi_{vp}\cdot\nabla\omega^\delta
  	-\frac{(t+\delta)^2}{\alpha}\Delta\chi_{vp}\omega^\delta,
  \end{align}
  where $\mathcal L:=\Delta_\eta+\frac{1}{2}\eta\cdot\nabla_\eta+1$.
  
 Moreover, we make the following decomposition
  \begin{align}\label{decomp 1}
  	\mathcal W^\delta:=\chi_{vp} G+\mathcal W_R,\quad
  	\mathcal V^\delta:=BS_{\mathbb R^2}[\chi_{vp}G]+\mathcal V_R
  	=\mathcal V^G
  	-BS_{\mathbb R^2}[(1-\chi_{vp})G]
  	+\mathcal V_R,
  \end{align}
  where  $\mathcal V^G:=BS_{\mathbb R^2}[G]=\frac{1}{2\pi}\frac{\eta^\perp}{|\eta|^2}(1-e^{-|\eta|^2/4})$. A direct computation gives $\widetilde{\mathcal V^G}=\mathcal V^G$, thus
  \begin{align}\label{decomp 2}
  	\widetilde{\mathcal V^\delta}=\mathcal V^G-\widetilde{BS_{\mathbb R^2}[(1-\chi_{vp})G]}
  	+\widetilde{\mathcal V_R}.
  \end{align}
  
  \medskip
  
  By substituting \eqref{decomp 1}, \eqref{decomp 2} into \eqref{eq: mathcal W}, we obtain the following lemma.

\begin{lemma}\label{lem: eq of WR}
	It hods that
	\begin{align}\label{eq: mathcal W R}
		\pa_\tau\mathcal W_R
  	+\alpha\mathcal V_R\cdot\nabla_\eta G
  	+\alpha\mathcal V^G\cdot\nabla_\eta\mathcal W_R
  	-\mathcal L \mathcal W_R
  	=\sum_{i=1}^7 F_i,\qquad \mathcal W_R|_{\tau=\tau_0}=0,
	\end{align}
where $F_1\sim F_7$ are defined by
	\begin{align*}
	F_1=&\alpha\Big(BS_{\mathbb R^2}[(1-\chi_{vp})G](\eta)
  	-\widetilde{BS_{\mathbb R^2}[(1-\chi_{vp})G]}(\eta+\frac{(0,40)}{e^{\tau/2}})\\
  	&+\mathcal V^G(\eta+\frac{(0,40)}{e^{\tau/2}})\Big)\cdot\nabla_\eta(\chi_{vp}G+\mathcal W_R),\\
	F_2=&\frac{1}{2}\eta\cdot\nabla_\eta( \chi_{vp})G+2\nabla_\eta(\chi_{vp})\cdot\nabla_\eta G
  	+\Delta_\eta(\chi_{vp}) G,\\
	F_3=&\alpha(1-\chi_{vp})\mathcal V_R\cdot\nabla_\eta G
	-\alpha\mathcal V_R\cdot\nabla_\eta(\chi_{vp}) G,\\
	F_4=&\alpha\widetilde{\mathcal V_R}(\eta+\frac{(0,40)}{e^{\tau/2}},\tau)\cdot\nabla_\eta(\chi_{vp}G),\\
	F_5=&-\alpha\mathcal V_R\cdot\nabla_\eta\mathcal W_R
  	+\alpha\widetilde{\mathcal V_R}(\eta+\frac{(0,40)}{e^{\tau/2}},\tau)\cdot\nabla_\eta\mathcal W_R,\\
	F_6=&-\alpha e^{\tau/2}BS_{\mathbb R^2_+}[(1-\chi_{vp})\omega]\cdot\nabla_\eta(\chi_{vp}G+\mathcal W_R),\\
	F_7=&\frac{(t+\delta)^2}{\alpha}U\cdot(\nabla\chi_{vp})\omega^\delta
  	-\frac{2(t+\delta)^2}{\alpha} (\nabla\chi_{vp})\cdot\nabla\omega^\delta
  	-\frac{(t+\delta)^2}{\alpha}(\Delta\chi_{vp})\omega^\delta.
\end{align*}

\end{lemma}


\medskip

\subsection{Reformulation near the boundary}

\subsubsection{Initial layer corrector}
Before presenting the new formulation of the vorticity near the boundary, we introduce an initial layer corrector accounting for the mismatch between the initial data and the boundary conditions. More precisely, by \eqref{BS law formulation 2} and \eqref{initial data mollify}, we have  $U_0^{\delta}|_{y=0}=BS_{\mathbb R^2_+}[\omega_0^\delta]=\big(u_0^\delta(x) ,0\big)$, where $u^\delta_0(x)$ satisfies 
\begin{align}\label{initial velocity}
	u^\delta_0(x)
	=\frac{1}{\pi}\int_{\{|(y_1,y_2)-(0,20)|\leq6\}} \frac{y_2}{(x-y_1)^2+y_2^2}\cdot\frac{\alpha}{\delta}G\big(\frac{(y_1,y_2)-(0,20)}{\delta^{1/2}}\big)dy_1dy_2,
\end{align}
which may not vanish on the boundary. This mismatch gives rise to the initial layer. To remedy this, we  introduce an initial layer corrector $u^\delta_c$ defined by
\begin{align}\label{def: uc}
	\left\{
	\begin{aligned}
		&\pa_t u_c^\delta-\pa_y^2 u_c^\delta=0,\\
		&u_c^\delta|_{t=0}=u_0^\delta(x)\chi_b(y),\qquad
		u_c^\delta|_{y=0}=0,
	\end{aligned}
	\right.
\end{align}
where $\chi_b$ is defined \eqref{def of chi b}. It is easy to see that
\begin{align*}
	u_c^\delta(t,x,y)
	=u_0^\delta(x)\int_0^{+\infty}\frac{1}{(4\pi t)^{1/2}}\big(e^{-\frac{(y-z)^2}{4t}}-e^{-\frac{(y+z)^2}{4t}}\big)\chi_b(z)dz.
\end{align*}
Moreover, we define 
\begin{align}\label{omega c}
	&\omega_c^\delta(t,x,y):=-\pa_y u_c^\delta(t,x,y)\\
	\nonumber
	&=-2u_0^\delta(x)\cdot\frac{1}{(4\pi t)^{1/2}}e^{-\frac{y^2}{4t}}
	-u_0^\delta(x)\int_0^{+\infty}\frac{1}{(4\pi t)^{1/2}}\big(e^{-\frac{(y-z)^2}{4t}}+e^{-\frac{(y+z)^2}{4t}}\big)\chi'_b(z) dz.
\end{align}

From the definition of $u_0^\delta$ (see \eqref{initial velocity}), we obtain
	\begin{align*}
		(u_0^\delta)_\xi
		=\int_{\{|(y_1,y_2)-(0,20)|\leq6\}} e^{2\pi iy_1}e^{-2\pi y_2|\xi|}\cdot\frac{\alpha}{\delta}G\big(\frac{(y_1,y_2)-(0,20)}{\delta^{1/2}}\big)dy_1dy_2 .
	\end{align*}
A direct computation yields the following estimates of $\omega_c^\delta$:

\begin{lemma}\label{est of omega c}
There exist constants $C_0, C'$ and $T_0$ independent of of $\delta$, such that for sufficiently small $\delta$, we have
	\begin{align}\label{boundedness of omega c}
	\sup_{0<t<T_0}\left\|\left\|e^{C'|\xi|}e^{\frac{C'y^2}{t}}(1,x)\big(\pa_x^i(y\pa y)^j\pa_y^k\omega^\delta_c(t)\big)_\xi\right\|_{L^1_y}\right\|_{L^1_\xi\cap L^2_\xi}\leq C_0 t^{-k/2},\quad \text{for}\ i,j,k\geq0,
\end{align}
\begin{align}\label{boundedness of omega c 2}
	\left\|e^{C'|\xi|}(\pa_y+|\xi|)((1,x)\omega^\delta_c)_\xi|_{y=0}(t)\right\|_{L^1_\xi\cap L^2_\xi}\leq C_0 t^{-1/2}.
\end{align}
\end{lemma}

\medskip

\subsubsection{The system of $\omega^\delta-\omega^\delta_c$}

Now, we derive the system of $\omega^\delta-\omega^\delta_c$ near the boundary.
 \begin{lemma}\label{lem: eq of omega-omega c}
It holds that	
\begin{align}\label{omega-omega c}
	\left\{
	\begin{aligned}
		&\pa_t\big(\chi_b(\omega^\delta- \omega^\delta_c)\big)
		-\Delta\big(\chi_b(\omega^\delta- \omega^\delta_c)\big)
		=N^\delta,\\
		&\lim_{t\rightarrow 0^+}\big(\chi_b(\omega^\delta- \omega^\delta_c)\big)=b^\delta,\\
		&(\pa_y+|D_x|)(\omega^\delta-\omega^\delta_c)|_{y=0}
		=B^\delta,
	\end{aligned}
	\right.
\end{align}
where $(b^\delta, N^\delta, B^\delta)$ is defined by 
\begin{align*}
b^\delta=& u_0^\delta \chi_b\chi_b',\\
N^\delta=& -\chi_b U\cdot\nabla(\omega^\delta-\omega^\delta_c)
		-\chi_b U\cdot\nabla\omega^\delta_c
		+\chi_b\pa_x^2\omega^\delta_c -(\chi_b''\omega^\delta+2\chi_b'\pa_y\omega^\delta)\\
&+(\chi_b''\omega^\delta_c+2\chi_b'\pa_y\omega^\delta_c),\\
B^\delta=&	\pa_y\Delta_D^{-1}(U^\delta\cdot\nabla\omega^\delta)|_{y=0}
		-(\pa_y+|D_x|)\omega^\delta_c|_{y=0}.
\end{align*}
 
\end{lemma}

\begin{proof}
The equation and boundary condition of $\chi_b(\omega^\delta- \omega^\delta_c)$ can be obtained from the definitions  of $\chi_b$ and $\omega_c$ together with the equation of $\omega^\delta$. All that remains is to prove the initial condition in \eqref{omega-omega c}.  Indeed, this initial condition was established in \cite{Ken}. We provide a sketch of the proof below. 

Using the solution formula from \cite{Ken}, we rewrite the solution $\omega^\delta$ to \eqref{eq: NS vorticity} with initial data \eqref{initial data mollify} as follows
\begin{align}\label{Ken formula}
	\omega^\delta(t)=T(t)\omega_0^\delta
	+\int_0^t \nabla^\perp\cdot S(t-s)\mathbb P(\omega^\delta (U^\delta)^\perp)(s)ds,
\end{align}
where $S(t)$ denotes the Stokes semigroup on $\mathbb R^2_+$, $\mathbb P$ denotes the Helmholtz projection. Lemma 4.1 in \cite{Ken} states that
\begin{align*}
	T(t)\omega_0^\delta
\rightharpoonup\omega_0^\delta-u_0^\delta \delta_{\partial\mathbb R^2_+} \quad in \ M(\overline{\mathbb R^2_+}) \quad as \ t\rightarrow0.
\end{align*}
Meantime, Section 5 of \cite{Ken} proved that 
\begin{align*}
	\int_0^t \|\nabla^\perp\cdot S(t-s)\mathbb P(\omega^\delta (U^\delta)^\perp)(s)\|_{L^1}ds<+\infty,
\end{align*}
which implies
\begin{align*}
	\lim_{t\rightarrow0} 
	\left\|\int_0^t \nabla^\perp\cdot S(t-s)\mathbb P(\omega^\delta (U^\delta)^\perp)(s)ds \right\|_{L^1}=0.
\end{align*}
Furthermore, by a direct calculation, we have
\begin{align*}
	\omega_c^\delta\rightharpoonup-u_0^\delta\delta_{\partial\mathbb R^2_+}-u_0^\delta\chi_b' \quad in \ M(\overline{\mathbb R^2_+}) \quad as \ t\rightarrow0,
\end{align*}
which implies \eqref{omega-omega c}.  
\end{proof}

Next, we present the new formulation of $\omega^\delta-\omega_c^\delta$.  By employing the solution formula for the heat equation in the half plane derived in \cite{Maekawa}, we obtain
\begin{align}\label{integral eq of omega-omega c-1}
	(\chi_b\omega^\delta -\chi_b\omega_c^\delta)_\xi(t,y)
 	=&\int_0^{+\infty}\big(H_\xi(t,y,z)+R_\xi(t,y,z)\big)b^\delta_\xi(z)dz\\
 	\nonumber
 	&+\int_0^t\int_0^{+\infty}\big(H_\xi(t-s,y,z)+R_\xi(t-s,y,z)\big)N^\delta_\xi(s,z)dzds\\
 	\nonumber
 	&-\int_0^t\big(H_\xi(t-s,y,0)+R_\xi(t-s,y,0)\big)B^\delta_\xi(s)ds,
\end{align}
 where 
\begin{align}\label{def of H xi R xi}
 	H_\xi(t,y,z)=e^{-\xi^2t}\big(g( t,y-z)+g( t,y+z)\big),\quad R_\xi(t,y,z)=\big(\Gamma( t)-\Gamma(0)\big)_\xi(y+z),
 \end{align}
 with 
  \begin{align}\label{def of g-G}
 	g(t, x)=\frac{1}{(4\pi t)^{1/2}}e^{-\frac{x^2}{4t}},\quad \Gamma(t, x, y)=\big(\Xi E\ast G(t)\big)(x, y).
 \end{align}
Here $(\Xi, E, G)$ is defined by 
 \begin{align}\label{def of E-G}
 \Xi=2(\pa_x^2+|D_x|\pa_y\big), \quad  E(x, y)=-\frac{1}{2\pi}\operatorname{log}(\sqrt{|x|^2+|y|^2}), \quad G(t,x,y)=g(t,x)g(t,y).
  \end{align}

Subsequently, we enumerate several properties of $R_\xi$.
 \begin{lemma}\label{prop of R 1}
 It holds that
 \begin{enumerate}[label=(\arabic*)]
 	\item \begin{align*}
 		\pa_y R_\xi(t,y,z)=\pa_z R_\xi(t,y,z).
 	\end{align*}
 	
 	\item \begin{align*}
 		&|(y\pa_y)^j\pa_y^k R_\xi(t,y,z)|\leq C \langle\xi\rangle ^{k+1}e^{-\theta_0 \langle\xi\rangle(y+z)}
 		+\frac{C}{ t^{(k+1)/2}}e^{-\theta_0\frac{(y+z)^2}{ t}}e^{-\frac{\xi^2t}{8}},\quad j,k\leq3,\\
 		&|(y\pa_y)^kR_\xi(t,y,z)|\leq C \langle\xi\rangle e^{-\frac{\theta_0}{2}\langle\xi\rangle(y+z)}
 		+\frac{C}{\sqrt{ t}}e^{-\frac{\theta_0}{2}\frac{(y+z)^2}{ t}}e^{-\frac{\xi^2t}{8}},\quad k\leq3,
 	\end{align*}
 	where $\theta_0$ is a universal constant and $C$ depends only on $\theta_0$.
 	
 	\item \begin{align*}
 		R_\xi(t,y,z)
 		=-2\int_0^t (-\xi^2+|\xi|\pa_y)\big(e^{-s\xi^2}g(s,y+z)\big)ds.
 	\end{align*}
 \end{enumerate}
 	
 \end{lemma}
 
 \begin{proof}
 	For the proof of (1) and (2), we refer to \cite{Kukavica} and \cite{TT Nguyen}. For (3), since $\pa_t G=\Delta G$, 
 	\begin{align*}
 		\Gamma(t)-\Gamma(0)
 		=\Xi E \ast \int_0^t \Delta G(s)ds
 		=-\int_0^t \Xi G(s)ds,
 	\end{align*}
which implies 
 	\begin{align*}
 		R_\xi(t,y,z)
 		=\big(\Gamma( t)-\Gamma(0)\big)_\xi(y+z)
 		=-2\int_0^t (-\xi^2+|\xi|\pa_y)\big(e^{-s\xi^2}g(s,y+z)\big)ds.
 	\end{align*}
 \end{proof}

\medskip

At the end of this subsection, we derive some estimates of initial data $b^\delta$ which are sufficiently small as $t\rightarrow 0$.
\begin{lemma}\label{est of HR b}
There exist positive constants $C, C'$ and $T>0$, such that for $t\in[0,T],j\leq10$, it holds that
	\begin{align*}
		&\left\|\left\|e^{C'|\xi|}e^{\frac{C'y^2}{t}}\int_0^{+\infty}
		(y\pa_y)^j\big(H_\xi(t,y,z)+R_\xi(t,y,z)\big)
		\big((1,x)b
		^\delta\big)_\xi(z) dz\right\|_{L^1_y\cap L^\infty_y(y\leq\frac{3}{2})}\right\|_{L^1_\xi\cap L^2_\xi}\\
		&\leq Ce^{-\frac{C'}{ t}}.
	\end{align*}
\end{lemma}
\begin{proof}

	By virtue of the expression of $H_\xi$ in \eqref{def of H xi R xi} and $R_\xi$ in Lemma \ref{prop of R 1}, we have for $z\in[2,3], y\in[0,\f32],j\leq10$,
	\begin{align*}
		e^{\frac{C'y^2}{t}}\Big| (y\pa_y)^j\big(H_\xi(t,y,z)+R_\xi(t,y,z)\big)\Big|
		\leq \frac{C}{t^{1/2}}e^{-\frac{(y-z)^2}{5t}}
		\leq Ce^{-\frac{C'}{t}},
	\end{align*}
	which implies
	\begin{align*}
		\left\| e^{\frac{C'y^2}{t}} (y\pa_y)^j\big(H_\xi(t,y,z)+R_\xi(t,y,z)\big) \right\|_{L^1_y\cap L^\infty_y(y\leq\frac{3}{2})}
		\leq Ce^{-\frac{C'}{t}}.
	\end{align*}
	Thus, we have
	\begin{align*}
		&\left\|\left\|e^{C'|\xi|}e^{\frac{C'y^2}{t}}\int_0^{+\infty}
		(y\pa_y)^j\big(H_\xi(t,y,z)+R_\xi(t,y,z)\big)
		\big((1,x)b
		^\delta\big)_\xi(z) dz\right\|_{L^1_y\cap L^\infty_y(y\leq\frac{3}{2})}\right\|_{L^1_\xi\cap L^2_\xi}\\
		&\leq Ce^{-\frac{C'}{t}}
		\left\|e^{C'|\xi|}\int_0^{+\infty}
		|\big((1,x)b
		^\delta\big)_\xi(z)| dz \right\|_{L^1_\xi \cap L^2_\xi}.
	\end{align*}
	Taking the Fourier transformation of $b^\delta$, we obtain
	\begin{align*}
		&\big((1,x)b
		^\delta\big)_\xi(z)\\
		&=\frac{\chi_b\chi_b'(z)}{\pi}\int_{\mathbb R^2_+} (1,-2\pi y_2sgn\xi) e^{2\pi iy_1\xi-2\pi y_2|\xi|}\frac{\alpha}{\delta}G(\frac{(y_1,y_2)-(0,20)}{\delta^{1/2}})\chi_{vp}(y_1,y_2)dy_1dy_2,
	\end{align*}
	which gives $\operatorname{supp} b^\delta_\xi\subseteq[2,3]$ and $|b^\delta_\xi(z)|\leq Ce^{-10|\xi|}.$ Therefore, we obtain the desired result.
\end{proof}

\smallskip
  
 \subsection{Energy functionals}
 
 In this subsection, we introduce the energy functionals.
The following energy functional $E_{vp}^\delta(t)$ is to describe the vorticity in the vicinity of the point vortex:
  \begin{align}\label{def: Evp(t)}
  	E_{vp}^\delta(t):=\sup_{\log\delta\leq\tau'<\log t}
  	\Big( \|\mathcal W_R (\tau')\|_{L^2(m)}
  	+\|\nabla_\eta\mathcal W_R(\tau')\|_{L^2(m)}\Big),
  \end{align}
where $ L^p(m)$ is defined by
  \begin{align*}
  	L^p(m):=\Big\{f\in L^p(\mathbb R^2)
  	:\|f\|^p_{L^p(m)}=\int_{\mathbb R^2}|f(\eta)|^p\langle\eta\rangle^{pm}d\eta<+\infty \Big\},\quad 1\leq p<+\infty.
  \end{align*}

 We utilize the following energy functional $E_b^\delta(t)$ to describe the vorticity in the vicinity of the boundary:
\begin{align}\label{def: Eb(t)}
	E_b^\delta(t):=\left\|(1,x)\big(\omega^\delta(t)-\omega_c^\delta(t)\big)\right\|_{Y_1(t)\cap Y_2(t)},
\end{align}
where $Y_k(t)$ is defined by
\begin{align}\label{def: Yk(t)}
	\|f\|_{Y_k(t)}
	=\sup_{\mu<\mu_0-\gamma t}
	\Big(\sum_{i+j\leq2}\|\pa_x^i(y\pa_y)^jf\|_{Y^k_{\mu,t}}
	+(\mu_0-\mu-\gamma t)^\beta\sum_{i+j=3}\|\pa_x^i(y\pa_y)^jf\|_{Y^k_{\mu,t}}\Big),\quad k=1,2, 
\end{align}
 with $\mu_0=\frac{1}{10}, \beta\in(\frac{1}{2},1)$ and
\begin{align*}
	\|f\|_{Y^1_{\mu,t}}
	=\int_{-\infty}^{+\infty}\|e^{\eps_0(1+\mu-y)_+|\xi|}f_\xi\|_{\mu,t}d\xi,\quad
	\|f\|_{Y^2_{\mu,t}}
	=\Big(\int_{-\infty}^{+\infty}\|e^{\eps_0(1+\mu-y)_+|\xi|}f_\xi\|_{\mu,t}^2d\xi\Big)^{\frac{1}{2}},
\end{align*}
and
\begin{align*}
	\| f\|_{\mu,t}
	=\int_0^{1+\mu}e^{\eps_0(1+\mu)\frac{y^2}{t}}|f(y)|dy.
\end{align*}
Here, the parameter $\eps_0$ depends on the initial data and will be chosen later.

Last but not least, the following energy functional $E^\delta_m(t)$ is to describe the vorticity in the middle region:
   \begin{align}\label{def: Em(t)}
   	E_m^\delta(t):=&\sup_{0<s<t}\left\|e^\Psi\psi\chi_m\omega^\delta(s)\right\|_{L^2\cap L^4}+\left\|e^\Psi\psi \na(\chi_m\omega^\delta)\right\|_{L^2(0, t;L^2)}\\
   	\nonumber
   	&+e^{\frac{5\eps_0}{t}}\sup_{0<s<t}\|(1,x)\omega^\delta(s)\|_{H^4(\frac{7}{8}\leq y\leq 4)},
   \end{align}
   where
  \begin{align}\label{def: psi}
\psi(x,y)=y^2(1+|x|),\quad  \Psi(t,x,y)=\lim_{t_0\rightarrow0}\Psi_{t_0}(t,x,y),
  \end{align}
  with
     \begin{align*}
  	\Psi_{t_0}(t,x,y)=\frac{20\eps_0}{t+t_0}\big(1-\gamma t-\theta(x,y)\big)_+^2,
  \end{align*}
and $\theta(x,y)$ is defined as 
  \begin{align}\label{def: theta}
  	\theta(x,y)=
  	\left\{
  	\begin{aligned}
  		&1,\quad|(x,y)-(0,20)|\leq4 \ \text{or}\ y\leq\frac{3}{8},\\
  		&\frac{1}{4},\quad|(x,y)-(0,20)|\geq5 \ \text{and}\ y\geq\frac{1}{2}.
  	\end{aligned}
  	\right.
  \end{align}
 Here,  $\gamma, \frac1{\eps_0}$ are large quantities which will be determined later. Define $\Gamma(t):=\{(x,y)\in\mathbb R^2_+: 1-\gamma t-\theta(x,y)\geq0\}$.

  Now, we are in a position to define the total energy functional $E^\delta(t)$:
\begin{align}\label{def: E(t)}
	E^\delta(t):=E^\delta_{vp}(t)
	+E^\delta_m(t)
	+E^\delta_b(t).
\end{align}

At the end of this subsection, we enumerate uniform estimates for the energy functionals defined above.

\begin{proposition}\label{prop: uniform estimates for E(t)}
	There exist positive constants $T_0, \gamma_0, C$ independent of $\delta$, such that for all $0<t_0<T_0$ and $\gamma>\gamma_0$, we have
	\begin{align}
		\sup_{0<\delta\leq\delta_0}E^\delta(t)\leq C,
		\qquad
		\lim_{t\rightarrow0}E^\delta(t)\leq C\big( \delta^{1/2}+\gamma^{-1/2} \big).
	\end{align}
\end{proposition}
 The proof of Proposition \ref{prop: uniform estimates for E(t)} is structured into the following three parts.
\begin{proposition}\label{prop: uniform estimates for Evp(t)}
There exist positive constants $C, T_0>0$ independent of $\delta$,  such that for $0<t<T_0$, it holds that
		\begin{align}\label{est: Evp(t)}
			E_{vp}^\delta (t)\leq CE^\delta(t)^2
		+C(t+\delta)^{1/4}.
		\end{align}
\end{proposition}
	
\begin{proposition}\label{prop: uniform estimates for Em(t)}
	There exist positive constants $C, T_0>0$ independent of $\delta$,  such that for $0<t<T_0$, it holds that
	\begin{align*}
			E_m^\delta(t) \leq C t^{1/4}\big( E^\delta(t)+1\big)^8.
		\end{align*}
\end{proposition}

\begin{proposition}\label{prop: uniform estimates for Eb(t)}
			There exist positive constants $C, T_0>0$ independent of $\delta$,  such that for $0<t<T_0$, it holds that
		\begin{align}\label{est: Eb(t)}
			E_b^\delta(t)\leq C(\gamma^{-1/2}+t^{1/2})    \big( E^\delta(t)+1\big)^2.
		\end{align}
\end{proposition}

We postpone the proofs of Proposition \ref{prop: uniform estimates for Evp(t)} $\sim$ Proposition \ref{prop: uniform estimates for Eb(t)} to later sections.

\medskip

\noindent\underline{\bf Proof of Proposition \ref{prop: uniform estimates for E(t)}.}  Armed with Proposition \ref{prop: uniform estimates for Evp(t)} $\sim$ Proposition \ref{prop: uniform estimates for Eb(t)}, we turn to prove Proposition \ref{prop: uniform estimates for E(t)}. Collecting the estimates above together, we have
\begin{align*}
	E^\delta(t)
	&\leq CE^\delta(t)^2
		+C(t+\delta)^{1/4}
		+Ct^{1/4}\big( E^\delta(t)+1\big)^8
		+C(\gamma^{-1/2}+t^{1/2})    \big( E^\delta(t)+1\big)^2,
\end{align*}
where $C$ is independent of $\delta$.  Choosing $t, \delta$ small enough and $\gamma$ large enough, we deduce that for $T_0$ small enough, Proposition \ref{prop: uniform estimates for E(t)} holds by a continuous argument.
\ef

Based on Proposition \ref{prop: uniform estimates for E(t)}, we have following two corollaries, which derive uniform estimates for the vorticity and are utilized to establish the convergence results in Section 4. 
\begin{corollary}\label{cor:E}
	There exist positive constants $T_0, C$ independent of $\delta$,  such that for $0\leq t\leq T_0$, $k=1,2$, it holds that
	\begin{align*}
		&\sum_{i+j\leq2}\left\|\langle\xi\rangle^{1+i}(y\pa_y)^j(\chi_b\omega^\delta-\chi_b\omega^\delta_c)_\xi(t)\right\|_{L^k_\xi L^1_y}
		+\left\|\langle\xi\rangle^{1+i}(y\pa_y)^j \pa_\xi(\chi_b\omega^\delta-\chi_b\omega^\delta_c)_\xi(t)\right\|_{L^k_\xi L^1_y}
		\leq C,\\
		&\qquad\qquad\qquad\sum_{i+j\leq2}\left\|\langle\xi\rangle^{1+i}(y\pa_y)^j \pa_y\big((1,x)(\chi_b\omega^\delta-\chi_b\omega^\delta_c)\big)_\xi(t)\right\|_{L^k_\xi L^1_y}
		\leq Ct^{-1/2} ,
	\end{align*}
	and
		\begin{align*}
		&\left\|\langle x\rangle^{1/4}\langle y\rangle\omega^\delta(t)\right\|_{L^1(y\geq1/2)}
		+t^{1/2}\left\|\langle x\rangle^{1/4}\langle y\rangle\nabla\omega^\delta(t)\right\|_{L^1(y\geq1/2)}\\
		&\qquad+t^{1/4}\left\|\langle x\rangle^{1/4}\langle y\rangle\omega^\delta(t)\right\|_{L^{4/3}(y\geq1/2)}
		+t^{3/4}\left\|\langle x\rangle^{1/4}\langle y\rangle\nabla\omega^\delta(t)\right\|_{L^{4/3}(y\geq1/2)}
		\leq C.
	\end{align*}
\end{corollary}

\begin{corollary}\label{prop: Sobolev estimate}
	There exist positive constants $T_0, C$ independent of $\delta$,   such that 
	\begin{align*}
		&e^{\frac{5\eps_0}{t}}\sup_{[0,t]}\|(1,x)\omega^\delta(s)\|_{H^4(\frac{7}{8}\leq y\leq 4)}
		\leq C  t^{\frac{1}{2}} , \qquad \|U^\delta(t)\|_{L^\infty( y\geq 2)}
		\leq C  t^{-\f12}, 
	\end{align*}	
	when $t\in (0, T_0]$.
\end{corollary}

Since the proof of the above two corollaries rely on the velocity estimates, we postpone them to the end of Section 3.

\section{The estimates of the velocity via Biot-Savart law}
 In this section, we derive several estimates of the velocity by means of the Biot-Savart law. 
To begin with, we present the Biot-Savart law in terms of the Fourier transform, as established in \cite{Maekawa}.

\begin{lemma}\label{lem: velocity formula}
Let $U=BS_{\mathbb R_+^2}[\omega]$. It holds that
	\begin{align*}
		&u_\xi(y)=\frac{1}{2}\Big(-\int_0^y e^{-|\xi|(y-z)}\big(1-e^{-2|\xi|z}\big)\omega_\xi(z)dz
		+\int_y^{+\infty}e^{-|\xi|(z-y)}\big(1+e^{-2|\xi|y}\big)\omega_\xi(z)dz\Big),\\
		&v_\xi(y)=-\frac{i\xi}{2|\xi|}\Big(\int_0^y e^{-|\xi|(y-z)}\big(1-e^{-2|\xi|z}\big)\omega_\xi(z)dz
		+\int_y^{+\infty}e^{-|\xi|(z-y)}\big(1-e^{-2|\xi|y}\big)\omega_\xi(z)dz\Big).
	\end{align*}
\end{lemma}

 Subsequently, we derive some estimates for $U^\delta$ near the boundary.
	\begin{lemma}\label{velocity estimates 1}
For $u^\delta$, we have
	\begin{align*}
		\Big\|\sup_{0<y<1+\mu}e^{\eps_0(1+\mu-y)_+|\xi|}|(\pa_x^i u^\delta)_\xi(s)|\Big\|_{L^1_\xi}
		\leq C_0\big( E^\delta(s)+1\big),\quad i=0,1, 2.
	\end{align*}
For $v^\delta$, we have
	\begin{align*}
		\Big\|\sup_{0<y<1+\mu}e^{\eps_0(1+\mu-y)_+|\xi|}\big|\frac{(\pa_x^i v^\delta)_\xi(s)}{y}\big|\Big\|_{L^1_\xi}
		\leq C_0\big( E^\delta(s)+1\big),\quad i=0,1.
	\end{align*}
	and
	\begin{align*}
		\Big\|\sup_{0<y<1+\mu}e^{\eps_0(1+\mu-y)_+|\xi|}\big|\frac{(\pa_x^2 v^\delta)_\xi(s)}{y}\big|\Big\|_{L^1_\xi}
		\leq C_0\Big((\mu_0-\mu-\gamma s)^{-\beta} E^\delta(s)+1\Big).
	\end{align*}
Moreover, we have 	
\begin{align*}
		\Big\|\sup_{0<y<1+\mu}e^{\eps_0(1+\mu-y)_+|\xi|}\big|\pa_x^i(y\pa_y) \big(u^\delta_\xi(s),\frac{v^\delta_\xi(s)}{y} \big)\big|\Big\|_{L^1_\xi}
		\leq C_0\big( E^\delta(s)+1\big),\quad i=0, 1. 
	\end{align*}
%
%
%
%
%
	 

\end{lemma}

\begin{proof}  

\underline{Estimates of $u^\delta$.} First, by means of Lemma \ref{lem: velocity formula}, we notice that
	\begin{align*}
		u^\delta_\xi(s,y)=&-\frac{1}{2}\int_0^y e^{-|\xi|(y-z)}\big(1-e^{-2|\xi|z}\big)\omega^\delta_\xi(s,z)dz\\
		&+\frac{1}{2}\big(\int_y^{1+\mu}+\int_{1+\mu}^{+\infty}\big)e^{-|\xi|(z-y)}\big(1+e^{-2|\xi|y}\big)\omega^\delta_\xi(s,z)dz:=I_1+I_2+I_3.
	\end{align*}
	Based on the following relation
	\begin{align}\label{weight transform 3}
		e^{\eps_0(1+\mu-y)_+|\xi|}e^{-|\xi||y-z|}
		\leq e^{\eps_0(1+\mu-z)_+|\xi|},
	\end{align}
we have
	\begin{align*}
		e^{\eps_0(1+\mu-y)_+|\xi|}\big(|I_1|+|I_2|\big)
		\leq C_0\int_0^{1+\mu}e^{\eps_0(1+\mu-z)_+|\xi|}|\omega^\delta_\xi(s,z)|dz,
	\end{align*}
	which implies
	\begin{align*}
		&\big\|\sup_{0<y<1+\mu}e^{\eps_0(1+\mu-y)_+|\xi|}\big(|I_1|+|I_2|\big)\big\|_{L^1_\xi}\\
		&\leq C_0\|\omega^\delta(s)\|_{Y^1_{\mu,s}}\\
		&\leq C_0\|\omega^\delta(s)-\omega^\delta_c(s)\|_{Y^1_{\mu,s}}
		+C_0\|\omega^\delta_c(s)\|_{Y^1_{\mu,s}}
		\leq C_0\big(E^\delta_b(s)+1\big).
	\end{align*}
	
	A direct computation yields
	\begin{align*}
		e^{\eps_0(1+\mu-y)_+|\xi|}|I_3|
		\leq C_0\int_{1+\mu}^2 |\omega^\delta_\xi(s,z)|dz
		+C_0\int_2^{+\infty}e^{-|\xi|/2}|\omega^\delta_\xi(s,z)|dz,
	\end{align*}
	which implies
	\begin{align*}
		&\big\|\sup_{0<y<1+\mu}e^{\eps_0(1+\mu-y)_+|\xi|}|I_3|\big\|_{L^1_\xi}
		\leq C_0\big\|\int_{1+\mu}^2 |\omega^\delta_\xi(s,z)|dz\big\|_{L^1_\xi}
		+C_0\big\|\int_2^{+\infty}e^{-|\xi|/2}|\omega^\delta_\xi(s,z)|dz\big\|_{L^1_\xi}\\
		&\leq C_0\int_1^2\big\|(1+|\xi|^2)^{-1/2}\big\|_{L^2_\xi}\big\|(1+|\xi|^2)^{1/2}\omega^\delta_\xi(s,z)\big\|_{L^2_\xi}dz
		+C_0\|\omega^\delta(s)\|_{L^1(y\geq2)}\\
		&\leq C_0\|\omega^\delta\|_{H^1(1\leq y\leq2)}
		+C_0\big\|e^\Psi\chi_m\psi\omega^\delta\big\|_{L^2}
		+C_0\|G+\mathcal W_R\|_{L^1}\\
		&\leq C_0\big( E^\delta(s)+1\big).
	\end{align*}
	
	Collecting these estimates, we obtain the desired result. The case $i=1,2$ follow by similar arguments, with $\omega^\delta$ replaced by $\pa_x\omega^\delta$, and we omit the details.
	
	\medskip
	
	\underline{Estimates of $v^\delta$.}
Next, we consider the $v^\delta$. 	By Lemma \ref{lem: velocity formula}, we obtain
	\beno
	v^\delta_\xi(y)=-\frac{i\xi}{2|\xi|}\big(\int_0^y e^{-|\xi|(y-z)}\big(1-e^{-2|\xi|z}\big)\omega^\delta_\xi(z)dz
		+\int_y^{+\infty}e^{-|\xi|(z-y)}\big(1-e^{-2|\xi|y}\big)\omega^\delta_\xi(z)dz\big),
		\eeno
which gives
	\begin{align*}
		\big|\frac{v^\delta_\xi(s,y)}{y}\big|\leq& \frac{1}{2y}\int_0^y e^{-|\xi|(y-z)}\big(1-e^{-2|\xi|z}\big)|\omega^\delta_\xi(s,z)|dz\\
		&+\frac{1}{2y}\big(\int_y^{1+\mu}+\int_{1+\mu}^{+\infty}\big)e^{-|\xi|(z-y)}\big(1-e^{-2|\xi|y}\big)|\omega^\delta_\xi(s,z)|dz\\
		&:=J_1+J_2+J_3.
	\end{align*}
	
	We note that
	\begin{align*}
		\left|1-e^{-2|\xi|z}\right|\leq 2|\xi|z\leq2|\xi|y,
		\quad
		\left|1-e^{-2|\xi|y}\right|\leq 2|\xi|y,
		\quad
		\text{for}
		\quad
		z\leq y,
	\end{align*}
	which together with \eqref{weight transform 3} imply
	\begin{align*}
		e^{\eps_0(1+\mu-y)_+|\xi|}\big(|J_1|+|J_2|\big)
		\leq C_0\int_0^{1+\mu}e^{\eps_0(1+\mu-z)_+|\xi|}|\xi||\omega^\delta_\xi(s,z)|dz,
	\end{align*}
	which leads to
	\begin{align*}
		\big\|\sup_{0<y<1+\mu}e^{\eps_0(1+\mu-y)_+|\xi|}\big(|J_1|+|J_2|\big)\big\|_{L^1_\xi}
		\leq C_0\|\pa_x\omega^\delta(s)\|_{Y^1_{\mu,s}}
		\leq C_0\big(E^\delta_b(s)+1\big).
	\end{align*}
	
	The term $J_3$ is treated in the same way as $I_3$ in the proof of (1), so that
	\begin{align*}
		\big\|\sup_{0<y<1+\mu}e^{\eps_0(1+\mu-y)_+|\xi|}|J_3|\big\|_{L^1_\xi}
		\leq C_0\big( E^\delta(s)+1\big).
	\end{align*}
	
	For $i=1, 2$, the proof is similar by replacing $\omega^\delta$ with $\pa^i_x\omega^\delta$.

\medskip

	\underline{Estimates involving conormal derivatives.}
We now turn to deriving estimates involving conormal derivatives. A direct computation yields
	\begin{align*}
		y\pa_y u^\delta_\xi(s,y)
		=&\frac{y}{2}\Big(\int_0^y e^{-|\xi|(y-z)}\big(1-e^{-2|\xi|z}\big)|\xi|\omega^\delta_\xi(s,z)dz\\
		&+\int_y^{+\infty}e^{-|\xi|(z-y)}\big(1+e^{-2|\xi|y}\big)|\xi|\omega^\delta_\xi(s,z)dz\\
		&-2\int_y^{+\infty}e^{-|\xi|(z-y)}e^{-2|\xi|y}|\xi|\omega^\delta_\xi(s,z)dz\Big)
		-y\omega^\delta_\xi(s,y).
	\end{align*}
	The first three terms are treated as (1) and (2). For the last term, the fundamental theorem of calculus gives rise to
	\begin{align*}
		\left\|\sup_{0<y<1+\mu}e^{\eps_0(1+\mu-y)_+|\xi|}|y\omega^\delta_\xi(s,y)|\right\|_{L^1_\xi}
		\leq C\|\omega^\delta(s)\|_{Y^1_{\mu,s}}
		+C\|y\pa_y\omega^\delta(s)\|_{Y^1_{\mu,s}}
		+C\|\pa_x\omega^\delta(s)\|_{Y^1_{\mu,s}}.
	\end{align*}
Thus, we obtain the inequality for $y\pa_y u^\delta_\xi$. The case $y\pa_y\big(\frac{v^\delta_\xi(s)}{y}\big)$ is derived from the relation
	\begin{align*}
		y\pa_y\big(\frac{v^\delta_\xi(s)}{y}\big)
		=\pa_y v^\delta_\xi(s)-\frac{v^\delta_\xi(s)}{y}
		=-(\pa_x u^\delta)_\xi(s)-\frac{v^\delta_\xi(s)}{y}.
	\end{align*}
	
	This completes the proof of the lemma.
	\end{proof}

The following lemma provides several estimates for $\chi_b \omega^\delta$.

	\begin{lemma}\label{velocity estimates 2}
	It holds that
\begin{align*}
		\big\|BS_{\mathbb R^2_+}[\zeta_1\omega^\delta(s)]\big\|_{L^\infty}
		\leq C_0\big\|\omega^\delta(s) \big\|_{Y_1(s)},
	\end{align*}
	and
\begin{align*}
		\big\|  BS_{\mathbb R^2_+}[\chi_b\omega^\delta](s)\big\|_{H^1(0,1+\mu)}
		&\leq C_0\big( E^\delta(s)+1\big),
	\end{align*}		
where the cut-off function $\zeta_1$ is defined by
	\begin{align}\label{def zeta1}
	\zeta_1=
	\left\{
	\begin{aligned}
		&1,\quad y\leq3/8,\\
		&0,\quad y\geq1/2.
	\end{aligned}
	\right.
\end{align}
\end{lemma}

\begin{proof}
Based on the definition $BS_{\mathbb R^2_+}[f]$ and Lemma \ref{lem: velocity formula}, we have
	\begin{align*}
		(BS_{\mathbb R^2_+}[\zeta_1\omega^\delta(s)])_\xi (s,y)=&-\frac{1}{2}\int_0^y e^{-|\xi|(y-z)}\big(1-e^{-2|\xi|z}\big)(\zeta_1\omega^\delta)_\xi dz\\
		&+\frac{1}{2} \int_y^{\infty} e^{-|\xi|(z-y)}\big(1+e^{-2|\xi|y}\big)(\zeta_1\omega^\delta)_\xi(s,z)dz
	\end{align*}
which implies the first desired result.

For the $H^1$ norm, we use Lemma \ref{lem: velocity formula} to obtain
	\begin{align*}
		\left\|BS_{\mathbb R^2_+}[\chi_b\omega^\delta(s)]\right\|_{L^2(0,1+\mu)}&=\left\|\left\|\big(BS_{\mathbb R^2_+}[\chi_b\omega^\delta(s)]\big)_\xi\right\|_{L^2_y(0,1+\mu)}\right\|_{L^2_\xi}\\
		&\leq C_0\left\|\left\|\int_0^y e^{-|\xi|(y-z)}\big(1-e^{-2|\xi|z}\big)(\chi_b\omega^\delta)_\xi dz\right\|_{L^2_y(0,1+\mu)}\right\|_{L^2_\xi}\\
		&\quad+\left\|\left\|\int_y^\infty e^{-|\xi|(y-z)}\big(1+e^{-2|\xi|z}\big)(\chi_b\omega^\delta)_\xi dz\right\|_{L^2_y(0,1+\mu)}\right\|_{L^2_\xi}\\
		&\leq C_0\|\omega^\delta(s)\|_{Y^1_{\mu,s}}
		+C_0\|\omega^\delta(s)\|_{L^2(1+\mu\leq y\leq3)}\\
		&\leq C_0\|\omega^\delta(s)\|_{Y_1(s)}
		+C_0\|e^\Psi\psi\chi_m\omega^\delta(s)\|_{L^2}\\
		&\leq C_0\big( E^\delta(s)+1\big).
	\end{align*}
By the same argument as above, we obtain
	\begin{align*}
		&\left\|\nabla BS_{\mathbb R^2_+}[\chi_b\omega^\delta(s)]\right\|_{L^2(0,1+\mu)}  
		\leq C_0\big( E^\delta(s)+1\big).
	\end{align*}

This completes the proof of the lemma.
\end{proof}

\medskip

Next, we derive some estimates for $U^\delta$ away the boundary.
	\begin{lemma}\label{velocity estimates 3}
	It holds that
  \begin{align*}
		\|U^\delta(s)\|_{L^\infty(\operatorname{supp}\chi_m)}+s^{\f12}\|U^\delta(s)\|_{L^\infty(y\geq1)}
		\leq C_0\big(E^\delta(s)+1\big),
	\end{align*}
and	
	\begin{align*}
		\left\|\pa_x^i\pa_y^j U^\delta(s)\right\|_{L^\infty(a\leq y\leq b)}
		\leq C_0\big( E^\delta(s)+1\big)
		+C_0\left\|\omega^\delta(s)\right\|_{H^{i+1}(a-\frac{1}{100}\leq y\leq{b+\frac{1}{100}})},
	\end{align*}
	where  $i,j\geq0$ and $\frac{1}{2}\leq a<b\leq3$.

\end{lemma}

\begin{proof} 
First, we have
	\begin{align*}
		\|U^\delta(s)\|_{L^\infty(\operatorname{supp}\chi_m)}
		\leq& \|BS_{\mathbb R^2_+}[\zeta_1\omega^\delta(s)]\|_{L^\infty}
		+\|BS_{\mathbb R^2_+}[\chi_{vp}\omega^\delta(s)]\|_{L^\infty(\operatorname{supp}\chi_m)}\\
		&+\|BS_{\mathbb R^2_+}[(1-\chi_{vp}-\zeta_1)\omega^\delta(s)]\|_{L^\infty}\\
		:=&T_1+T_2+T_3.
	\end{align*}
	
For $T_1$, it has been estimated by Lemma \ref{velocity estimates 2}. For $T_2$, we define $\zeta_2$ as
	\begin{align}\label{def zeta2}
		\zeta_2(x,y)=
		\left\{
		\begin{aligned}
			&1,\quad |(x,y)-(0,20)|\leq 2,\\
			&0,\quad |(x,y)-(0,20)|\geq 3.
		\end{aligned}
		\right.
	\end{align}
We utilize Lemma \ref{lem: velocity weighted est} to obtain
	\begin{align*}
		T_2\leq& \left\|BS_{\mathbb R^2_+}[\chi_{vp}\cdot\frac{\alpha}{s+\delta}G(\eta,\tau')]\right\|_{L^\infty(\operatorname{supp}\chi_m)}
		+\left\|BS_{\mathbb R^2_+}[\chi_{vp}\zeta_2 \big(\omega^\delta-\frac{\alpha}{s+\delta}G(\eta,\tau') \big)]\right\|_{L^\infty(\operatorname{supp}\chi_m)}\\
		&+\left\|BS_{\mathbb R^2_+}[\chi_{vp}(1-\zeta_2) \big(\omega^\delta-\frac{\alpha}{s+\delta}G(\eta,\tau') \big)]\right\|_{L^\infty}\\
		\leq& C_0
		+C_0\left\|\omega^\delta-\frac{\alpha}{s+\delta}G(\eta,\tau')\right\|_{L^1}\\
		&+C_0\left\|\chi_{vp}(1-\zeta_2) \big(\omega^\delta-\frac{\alpha}{s+\delta}G(\eta,\tau') \big)\right\|_{L^{4/3}}^{1/2}
		\left\|\chi_{vp}(1-\zeta_2) \big(\omega^\delta-\frac{\alpha}{s+\delta}G(\eta,\tau') \big)\right\|_{L^4}^{1/2}\\
		\leq& C_0+C_0\|\mathcal W_R\|_{L^1}
		+\frac{C_0\alpha}{(s+\delta)^{1/2}}\|\mathcal W_R\|_{L^{4/3}(\frac{3}{\sqrt{s+\delta}}\leq |\eta|\leq \frac{6}{\sqrt{s+\delta}})}^{1/2} \|\mathcal W_R\|_{L^4(\frac{3}{\sqrt{s+\delta}}\leq |\eta|\leq \frac{6}{\sqrt{s+\delta}})}^{1/2}\\
		\leq& C_0\big(E^\delta_{vp}(s)+1\big)
		+C_0 (s+\delta)^{\frac{m}{2}-\frac{5}{8}}\|\mathcal W_R\|_{L^2(m)}^{\f34}\|\nabla\mathcal W_R\|_{L^2(m)}^{\f14}\\
		\leq& C_0\big(E^\delta_{vp}(s)+1\big),
	\end{align*}
where we use the self-similar transformation \eqref{def: self-similar} and $\mathcal W^\delta=\chi_{vp}G+\mathcal W_R$.

For $T_3$,  we take advantage of Lemma \ref{lem: velocity weighted est} to deduce
	\begin{align*}
		T_3 &\leq C_0\|(1-\chi_{vp}-\zeta_1)\omega^\delta(s)\|_{L^{4/3}}^{1/2}\|(1-\chi_{vp}-\zeta_1)\omega^\delta(s)\|_{L^4}^{1/2}\\
		&\leq C_0\left\|e^{\Psi}\psi\chi_m\omega^\delta(s)\right\|_{L^2}^{1/2} \left\|e^{\Psi}\psi\chi_m\omega^\delta(s)\right\|_{L^4}^{1/2}\\
		&\leq C_0 E^\delta_m(s).
	\end{align*}
	
	Collecting estimates of $T_1\sim T_3$, we obtain
	\begin{align*}
		\|U^\delta(s)\|_{L^\infty(\operatorname{supp}\chi_m)}
		\leq C_0\big(E^\delta(s)+1\big).
	\end{align*}

Using Lemma \ref{velocity estimates 2} and Lemma \ref{lem: velocity weighted est}, we deduce that
	\begin{align*}
		&\quad\|U^\delta(s)\|_{L^\infty(y\geq1)}\\
		&\leq \left\|BS_{\mathbb R^2_+}[\zeta_1\omega^\delta]\right\|_{L^\infty}
		+\left\|BS_{\mathbb R^2_+}[(1-\zeta_1-\chi_{vp})\omega^\delta]\right\|_{L^\infty}
		+\left\|BS_{\mathbb R^2_+}[\chi_{vp}\omega^\delta]\right\|_{L^\infty}\\
		&\leq C_0\|\omega^\delta(s)\|_{Y_1(s)}
		+C_0\|(1-\zeta_1-\chi_{vp})\omega^\delta(s)\|_{L^{4/3}\cap L^4}
		+C_0\|\chi_{vp}\omega^\delta(s)\|_{L^{4/3}}^{1/2}\| \chi_{vp}\omega^\delta(s)\|_{L^4}^{1/2}\\
		\nonumber
		&\leq C_0\|\omega^\delta(s)\|_{Y_1(s)}
		+C_0\|e^\Psi \psi\chi_m\omega^\delta(s)\|_{L^2\cap L^4}
		+C_0 (s+\delta)^{-1/2}\|G+\mathcal W_R\|_{L^{\f43}}^{\f12} \|\nabla_\eta(G+\mathcal W_R)\|^{\f12}_{L^{\f43}}
		\\
		\nonumber
		&\leq C_0\big( E^\delta(s)+1\big)
		+C_0 (s+\delta)^{-1/2}\big(1+\|\mathcal W_R\|_{L^2(m)}^{\f12}\|\nabla_\eta\mathcal W_R\|_{L^2(m)}^{\f12}\big)\\
		&\leq C_0 s^{-\f12}\big( E^\delta(s)+1\big).
	\end{align*}

All that remains is to derive estimates for $U^\delta$ between the point vortex and the boundary. Here, we only consider the case  $\pa_x^i u^\delta$, since the other cases can be treated in a similar way. First, we observe that
	\begin{align*}
		 |(\pa_x^i u^\delta)_\xi(s,y)|
		\leq &\int_0^{a-\frac{1}{100}}e^{-\frac{|\xi|}{100}}|\xi|^i|\omega^\delta_\xi(s,z)|dz
		+\int_{a-\frac{1}{100}}^{b+\frac{1}{100}} |\xi|^i|\omega^\delta_\xi(s,z)|dz
		\\
		&\quad+\int_{b+\frac{1}{100}}^{+\infty}e^{-|\xi|}|\xi|^i|\omega^\delta_\xi(s,z)|dz\\
		&\leq C_0\int_0^{a-\frac{1}{100}}|\omega^\delta_\xi(s,z)|dz
		+C_0\|(\pa_x^i\omega^\delta)_\xi(s,z)\|_{L^2_z(a-\frac{1}{100}\leq z\leq{b+\frac{1}{100}})}\\
		&\quad+C_0 \int_{b+\frac{1}{100}}^{+\infty}e^{-|\xi|/2}|\omega^\delta_\xi(s,z)|dz,
	\end{align*}
	which yields
	\begin{align*}
		 &\quad\|\pa_x^i u^\delta(s)\|_{L^\infty(a\leq y\leq b)} \\ 
		&\leq C_0\|\omega^\delta(s)\|_{Y_1(s)}+C_0\|\omega^\delta(s,z)\|_{L^1(z\geq3)}
		+C_0 \left\|\langle\xi\rangle^{-1}\right\|_{L^2_\xi}\left\|\left\|\langle\xi\rangle(\pa_x^i\omega^\delta)_\xi(s,z)\right\|_{L^2_z(a-\frac{1}{100}\leq z\leq{b+\frac{1}{100}})}\right\|_{L^2_\xi}
		\\
		&\leq C_0\big(E^\delta_b(s)+1\big)
		+C_0\|\omega^\delta(s)\|_{H^{i+1}(a-\frac{1}{100}\leq z\leq{b+\frac{1}{100}})}
		+C_0\|e^\Psi \psi\chi_m\omega^\delta(s)\|_{L^2}
		+C_0\|G+\mathcal W_R\|_{L^1}\\
		&\leq C_0\big( E^\delta(s)+1\big)
		+C_0\|\omega^\delta(s)\|_{H^{i+1}(a-\frac{1}{100}\leq z\leq{b+\frac{1}{100}})}.
	\end{align*}
	
	By now, we obtain all desired results.
\end{proof}

\begin{remark}
From the above two lemmas, we obtain that the velocity given by $BS_{\mathbb R^2_+}[\zeta_1\omega^\delta(s)]$ admits a uniform upper bound. This indicates that the boundary exerts only a weak influence on the point vortex. This is the key point that allows us to remove the smallness condition on the initial data.
 \end{remark}

To close this section, we now present the proof of the two corollaries stated in Section 2.4.\smallskip

\noindent\textbf{Proof of Corollary \ref{cor:E}}. The first inequality is obtained by the definition of $E_{b}^\delta$ and the uniform estimates in Proposition \ref{prop: uniform estimates for E(t)}. 

 For the second inequality, recalling \eqref{integral eq of omega-omega c-1}, we have
	\begin{align}\label{est: cotangential derivative est of omega}
		&\left\|\langle\xi\rangle^{1+i}(y\pa_y)^j \pa_y(\chi_b\omega^\delta-\chi_b\omega^\delta_c)_\xi\right\|_{L^k_\xi L^1_y} \\
		\nonumber
		&\quad\leq \big\|\int_0^{+\infty}\left\|(y\pa_y)^j\pa_y\big(H_\xi(t,y,z)+R_\xi(t,y,z)\big)\big\|_{L^1_y}
		\big|\langle\xi\rangle^{1+i}b_\xi^\delta(z)\right|dz\big\|_{L^k_\xi}\\
		\nonumber
		&\qquad+\big\|\int_0^t\int_0^{+\infty}\big\|(y\pa_y)^j\pa_y\big(H_\xi(t-s,y,z)+R_\xi(t-s,y,z)\big)\big\|_{L^1_y}
		\big|\langle\xi\rangle^{1+i}N_\xi^\delta(s,z)\big|dzds \big\|_{L^k_\xi}\\
		\nonumber
		&\qquad+\big\|\int_0^t\big\|(y\pa_y)^j\pa_y\big(H_\xi(t-s,y,0)+R_\xi(t-s,y,0)\big)\big\|_{L^1_y}
		\left|\langle\xi\rangle^{1+i}B_\xi^\delta(s)\right|ds \big\|_{L^k_\xi}\\
		\nonumber
		&\quad\leq C\big\|\int_0^{+\infty}(t^{-1/2}+\langle\xi\rangle)\big|\langle\xi\rangle^{1+i}b_\xi^\delta(z)\big|dz\big\|_{L^k_\xi}
		+\big\|\int_0^t \big((t-s)^{-1/2}+\langle\xi\rangle\big)\big|\langle\xi\rangle^{1+i}B_\xi^\delta(s)\big|ds \big\|_{L^k_\xi} \\
		\nonumber
		&\qquad+C\big\|\int_0^t\int_0^{+\infty}\big((t-s)^{-1/2}+\langle\xi\rangle\big)\big|\langle\xi\rangle^{1+i}N_\xi^\delta(s,z)\big|dzds \big\|_{L^k_\xi}\nonumber\\
		&\quad\leq Ct^{-1/2},\nonumber
	\end{align}
	where we use Lemma \ref{prop of R 1}, Lemma \ref{est of HR b}, Proposition \ref{prop: estimate of N in new norm} and Proposition \ref{prop: estimate of B}. 

The proof of $(x\chi_b\omega^\delta-x\chi_b\omega^\delta_c)_\xi$ can be established using the same argument, and thus we omit the details.
Now, we prove the last inequality and  focus on the first two terms, since the remaining terms are handled similarly.
By Proposition \ref{prop: uniform estimates for E(t)}, we have
	\begin{align*}
		&\|\langle x\rangle^{1/4}\langle y\rangle\omega^\delta(t)\|_{L^1(y\geq1/2)}\\
		&\leq \|\langle x\rangle^{1/4}\langle y\rangle(1-\chi_{vp})\omega^\delta(t)\|_{L^1(y\geq1/2)}
		+C\|\chi_{vp}\omega^\delta(t)\|_{L^1}\\
		&\leq C\|\psi\chi_m\omega^\delta(t)\|_{L^2}
		+C\|\mathcal W_R(t)+G\|_{L^1_\eta}\\
&		\leq C.
	\end{align*}
	
For the second term, we similarly have
	\begin{align}\label{ineq:111}
		\left\|\langle x\rangle^{1/4}\langle y\rangle\nabla\omega^\delta(t)\right\|_{L^1(y\geq1/2)}
		&\leq C\left\|\psi\nabla\big(\chi_m\omega^\delta(t)\big)\right\|_{L^2}
		+Ct^{-1/2}\|\nabla_\eta(\mathcal W_R+G)\|_{L^1_\eta}\\
		&\leq C\left\|\psi\nabla\big(\chi_m\omega^\delta(t)\big)\right\|_{L^2}
		+Ct^{-1/2}E^\delta(t).\nonumber
	\end{align}
Next, we estimate the first term on the right-hand side of the above inequality.  Multiplying $\chi_m$ and applying $\pa_x$ on \eqref{eq: NS vorticity}, we obtain
	\begin{align*}
		&\pa_t\pa_x(\chi_m\omega^\delta)
		+\pa_x\big(U\cdot\nabla(\chi_m\omega^\delta)\big)
		-\Delta\pa_x(\chi_m\omega^\delta)\\
		&\qquad=\pa_x\big(U^\delta\cdot\nabla\chi_m\omega^\delta\big)
		-2\pa_x(\nabla\chi_m\cdot\nabla\omega^\delta)
		-\pa_x(\Delta\chi_m\omega^\delta).
	\end{align*}
	
	Taking $L^2$ inner product with $\psi^2\pa_x(\chi_m\omega^\delta)$ yields
	\begin{align*}
		&\frac{1}{2}\frac{d}{dt}\|\psi\pa_x(\chi_m\omega^\delta)\|_{L^2}^2
		+\langle \pa_x\big(U^\delta\cdot\nabla(\chi_m\omega^\delta)\big), \psi^2\pa_x(\chi_m\omega^\delta)\rangle
		-\langle\Delta\pa_x(\chi_m\omega^\delta),\psi^2\pa_x(\chi_m\omega^\delta)\rangle\\
		&
		=\langle \pa_x\big(U^\delta\cdot\nabla\chi_m\omega^\delta\big)
		-2\pa_x(\nabla\chi_m\cdot\nabla\omega^\delta)-\pa_x(\Delta\chi_m\omega^\delta) ,\psi^2\pa_x(\chi_m\omega^\delta)\rangle	\end{align*}
By integration by parts, the dissipation term becomes
	\begin{align*}
		&-\langle\Delta\pa_x(\chi_m\omega^\delta),\psi^2\pa_x(\chi_m\omega^\delta)\rangle\\
		&=\|\psi\pa_x\nabla(\chi_m\omega^\delta)\|^2_{L^2}
		+\langle\pa_x\nabla(\chi_m\omega^\delta),\nabla(\psi^2)\pa_x(\chi_m\omega^\delta)\rangle\\
		&=\|\psi\pa_x\nabla(\chi_m\omega^\delta)\|^2_{L^2}
		-\frac{1}{2}\langle\Delta(\psi^2)\pa_x(\chi_m\omega^\delta),\pa_x(\chi_m\omega^\delta)\rangle\\
		&\geq \|\psi\pa_x\nabla(\chi_m\omega^\delta)\|^2_{L^2}
		-C\|\psi\pa_x(\chi_m\omega^\delta)\|_{L^2}^2.
	\end{align*}
For the nonlinear terms,	using integration by parts, Lemma \ref{velocity estimates 3} and Proposition \ref{prop: uniform estimates for E(t)}, we deduce
	\begin{align*}
		&|\langle \pa_x\big(U^\delta\cdot\nabla(\chi_m\omega^\delta)\big), \psi^2\pa_x(\chi_m\omega^\delta)\rangle|\\
		&\leq |\langle U^\delta\cdot\nabla(\chi_m\omega^\delta),\psi^2\pa_x^2(\chi_m\omega^\delta)\rangle|
		+|\langle U^\delta\cdot\nabla(\chi_m\omega^\delta),
		\pa_x(\psi^2)\pa_x(\chi_m\omega^\delta)\rangle|\\
		&\leq C\|U^\delta\|_{L^\infty(\operatorname{supp}\chi_m)}
		\big(\|\psi\nabla(\chi_m\omega^\delta)\|_{L^2}
		\|\psi\pa_x^2(\chi_m\omega^\delta)\|_{L^2}
		+\|\psi\nabla(\chi_m\omega^\delta)\|_{L^2}
		\|\psi\pa_x(\chi_m\omega^\delta)\|_{L^2}\big)\\
		&\leq C\big(1+ E^\delta(t)\big)^2
		\|\psi\nabla(\chi_m\omega^\delta)\|_{L^2}^2
		+\frac{1}{10}\|\psi\pa_x^2(\chi_m\omega^\delta)\|^2_{L^2}\\
		&\leq C\|\psi\nabla(\chi_m\omega^\delta)\|_{L^2}^2
		+\frac{1}{10}\|\psi\pa_x^2(\chi_m\omega^\delta)\|^2_{L^2},
	\end{align*}
and
	\begin{align*}
		&|\langle \pa_x\big(U^\delta\cdot\nabla\chi_m\omega^\delta\big)
		-2\pa_x(\nabla\chi_m\cdot\nabla\omega^\delta)-\pa_x(\Delta\chi_m\omega^\delta) ,\psi^2\pa_x(\chi_m\omega^\delta)\rangle|\\
		&\leq C\big(\|U^\delta\|_{L^\infty(\operatorname{supp}\chi_m)}+1\big)\|\psi\pa_x^2(\chi_m\omega^\delta)\|_{L^2}
		\\
		&\quad\cdot\big(\|(1+|x|)(\omega^\delta,\nabla\omega^\delta)\|_{L^2(\frac{1}{4}\leq y\leq \frac{3}{8})}
		+\|(\omega^\delta,\nabla\omega^\delta)\|_{L^2(3\leq|(x,y)-(0,20)|\leq4)}\big)\\
		&\leq C\big( E^\delta(t)+1\big)\big(\|(1,x)\omega^\delta(t)\|_{Y_2(t)}
		+(t+\delta)^{\frac{m-1}{2}}\|\nabla(G+\mathcal W_R)\|_{L^2(m)}\big)
		\|\psi\pa_x^2(\chi_m\omega^\delta)\|_{L^2}\\
		&\leq C E^\delta(t)\big( E^\delta(t)+1\big)
		\|\psi\pa_x^2(\chi_m\omega^\delta)\|_{L^2}\\
		&\leq C+\frac{1}{10}\|\psi\pa_x^2(\chi_m\omega^\delta)\|_{L^2}^2.
	\end{align*}
Combining all above estimates, we obtain
\begin{align*}
		\sup_{0\leq t\leq T}\|\psi\pa_x\big(\chi_m\omega^\delta(t)\big)\|_{L^2}
		\leq C.
	\end{align*}		
In the same way, we obtain the estimate $\|\psi\pa_y(\chi_m\omega^\delta)\|_{L^2} \leq C$. 

Substituting the two estimates into \eqref{ineq:111}, we obtain the desired result.
\ef

\medskip

\noindent\textbf{Proof of Corollary \ref{prop: Sobolev estimate}}. The first inequality is proved similarly as Proposition 3.6 in \cite{HWYZ} and we have to use Lemma \ref{velocity estimates 3} in this paper. The second one is obtained by Lemma \ref{velocity estimates 3} and Proposition \ref{prop: uniform estimates for E(t)}.
\ef

\section{Proof of Theorem \ref{main theorem}: existence part}

This section is devoted to proving the existence part of Theorem \ref{main theorem}. More precisely, the main result is stated as follows, with its proof relying on finding a convergent subsequence of the sequence $(\omega^\delta,U^\delta).$

  \begin{proposition}\label{prop: existence part}
	With the initial vorticity $\omega_0$ given by  \eqref{initial data}, the Navier-Stokes equations \eqref{eq: NS vorticity} have a global solution $(U, \omega)\in L^{2,\infty}(\mathbb R^2_+)\times L^1(\mathbb R^2_+)$, where the vorticity $\omega$ satisfies \eqref{initial limit}. 
	\end{proposition}

 \subsection{Convergence of the solution sequence}
 
 Before treating $\omega^{\delta}$, we consider the convergence of the corrector $\omega^{\delta}_c$. First, we define the limit functions $\omega_c, b$ by  replacing $u_0^\delta$ by $u_0$ in \eqref{omega c} and \eqref{omega-omega c}. More precisely, we set $b=u_0\chi_b\chi_b'$,
\begin{align}\label{def: omega c without delta}
\omega_c(t,x,y)	=&-2u_0(x)\cdot\frac{1}{(4\pi t)^{1/2}}e^{-\frac{y^2}{4t}}\\
	&-u_0(x)\int_0^{+\infty}\frac{1}{(4\pi t)^{1/2}}\big(e^{-\frac{(y-z)^2}{4t}}+e^{-\frac{(y+z)^2}{4t}}\big)\chi'_b(z) dz,\nonumber
\end{align}
and 
\begin{align*}
	u_0(x)=\frac{\alpha}{\pi}\cdot\frac{20}{x^2+20^2}.
\end{align*}
By a direct computation, we have for $k=1,2$, $0< t\leq T$ and some contant $C'$
\begin{align}\label{convergence of omega c delta}
	\lim_{\delta\rightarrow0}
	\Big(\sum_{i+j\leq2}\left\|e^{C'|\xi|}\big(\pa_x^i(y\pa_y)^j(1,x)(\omega_c^\delta-\omega_c)\big)_\xi \right\|_{L^k_\xi L^1_y}
	+\left\|(1,x)(\omega_c^\delta-\omega_c)\right\|_{H^4(y\geq1)} \Big)=0,
\end{align}
and
\begin{align}\label{convergence of b delta}
	\lim_{\delta\rightarrow0}\left\|\|e^{C'|\xi|}(b_\xi^\delta-b_\xi)(y)\|_{L^1_y}\right\|_{L^k_\xi}=0.
\end{align}
 
 Based on the uniform estimates for $E^\delta$, we obtain the following convergence results for the vorticity.

 \begin{lemma}\label{lem: convergence in space-time norms}
	There exist a subsequence (still denoted by $\omega^\delta$) and limit functions $(\omega_1,\omega_2,\omega_3)\in L^{2}(0,T_0; L^1)$, such that for $0\leq t\leq T_0, k=1,2$, it holds that
	\begin{align*}	 &\lim_{\delta\rightarrow0} \sum_{i+j\leq1}\big\| \|\big(\pa_x^i(y\pa_y)^j(\chi_b\omega^\delta-\omega_1)\big)_\xi(s,y) \|_{L^k_\xi L^1_y}\big\|_{L^{3}_{(0,t)}} =0,\\
	&\lim_{\delta\rightarrow0}\big\|\|( \omega^\delta-\omega_2)(s)\|_{L^1\cap L^{\f43}(y\geq2)}\big\|_{L^{3}_{(0,t)}}=0.
		\end{align*}
Moreover, on the domain $1\leq y\leq 4$, we have
		\begin{align*}	 
	\lim_{\delta\rightarrow0} \left\|\|(\omega^\delta-\omega_3)(s)\|_{H^3(1\leq y\leq4)}\right\|_{L^{3}_{(0,t)}}=0.
	\end{align*}
\end{lemma}

\begin{proof}
To begin with, using Corollary \ref{cor:E}, we obtain
\beno
&&\Big\|\sum_{i+j\leq1}\left\|\langle\xi\rangle^{1+i}(y\pa_y)^j(\chi_b\omega^\delta-\chi_b\omega^\delta_c)_\xi(t)\right\|_{L^k_\xi L^1_y} \Big\|_{L^\infty(0, T)}\\
&&+\Big\|\sum_{i+j\leq1}\left\|(\pa_\xi,\pa_y)\Big( \langle\xi\rangle^{1+i}(y\pa_y)^j(\chi_b\omega^\delta-\chi_b\omega^\delta_c)_\xi(t) \Big) \right\|_{L^k_\xi L^1_y} \Big\|_{L^p(0, T)}\leq C, \qquad for \ 1\leq p<2.
\eeno
Next, we estimate  $\pa_t \omega^\delta$.  To achieve that,  multiplying $\chi_b$ on \eqref{eq: NS vorticity} gives
	\begin{align*}
		\pa_t(\chi_b\omega^\delta)
		=\Delta(\chi_b\omega^\delta)
		-U^\delta\cdot\nabla(\chi_b\omega^\delta)
		-\chi_b'u^\delta\omega^\delta
		-2\chi_b'\pa_y\omega^\delta
		-\chi_b''\omega^\delta,
	\end{align*}
	which together with Corollary \ref{cor:E} implies
	\begin{align*}
		&\left\|\langle\xi\rangle^{-4}\big(\pa_x^i(y\pa_y)^j\pa_t(\chi_b\omega^\delta)\big)_\xi(s,y)\right\|_{L^k_\xi W_y^{-2,1}}\\
		&\leq \left\|(y\pa_y)^j(\chi_b\omega^\delta)_\xi(s,y)\right\|_{L^k_\xi L^1_y}
		+\|\chi_b' u^\delta\omega^\delta\|_{L^2} +\|\chi_b''(1,x)\omega^\delta\|_{L^2}
		+2\|\chi_b'\pa_y\omega^\delta\|_{L^2}\\
		&\quad+\big\|(y\pa_y)^j\big(u^\delta\pa_x(\chi_b\omega^\delta)+\frac{v^\delta}{y}y\pa_y(\chi_b\omega^\delta)\big)_\xi(s,y)\big\|_{L^k_\xi L^1_y}
\\		
		&\leq C.
	\end{align*}
 
Then, by Aubin-Lions Lemma and \eqref{convergence of omega c delta}, we deduce that there exists  $\om_1$ such that
\beno
\lim_{\delta\rightarrow0}\sum_{i+j\leq1}\left\|\left\|\big(\pa_x^i(y\pa_y)^j(\chi_b\omega^\delta-\omega_1)\big)_\xi(s,y)\right\|_{L^k_\xi L^1_y}\right\|_{L^{3}_{(0,t)}}=0.
 \eeno

he remaining conclusions of the lemma follow by analogous arguments, and we omit the details.
\end{proof}

Now, we define the limit functions $\omega,U=(u,v)$ as follows
\begin{align}\label{def: limit functions omega U}
	\omega:=\omega_1+(1-\chi_b)\omega_2,\qquad U:=BS_{\mathbb R^2_+}[\omega], \quad for \ t\in(0,T).
\end{align}

\begin{remark}
We emphasize here that the limit function $\omega$ still satisfies the estimates in Corollary \ref{cor:E}, with the index $\delta$ removed.
\end{remark}

We now establish the convergence of the velocity, as stated below.
\begin{lemma}\label{lem: convergence in space-time norms of velocity}
	There exists a subsequence (still denoted by $U^\delta$), such that for $0\leq t\leq T_0$,
	\begin{align*}
		\lim_{\delta\rightarrow0}\left\|\left\|\left\|\big(u^\delta-u,\frac{v^\delta-v}{y}\big)_\xi(s)\right\|_{ L^\infty_y(y\leq3)}  \right\|_{L^1_\xi }\right\| _{L^{3}_{(0,t)}}  =0.
	\end{align*}
\end{lemma}
\begin{proof}
	
Here, the proofs of $u^\delta-u$ and $\frac{v^\delta-v}{y}$ follow by the same argument. We therefore only present the proof for $u^\delta-u$. First, by Lemma \ref{lem: velocity formula}, we have
	\begin{align*}
		|(u^\delta-u)_\xi(t,y)|
		\leq \int_0^4|(\omega^\delta-\omega)_\xi(t,z)|dz
		+\int_4^{+\infty}e^{-|\xi|}|(\omega^\delta-\omega)_\xi(t,z)|dz,
	\end{align*}
	which implies
	\begin{align*}
		 &\big\|\|(u^\delta-u)_\xi(t)\|_{L^\infty_y (y\leq3)}\big\|_{L^1_\xi }\\
		&\leq \|(\omega^\delta-\omega)_\xi(t)\|_{L^1_\xi L^1_y(y\leq2)}
		+\|\langle\xi\rangle^{-1}\|_{L^2_\xi}\|\langle\xi\rangle(\omega^\delta-\omega)_\xi\|_{L^2_\xi L^2_y(2\leq y\leq4)} \\
		&\qquad +\int_4^{+\infty}\|e^{-|\xi|}\|_{L^1_\xi}\|(\omega^\delta-\omega)_\xi(t,z)\|_{L^\infty_\xi}dz\\
	 &\leq \|(\omega^\delta-\omega)_\xi(t)\|_{L^1_\xi L^1_y(y\leq2)}
		+\|(\omega^\delta-\omega)(t)\|_{H^1(2\leq y\leq4)}
		+\|(\omega^\delta-\omega)(t)\|_{L^1(y\geq4)}.
	\end{align*}
	Thus, by Lemma \ref{lem: convergence in space-time norms}, we obtain the desired result.
\end{proof}

Next, we analyze the convergence of the nonlinear term $N^\delta$, which is supported in a neighborhood of the boundary.
\begin{lemma}\label{lem: convergence of product terms in space-time}
	For $0\leq t\leq T_0$, it holds that
	\begin{align*}
		&\lim_{\delta\rightarrow0}\int_0^t\left\|N_\xi^\delta(s,y)-N_\xi(s,y)\right\|_{L^k_\xi L^1_y}ds =0,\quad for \quad k=1,2,
	\end{align*}
	where $N_\xi=-\chi_b (U\cdot\nabla\omega)_\xi-\chi_b''\omega_\xi-2\chi_b'\pa_y\omega_\xi +\chi_b(\pa_x^2\omega_c)_\xi+\chi_b''(\omega_c)_\xi+2\chi_b'\pa_y(\omega_c)_\xi$.
\end{lemma}

\begin{proof}
By Lemma \ref{product estimate} and a direct computation, we obtain 
	\begin{align*}
		&\int_0^t\left\|N_\xi^\delta(s,y)-N_\xi(s,y)\right\|_{L^k_\xi L^1_y}ds
		\leq \int_0^t\left\|\|(u^\delta-u,\frac{v^\delta-v}{y})_\xi\|_{L^\infty_y(y\leq3)}\right\|_{L^1_\xi}\|((\pa_x,y\pa_y)\chi_b\omega^\delta)_\xi\|_{L^k_\xi L^1_y}ds\\
		&\quad+\int_0^t \left\|\left\|(u,\frac{v}{y})_\xi\right\|_{L^\infty_y(y\leq3)}\right\|_{L^1_\xi}\|((\pa_x,y\pa_y)(\chi_b\omega^\delta-\chi_b\omega))_\xi\|_{L^k_\xi L^1_y}ds
		+\int_0^t \|\omega^\delta-\omega\|_{H^2(2\leq y\leq3)}ds\\
		&\quad+\int_0^t \left\|\|(u^\delta-u,\frac{v^\delta-v}{y})_\xi\|_{L^\infty_y(y\leq3)}\right\|_{L^1_\xi}\|\omega^\delta\|_{H^1(2\leq y\leq3)}ds\\
		&\quad+\int_0^t \left\|\|(u_\xi,\frac{v_\xi}{y})\|_{L^\infty_y(y\leq3)}\right\|_{L^1_\xi}\|\omega^\delta-\omega\|_{H^1(2\leq y\leq3)}ds\\
		&\quad+\int_0^t \left\|\chi_b(\pa_x^2\omega_c^\delta-\pa_x^2\omega_c)_\xi\right\|_{L^k_\xi L^1_y}
		+\left\|\chi_b'\pa_y(\omega_c^\delta-\omega_c)_\xi\right\|_{L^k_\xi L^1_y}
		+\left\|\chi_b''(\omega_c^\delta-\omega_c)_\xi\right\|_{L^k_\xi L^1_y} ds.
	\end{align*}
Combining the above inequality with Lemma \ref{lem: convergence in space-time norms}, Lemma \ref{lem: convergence in space-time norms of velocity} and \eqref{convergence of omega c delta} and Corollary \ref{prop: Sobolev estimate}, we obtain the desired results.
\end{proof}

In order to handle the term $B^\delta$, we now establish the convergence of $U^\delta \om^\delta$, which is supported away from the boundary.

\begin{lemma} \label{lemma: u-om}
	For $0\leq t\leq T_0$, it holds that
	\begin{align*}		&\lim_{\delta\rightarrow0}\int_0^t  (t-s)^{-1/2}\left\|(U^\delta\omega^\delta-U\omega)(s)\right\|_{L^1(y\geq2)}ds=0.
	\end{align*}
 
\end{lemma}
\begin{proof}
Here, we employ the Dominated Convergence Theorem to prove this lemma. 

First, we prove that the sequence has an uniform upper bound. By using  Corollary \ref{cor:E}, we obtain that $\{\omega^\delta(s)\}_{\delta>0}$ is uniform bounded in $L^1(y\geq2)$, which implies that $\{\omega^\delta(s)\}_{\delta>0}$  converges weakly to $\omega$ in $L^1(y\geq2)$. This gives that
\[
\|\omega(s)\|_{L^1(y\geq2)}\leq C.
\]
By the same argument and Lemma \ref{velocity estimates 3}, we also have
\[
\|U(s)\|_{L^\infty(y\geq2)}\leq Cs^{-1/2}.
\]
Thus, we obtain the upper bound of $(U^\delta\omega^\delta-U\omega)$ as following
\begin{align*}
		\left\|(U^\delta\omega^\delta-U\omega)(s)\right\|_{L^1(y\geq2)}
		\leq \|(U^\delta,U)(s)\|_{L^\infty(y\geq2)}\|(\omega^\delta,\omega)(s)\|_{L^1(y\geq2)}
		\leq Cs^{-1/2},
	\end{align*}
	which implies 
\beno
 (t-s)^{-1/2}\left\|(U^\delta\omega^\delta-U\omega)(s)\right\|_{L^1(y\geq2)}ds \leq C (t-s)^{-1/2}s^{-1/2} \in L^1(0, t).
\eeno

Next,  we present the proof of pointwise convergence. By Lemma \ref{velocity estimates 3} and Proposition \ref{prop: uniform estimates for E(t)}, we have
	\begin{align}\label{pointwise convergence for nonlinear term}
		&\left\|(U^\delta\omega^\delta-U\omega)(s)\right\|_{L^1(y\geq2)}\\
		\nonumber
		&\leq \|(U^\delta-U)(s)\|_{L^4(y\geq2)}\|\omega^\delta(s)\|_{L^{4/3}(y\geq2)}
		+\|U(s)\|_{L^\infty(y\geq2)}\|(\omega^\delta-\omega)(s)\|_{L^1(y\geq2)}\\
		\nonumber
		&\leq Cs^{-1/4}\|(U^\delta-U)(s)\|_{L^4(y\geq2)}
		+C\|U(s)\|_{L^\infty(y\geq2)} \|(\omega^\delta-\omega)(s)\|_{L^1(y\geq2)}\nonumber\\
		&\to_{\delta\to 0} Cs^{-1/4}\lim_{\delta\rightarrow0}\|(U^\delta-U)(s)\|_{L^4(y\geq2)}, \quad a.e. \ s\in(0,t). \nonumber
	\end{align}

For $\|(U^\delta-U)(s)\|_{L^4(y\geq2)}$, a direct computation gives
	\begin{align*}
		&\|(U^\delta-U)(s)\|_{L^4(y\geq2)}\\
		&\leq \|BS_{\mathbb R^2_+}[\mathbb I_{y\leq3/4}(\omega^\delta-\omega)(s)]\|_{L^4(y\geq2)}
		+\|BS_{\mathbb R^2_+}[\mathbb I_{y\geq3/4}(\omega^\delta-\omega)(s)]\|_{L^4}\\
		&\leq C\left\|\|(\omega^\delta-\omega)_\xi(s)\|_{L^1(y\leq3/4)}\right\|_{L^{4/3}_\xi}
		+C\|(\omega^\delta-\omega)(s)\|_{L^{4/3}(y\geq3/4)}\\
		&\leq C\left\|\|(\omega^\delta-\omega)_\xi(s)\|_{L^1(y\leq3/4)}\right\|_{L^{1}_\xi\cap L^2_\xi}
		+C\|(\omega^\delta-\omega)(s)\|_{L^{4/3}(y\geq3/4)}\\
		&\to_{\delta\to 0} 0, \quad a.e. \ s\in(0,t). \nonumber
	\end{align*}
Substituting the above argument into \eqref{pointwise convergence for nonlinear term}, we have
\begin{align*}
	\lim_{\delta\rightarrow0}
	\left\|(U^\delta\omega^\delta-U\omega)(s)\right\|_{L^1(y\geq2)}=0, \qquad a.e. \ s\in(0,t).
\end{align*}

Then, by means of the Dominated Convergence Theorem, we obtain the desired result.
\end{proof}

Now we are in a  position to deal with the convergence of  boundary term.

\begin{lemma}\label{lem: boundary convergence overall} 
	For $0< t\leq T_0$, it holds that
	\begin{align}\label{convergence of the boundary condition}
	\lim_{\delta\rightarrow0}
	\int_0^t\|B^\delta_\xi(s)-B_\xi(s)\|_{L^1_\xi\cap L^2_\xi}ds=0,
\end{align}
where $B=\pa_y\Delta_D^{-1}(U\cdot\nabla\omega)|_{y=0}-(\pa_y+|D_x|)\omega_c|_{y=0}$.
\end{lemma}
\begin{proof}
We recall the definition of $B^\delta$:\beno
B^\delta=\pa_y\Delta_D^{-1}(U^\delta\cdot\nabla\omega^\delta)|_{y=0}-(\pa_y+|D_x|)\omega_c^\delta|_{y=0}.
\eeno
First, by the explicit formulation \eqref{omega c} and a direct computation, we obtain 
	\begin{align*}
		\lim_{\delta\rightarrow0}\|(\pa_y+|\xi|)(\omega_c^\delta)_\xi|_{y=0}-(\pa_y+|\xi|)(\omega_c)_\xi|_{y=0}\|_{L^1_\xi\cap L^2_\xi}=0,\quad 0<t\leq T_0.
	\end{align*}
	By the Dominated Convergence Theorem and \eqref{boundedness of omega c 2}, we have
	\begin{align*}
		\lim_{\delta\rightarrow0} \int_0^t \|(\pa_y+|\xi|)(\omega_c^\delta)_\xi|_{y=0}-(\pa_y+|\xi|)(\omega_c)_\xi|_{y=0}\|_{L^1_\xi\cap L^2_\xi}ds =0
	\end{align*}
	
Next, we  deal with the first term of $B^\delta$. Owing to the distinct regularity properties of $\omega^\delta$ in different domains, 
we introduce the following decomposition
	\begin{align}\label{eq:ow}
		\pa_y\Delta_D^{-1}(U^\delta\cdot\nabla\omega^\delta)|_{y=0}=&\pa_y\Delta_D^{-1}\big(U^\delta\cdot\nabla(\chi_b\omega^\delta)\big)\big|_{y=0}\\
		&\quad+\pa_y\Delta_D^{-1}\big(U^\delta\cdot\nabla((1-\chi_b)\omega^\delta)\big)\big|_{y=0}.\nonumber
	\end{align}
We rewrite the first term of \eqref{eq:ow} as
	\begin{align*}	\Big(\pa_y\Delta_D^{-1}\big(U^\delta\cdot\nabla(\chi_b\omega^\delta)\big)\Big)_\xi|_{y=0}
		= -\int_0^{+\infty} e^{-|\xi|z} \Big( \big(u^\delta\pa_x(\chi_b\omega^\delta)\big)_\xi(t,z) +\big(\frac{v^\delta}{y}y\pa_y(\chi_b\omega^\delta)\big)_\xi(t,z) \Big) dz.
	\end{align*}
Armed with Lemma \ref{lem: convergence in space-time norms} and Lemma \ref{lem: convergence in space-time norms of velocity}, we obtain  
	\begin{align*}
		\lim_{\delta\rightarrow0}\int_0^t\left\|\Big(\pa_y\Delta_D^{-1}\big(U^\delta\cdot\nabla(\chi_b\omega^\delta)\big)\Big)_\xi|_{y=0}-\Big(\pa_y\Delta_D^{-1}\big(U\cdot\nabla(\chi_b\omega)\big)\Big)_\xi|_{y=0} \right\|_{L^1_\xi\cap L^2_\xi}ds=0.
	\end{align*}	
	For the second term of \eqref{eq:ow}, taking Fourier transformation yields
	\begin{align*}
&\Big(\pa_y\Delta_D^{-1}\big(U^\delta\cdot\nabla((1-\chi_b)\omega^\delta)\big)\Big)_\xi|_{y=0}\\
		&=\int_{\mathbb R^2_+}e^{2\pi i x'\xi-2\pi y'|\xi|}\big(U^\delta\cdot\nabla((1-\chi_b)\omega^\delta)\big)(x',y')dx'dy'\\
		&=-\int_{\mathbb R^2_+} \nabla_{x',y'}\big(e^{2\pi i x'\xi-2\pi y'|\xi|}\big)
		\cdot \big((1-\chi_b)U^\delta\omega^\delta)\big)(x',y')dx'dy',
	\end{align*}
	which along with Lemma \ref{lemma: u-om} implies
	\begin{align*}
		\lim_{\delta\rightarrow0}\int_0^t\left\|\Big(\pa_y\Delta_D^{-1}\big(U^\delta\cdot\nabla((1-\chi_b)\omega^\delta)\big)\Big)_\xi|_{y=0}-\Big(\pa_y\Delta_D^{-1}\big(U\cdot\nabla((1-\chi_b)\omega)\big)\Big)_\xi|_{y=0} \right\|_{L^1_\xi\cap L^2_\xi}ds=0.
	\end{align*}
	Combining the above estimates yields the desired result.
\end{proof}

\subsection{Proof of Proposition \ref{prop: existence part}}\label{sec: Convergence of the integral equations}
In this subsection, we prove Proposition \ref{prop: existence part} by showing that the pair $(U, \omega)$ obtained in the previous subsection satisfies \eqref{eq: NS vorticity}.\smallskip

\underline{Near the boundary.}
 By Lemma \ref{lem: convergence in space-time norms}, 
 Lemma \ref{lem: convergence of product terms in space-time}, Lemma \ref{lem: boundary convergence overall} and \eqref{convergence of omega c delta}, \eqref{convergence of b delta},  letting $\delta\rightarrow 0^+$ in \eqref{integral eq of omega-omega c-1} yields that,  in the sense of distributions,  $(U,\omega)$ satisfies 
\begin{align}\label{eq: integral eq of omega delta-omega c without delta}
	(\chi_b\omega-\chi_b\omega_c)_\xi(t,y)
 	=&\int_0^{+\infty}\big(H_\xi(t,y,z)+R_\xi(t,y,z)\big)b_\xi(z)dz\\
 	\nonumber
 	&+\int_0^t\int_0^{+\infty}\big(H_\xi(t-s,y,z)+R_\xi(t-s,y,z)\big)N_\xi(s,z)dzds\\
 	\nonumber
 	&-\int_0^t\big(H_\xi(t-s,y,0)+R_\xi(t-s,y,0)\big)B_\xi(s)ds.
\end{align}
Due to the regularity of $\omega_\xi, N_\xi, B_\xi$, the equation \eqref{eq: integral eq of omega delta-omega c without delta} point-wise  holds.

\underline{Away from the boundary.} Since the mollified function $\omega^\delta$ satisfies \eqref{eq: NS vorticity}, we have 
\begin{align}\label{eq: (1-chib)omega with delta}
	(1-\chi_b)\omega^\delta(t)
	=&e^{t\Delta} \omega_0^\delta
	-\int_0^t e^{(t-s)\Delta}\dv\Big((1-\chi_b)U^\delta\omega^\delta\Big)(s)ds\\
	\nonumber
	&+\int_0^t e^{(t-s)\Delta}
	\Big(\chi_b''\omega^\delta
	+2\chi_b'\pa_y\omega^\delta
	-\chi_b' U^\delta\omega^\delta\Big)(s)ds.
\end{align}
By Lemma \ref{lem: convergence in space-time norms}, Lemma \ref{lemma: u-om}, letting $\delta\rightarrow 0^+$ in \eqref{eq: (1-chib)omega with delta}  yields the following equation

\begin{align}\label{eq: (1-chib)omega without delta}
	(1-\chi_b)\omega(t)
	=&e^{t\Delta} \omega_0
	-\int_0^t e^{(t-s)\Delta}\dv\Big((1-\chi_b)U\omega\Big)(s)ds\\
	\nonumber
	&+\int_0^t e^{(t-s)\Delta}
	\Big(\chi_b''\omega
	+2\chi_b'\pa_y\omega
	-\chi_b' U\omega\Big)(s)ds.
\end{align}
Rewriting \eqref{eq: integral eq of omega delta-omega c without delta} and \eqref{eq: (1-chib)omega without delta} back to differential equations and summing them up, and recalling that $\omega_c$ satisfies $\pa_t\omega_c=\pa_y^2\omega_c$ for $t>0$, we deduce that $(U,\omega)$ satisfies Navier-Stokes equations for $t\in(0, T_0)$.

Next, we prove the solution $(U, \omega)\in  L^{2,\infty}(\mathbb R^2_+)\times  L^1(\mathbb R^2_+)$ for $t\in (0, T_0]$. By Corollary \ref{cor:E}, we have
\begin{align*}
	\|\omega^\delta(t)\|_{L^1}
	&\leq \|\omega^\delta(t)\|_{L^1(y\geq1/2)}
	+\|\omega^\delta(t)\|_{L^1(y\leq1/2)}\\
	&\leq C+\|(1,x)\omega^\delta(t)\|_{L^2(y\leq1/2)}\\
	&\leq C+\left\|\big((1,x)\omega^\delta\big)_\xi(t)\right\|_{L^2_\xi L^2_y}\\
	&\leq C+\left\|\big((1,x)\omega^\delta\big)_\xi(t)\right\|_{L^2_\xi L^1_y}^{1/2}
	\left\|\pa_y\big((1,x)\omega^\delta\big)_\xi(t)\right\|_{L^2_\xi L^1_y}^{1/2}\\
	&\leq Ct^{-1/4}.
\end{align*} 
Therefore, for each fixed $t>0$, there exists a subsequence of $\omega^\delta$ converging to $\omega$ weakly in $L^1$ by uniqueness of the limit. Thus, we have $\omega\in L^1(\mathbb R^2_+)$ for $t>0$ and $U\in L^{2,\infty}(\mathbb R^2_+)$ holds by Lemma \ref{lem: velocity weighted est}. 

Taking $t=\frac{T_0}{2}$ as a new initial time, the data no longer contains the point vortex singularity. Therefore, by applying Theorem \ref{thm: Ken}, we can extend this local solution to a global one.

\medskip

\underline{Initial condition.}
Let $t\rightarrow0^+$ in \eqref{def: omega c without delta}. We derive
\begin{align}\label{convergence of omega c t to 0}
	\omega_c(t)\rightharpoonup -u_0\delta_{\partial \mathbb R^2_+}-u_0\chi_b'
	\qquad in \quad \mathcal{M}(\overline{\mathbb R^2_+}).
\end{align}
Let $t\rightarrow0^+$ in \eqref{eq: integral eq of omega delta-omega c without delta} and \eqref{eq: (1-chib)omega without delta} and recall $b=u_0\chi_b\chi_b'$. We obtain
\begin{align*}
	\omega(t)-\chi_b\omega_c(t)\rightharpoonup u_0\chi_b\chi_b'+\omega_0,
\end{align*}
which along  with \eqref{convergence of omega c t to 0} implies \eqref{initial limit}.
\ef

\medskip

  \section{Energy estimates near point vortex}\label{section-E_{vp}}
  
  In the following sections (from Section \ref{section-E_{vp}}--Section \ref{section-E_{b}}), we derive the uniform estimates for energy function $E^\delta(t)$ and prove Proposition \ref{prop: uniform estimates for E(t)}. To simplify notations, we drop superscript $\delta$ in the subsequent sections.

In this section, we present uniform estimates of $E_{vp}(t)$. We recall the definition of $E_{vp}(t)$:
  \begin{align*}
  	E_{vp}(t):=\sup_{\log\delta\leq\tau'<\log t}
  	\Big( \|\mathcal W_R(\tau')\|_{L^2(m)}
  	+\|\nabla_\eta\mathcal W_R(\tau')\|_{L^2(m)}\Big),
  \end{align*}
where $\mathcal W_R$ satisfies 
	\begin{align*} 
		\pa_\tau\mathcal W_R
  	+\alpha\mathcal V_R\cdot\nabla_\eta G
  	+\alpha\mathcal V^G\cdot\nabla_\eta\mathcal W_R
  	-\mathcal L \mathcal W_R
  	=\sum_{i=1}^7 F_i,\qquad \mathcal W_R|_{\tau=\tau_0}=0,
	\end{align*}
with $F_i$ defined in Lemma \ref{lem: eq of WR}. 

Here we note that this system of $\mathcal W_R$ is similar to that in the whole space. Thus, we employ some tools from \cite{Gallagher, Gallay 1} to deal with $\mathcal W_R$. First, we introduce a semigroup $T_\alpha(\tau)$ generated by
\begin{align*}
	\pa_\tau\mathcal W_R
  	+\alpha\mathcal V_R\cdot\nabla_\eta G
  	+\alpha\mathcal V^G\cdot\nabla_\eta\mathcal W_R
  	-\mathcal L \mathcal W_R=0.
\end{align*}
Thus,  for $\tau_0<\tau<\log(T_0)$, $\mathcal W_R$ is rewritten as follows
\begin{align}\label{eq: integral eq of W R}
	\mathcal W_R(\tau)
	=\sum_{i=1}^7 \int_{\tau_0}^\tau T_\alpha(\tau-\tau') F_i(\tau')d\tau'
	:=\sum_{i=1}^7 I_i.
\end{align}

\medskip
 
Subsequently, we state some properties of $T_\alpha(\tau)$.
 
\begin{proposition}\label{semigroup estimates of T alpha}
	Let $\alpha\in\mathbb R$, $m>2$ and $a(\tau)=1-e^{-\tau}.$ It holds that
\begin{enumerate}[label=(\arabic*)]	
	\item There exists $C_\alpha>0$ such that for all $w\in L^2(m)$,
	\begin{align}
	\label{semigroup est 1}
		\|T_\alpha(\tau)w\|_{L^2(m)}
		&\leq C_\alpha \|w\|_{L^2(m)},\quad \tau\geq 0,
\\
\label{semigroup est 2}
		a(\tau)^{1/2}\|\nabla T_\alpha(\tau)w\|_{L^2(m)}
		&\leq C_\alpha \|w\|_{L^2(m)},\quad \tau> 0.
	\end{align}
	
\item  For $1<p\leq2$, there exists $C_{\alpha,p}$ such that for all $w\in L^2(m)$,
	\begin{align}\label{semigroup est 3}
		\|T_\alpha(\tau)\nabla w\|_{L^2(m)}
		&\leq C_{\alpha,p}\frac{e^{-\tau/2}}{a(\tau)^{1/p}}\|w\|_{L^p(m)},\quad \tau> 0,\\
		\label{semigroup est 4}
		a(\tau)^{1/2}\|\nabla T_\alpha(\tau)\nabla w\|_{L^2(m)}
		&\leq C_{\alpha,p}\frac{e^{-\tau/2}}{a(\tau)^{1/p}}\|w\|_{L^p(m)},\quad \tau> 0.
	\end{align}

	\end{enumerate} 
	
\end{proposition}

\begin{proof}
The estimates \eqref{semigroup est 1}, \eqref{semigroup est 3} were established in \cite{Gallagher}. Here,  we focus on \eqref{semigroup est 2} and \eqref{semigroup est 4}.\smallskip
	
\underline{The proof of \eqref{semigroup est 2}.}  For any $w_0\in L^2(m)$, denote $w=T_\alpha(\tau)w_0$.  Thus,
	\begin{align*}
		\left\{
		\begin{aligned}
			&\pa_\tau w+\alpha\big(\mathcal V^G\cdot\nabla w+\mathcal V ^w\cdot\nabla G\big)-\mathcal L w=0,\\
			&w|_{\tau=0}=w_0,
		\end{aligned}
		\right.
	\end{align*}
	where $\mathcal V ^w$ denotes $BS_{\mathbb R^2}[w]$. For $i=1,2$, 
	\begin{align*}
		\left\{
		\begin{aligned}
			&\pa_\tau(\pa_i w)
			+\alpha\big(\mathcal V^G\cdot\pa_i w+\mathcal V^{\pa_i w}\cdot\nabla G\big)-\mathcal L(\pa_i w)
			+\alpha\big(\pa_i\mathcal V^G\cdot\nabla w
			+\mathcal V^w\cdot\nabla\pa_i G\big)
			-\frac{1}{2}\pa_i w=0,\\
			&\pa_i w|_{\tau=0}=\pa_i w_0.
		\end{aligned}
		\right.
	\end{align*}
	Thus,
	\begin{align*}
		&\pa_\tau(e^{-\tau/2}\pa_i w)
		+\alpha\big(\mathcal V^G\cdot\nabla(e^{-\tau/2}\pa_i w)+\mathcal V^{e^{-\tau/2}\pa_i w}\cdot\nabla G\big)\\
		&-\mathcal L(e^{-\tau/2}\pa_i w)
		+\alpha\big(\pa_i\mathcal V^G\cdot\nabla(e^{-\tau/2}w)+\mathcal V^{e^{-\tau/2}w}\cdot\nabla\pa_i G\big)=0,
	\end{align*}
	which implies
	\begin{align*}
		\pa_i w(\tau)
		=e^{\tau/2}T_\alpha(\tau)\pa_i w_0
		-\alpha\int_0^\tau T_\alpha(\tau-\tau')\big(\pa_i\mathcal V^G\cdot\nabla w
		+\mathcal V^w\cdot\nabla\pa_i G\big)(\tau')d\tau'.
	\end{align*}
	Based on \eqref{semigroup est 1}, \eqref{semigroup est 3}, Lemma \ref{lem: velocity weighted est} and the relation $w=T_\alpha(\tau)w_0$, we have
	\begin{align*}
		&\|\pa_i w(\tau)\|_{L^2(m)}
		\leq e^{\tau/2}\|T_\alpha(\tau)\pa_i w_0\|_{L^2(m)}
		+\alpha\int_0^\tau\left\|T_\alpha(\tau-\tau')\dv(\pa_i \mathcal V^G w+\mathcal V^w\pa_i G)\right\|_{L^2(m)}d\tau'\\
		&\leq C_\alpha a(\tau)^{-1/2}\|w_0\|_{L^2(m)}
		+C_\alpha \int_0^\tau \frac{e^{-\frac{\tau-\tau'}{2}}}{a(\tau-\tau')^{1/2}}\big(\|\pa_i\mathcal V^G\|_{L^\infty}\|w\|_{L^2(m)}
		+\|\mathcal V^w\|_{L^4}\|\pa_i G\|_{L^4(m)}\big)d\tau'\\
		&\leq C_\alpha a(\tau)^{-1/2}\|w_0\|_{L^2(m)}
		+C_\alpha \int_0^\tau \frac{e^{-\frac{\tau-\tau'}{2}}}{a(\tau-\tau')^{1/2}}d\tau'\cdot\|w_0\|_{L^2(m)}\\
		&\leq C_\alpha a(\tau)^{-1/2}\|w_0\|_{L^2(m)}.
	\end{align*}

	\underline{The proof of \eqref{semigroup est 4}.}  If $\tau\leq2$, it holds that
	\begin{align*}
		a(\tau)^{1/2}\|\nabla T_\alpha(\tau)\nabla w\|_{L^2(m)}
		&= a(\tau)^{1/2}\|\nabla T_\alpha(\tau/2)T_\alpha(\tau/2)\nabla w\|_{L^2(m)}\\
		&\leq C_\alpha\|T_\alpha(\tau/2)\nabla w\|_{L^2(m)}\\
		&\leq C_\alpha a(\tau)^{-1/p}\|w\|_{L^p(m)}.
	\end{align*}
For $\tau>2$, we have
	\begin{align*}
		 a(\tau)^{1/2}\|\nabla T_\alpha(\tau)\nabla w\|_{L^2(m)}
		&= a(\tau)^{1/2}\|\nabla T_\alpha(1)T_\alpha(\tau-1)\nabla w\|_{L^2(m)}\\
		&\leq C_\alpha\|T_\alpha(\tau-1)\nabla w\|_{L^2(m)}
		\leq C_\alpha e^{-\frac{\tau-1}{2}}\|w\|_{L^p(m)}\\
		&\leq C_\alpha e^{-\tau/2}\|w\|_{L^p(m)}.
	\end{align*}
	
	Combing all above estimates, we obtain the desired results.

\end{proof}

\medskip

We are now ready to prove Proposition \ref{prop: uniform estimates for Evp(t)} using Lemmas \ref{lem: estimate of Evp 1} and \ref{lem: estimate of Evp 2} below.

\begin{lemma}\label{lem: estimate of Evp 1}
	There exists $T_0>0$ such that for $0<t\leq T_0$, it holds that
	\begin{align*}
		\sup_{\log\delta\leq \tau'<\log t}\|\mathcal W_R(\tau')\|_{L^2(m)}
		\leq& C(t+\delta)^{1/4} +CE(t)^2.
	\end{align*}
\end{lemma}

\begin{proof}

We recall the formulation of $\mathcal W_R=\sum_{i=1}^7 I_i$ with $I_i$ defined by  \eqref{eq: integral eq of W R}. Next, we present the estimates of $I_i$ term by term.
	
	For $I_1$, we get by \eqref{semigroup est 3} that
\begin{align*}
	\|I_1(\tau)\|_{L^2(m)}
	&\leq \int_{\tau_0}^\tau \Big\|T_\alpha(\tau-\tau')\dv \big\{\alpha\Big(BS_{\mathbb R^2}[(1-\chi_{vp})G](\eta)
  	-\widetilde{BS_{\mathbb R^2}[(1-\chi_{vp})G]}(\eta+\frac{(0,40)}{e^{\tau'/2}})\\
  	&\quad\quad+\mathcal V^G(\eta+\frac{(0,40)}{e^{\tau'/2}})\Big)(\chi_{vp} G+\mathcal W_R) \big\}\Big\|_{L^2(m)}d\tau'\\
  	&\leq C_0 \int_{\tau_0}^\tau \frac{e^{-\frac{\tau-\tau'}{2}}}{a(\tau-\tau')^{1/2}}\Big\|BS_{\mathbb R^2}[(1-\chi_{vp})G](\eta)
  	-\widetilde{BS_{\mathbb R^2}[(1-\chi_{vp})G]}(\eta+\frac{(0,40)}{e^{\tau'/2}})\\
  	&\quad\quad+\mathcal V^G(\eta+\frac{(0,40)}{e^{\tau'/2}})\Big\|_{L^\infty(|\eta|\leq6e^{-\tau'/2})}\|\chi_{vp} G+\mathcal W_R \|_{L^2(m)}d\tau'\\
  	&\leq C_0 \int_{\tau_0}^\tau \frac{e^{-\frac{\tau-\tau'}{2}}}{a(\tau-\tau')^{1/2}} e^{\tau'/2}(1+\|\mathcal W_R\|_{L^2(m)})d\tau'\\
  	&=C_0 \int_0^t (\frac{s+\delta}{t+\delta})^{1/2}(1-\frac{s+\delta}{t+\delta})^{-1/2}(s+\delta)^{-1/2} ds\cdot(1+E_{vp}(t))\\
  	&\leq C_0 t^{1/2}(1+E_{vp}(t)),
\end{align*}
	where we usd the following fact
\begin{align}\label{fact 1 in lemma 4.3}
	\|BS_{\mathbb R^2}[(1-\chi_{vp})G]\|_{L^\infty}
	\leq C_0\|(1-\chi_{vp})G\|_{L^{4/3}}\|(1-\chi_{vp})G\|_{L^4}\leq C_0e^{\tau'/2}, 
\end{align}
due to the fact $\operatorname{supp}(1-\chi_{vp})\subseteq\{|\eta|\geq 5e^{-\tau'/2}\}$.

For $I_2$, we get by  \eqref{semigroup est 1} that 
\begin{align*}
	&\|I_2(\tau)\|_{L^2(m)}
	\leq C_0\int_{\tau_0}^\tau \left\|\frac{1}{2}\eta\cdot\nabla_\eta( \chi_{vp})G+2\nabla_\eta(\chi_{vp})\cdot\nabla_\eta G
  	+\Delta_\eta(\chi_{vp}) G \right\|_{L^2(m)}d\tau'\\
  	&\leq C_0\int_{\tau_0}^\tau \left\|e^{-|\eta|^2/8}\chi_{(5e^{-\tau'/2}\leq|\eta|\leq6e^{-\tau'/2})}\right\|_{L^2(m)}d\tau'
  	\leq C_0 \int_{\tau_0}^\tau e^{\tau'/2}d\tau'\\
  	&\leq C_0 (t+\delta)^{1/2}.
\end{align*}

For $I_3$, we similarly have
\begin{align*}
	\|I_3(\tau)\|_{L^2(m)}
	&\leq C_0\int_{\tau_0}^\tau \left\|\alpha(1-\chi_{vp})\mathcal V_R\cdot\nabla_\eta G-\alpha\mathcal V_R\cdot\nabla	_\eta(\chi_{vp})G\right\|_{L^2(m)}d\tau'\\
	&\leq C_0\int_{\tau_0}^\tau \|\mathcal V_R\|_{L^4}\|e^{-|\eta|^2/8}\|_{L^4(|\eta|\geq5e^{-\tau'/2})}d\tau'\\
	&\leq C_0\int_{\tau_0}^\tau e^{\tau'/2}\|\mathcal W_R\|_{L^2(m)}d\tau'\\
	&\leq C_0 (t+\delta)^{1/2} E_{vp}(t),
\end{align*}
where we used $\|\mathcal V_R\|_{L^4} \leq C \|\mathcal W_R\|_{L^2(m)}$ by   Lemma \ref{lem: velocity weighted est}.

For $I_4$, by \eqref{semigroup est 3}, we have
\begin{align*}
	&\|I_4(\tau)\|_{L^2(m)}
	\leq \int_{\tau_0}^\tau
	\left\|T_\alpha(\tau-\tau')\dv\{\alpha\widetilde{\mathcal V_R}(\eta+\frac{(0,40)}{e^{\tau'/2}},\tau')(\chi_{vp} G)\}\right\|_{L^2(m)}d\tau'\\
	&\leq C_0\int_{\tau_0}^\tau \frac{e^{-\frac{\tau-\tau'}{2}}}{a(\tau-\tau')^{1/2}}\left\|\widetilde{\mathcal V_R}(\eta+\frac{(0,40)}{e^{\tau'/2}},\tau')\cdot\chi_{vp} G(\eta)\right\|_{L^2(m)}d\tau'\\
	&\leq C_0\int_{\tau_0}^\tau \frac{e^{-\frac{\tau-\tau'}{2}}}{a(\tau-\tau')^{1/2}}\left\|\langle\eta\rangle^{\gamma_0-\frac{2}{q}}\widetilde{\mathcal V_R}(\eta)\right\|_{L^q}
	\left\|\langle\eta+\frac{(0,40)}{e^{\tau'/2}}\rangle^{\frac{2}{q}-\gamma_0}\langle\eta\rangle(\chi_{vp} G)(\eta)\right\|_{L^{\frac{2q}{q-2}}}d\tau'\\
	&\leq C_0\int_{\tau_0}^\tau \frac{e^{-\frac{\tau-\tau'}{2}}}{a(\tau-\tau')^{1/2}}\|\mathcal W_R\|_{L^2(\gamma_0)}e^{\tau'(\frac{\gamma_0}{2}-\frac{1}{q})}d\tau'\\
	&\leq C_0 (t+\delta)^{\frac{\gamma_0}{2}-\frac{1}{q}}E_{vp}(t)
    =C_0 (t+\delta)^{1/4}E_{vp}(t),
\end{align*}
where we used Lemma \ref{lem: velocity weighted est} and choose some $\gamma_0\in(0,1), q>2$ such that $\frac{\gamma_0}{2}-\frac{1}{q}=\frac{1}{4}$.

For $I_5$, taking $p=4/3$ and by \eqref{semigroup est 3}, Lemma \ref{lem: velocity weighted est}, we find
\begin{align*}
	\|I_5(\tau)\|_{L^2(m)}
	&\leq \int_{\tau_0}^\tau
	\left\|T_\alpha(\tau-\tau')\dv \{-\alpha\mathcal V_R\mathcal W_R+\alpha\widetilde{\mathcal V_R}(\eta+\frac{(0,40)}{e^{\tau'/2}},\tau')\mathcal W_R\}\right\|_{L^2(m)} d\tau'\\
	&\leq C_0 \int_{\tau_0}^\tau \frac{e^{-\frac{\tau-\tau'}{2}}}{a(\tau-\tau')^{1/p}}\left\|\big(-\alpha\mathcal V_R\mathcal W_R+\alpha\widetilde{\mathcal V_R}(\eta+\frac{(0,40)}{e^{\tau'/2}},\tau')\big)\mathcal W_R\right\|_{L^p(m)}d\tau'\\
	&\leq C_0 \int_{\tau_0}^\tau \frac{e^{-\frac{\tau-\tau'}{2}}}{a(\tau-\tau')^{1/p}}\|\mathcal V_R\|_{L^{\frac{2p}{2-p}}}\|\mathcal W_R\|_{L^2(m)}d\tau'\\
	&\leq C_0 \int_{\tau_0}^\tau \frac{e^{-\frac{\tau-\tau'}{2}}}{a(\tau-\tau')^{1/p}}\|\mathcal W_R\|_{L^2(m)}^2d\tau'\\
	&\leq C_0 E_{vp}(t)^2.
\end{align*}

For $I_6$, recalling $\zeta_1$ defined in \eqref{def zeta1}, and using \eqref{semigroup est 3} with $p=4/3$, we apply Lemmas \ref{velocity estimates 2} and \ref{lem: velocity weighted est} to obtain
\begin{align*}
	\|I_6(\tau)&\|_{L^2(m)}
	\leq C_0 \int_{\tau_0}^\tau \frac{e^{-\frac{\tau-\tau'}{2}}}{a(\tau-\tau')^{1/2}}\cdot e^{\tau'/2}\left\|BS_{\mathbb R^2_+}[\zeta_1\omega]\right\|_{L^\infty}\big(\|G\|_{L^2(m)}+\|\mathcal W_R\|_{L^2(m)}\big)d\tau'\\
	&\quad+C_0 \int_{\tau_0}^\tau \frac{e^{-\frac{\tau-\tau'}{2}}}{a(\tau-\tau')^{1/p}}\cdot e^{\tau'/2}\left\|BS_{\mathbb R^2_+}[(1-\zeta_1-\chi_{vp})\omega]\right\|_{L^{\frac{2p}{2-p}}}\big(\|G\|_{L^2(m)}+\|\mathcal W_R\|_{L^2(m)}\big)d\tau'\\
	&\leq C_0 \int_{\tau_0}^\tau \frac{e^{-\frac{\tau-\tau'}{2}}}{a(\tau-\tau')^{1/2}}\cdot e^{\tau'/2}\|\omega(s)\|_{Y_1(s)}d\tau'\cdot(1+E_{vp}(t))\\
	&\quad+C_0 \int_{\tau_0}^\tau \frac{e^{-\frac{\tau-\tau'}{2}}}{a(\tau-\tau')^{1/p}}\cdot e^{\tau'/2}\|e^\Psi \psi\chi_m\omega(s)\|_{L^2}d\tau'\cdot(1+E_{vp}(t))\\
	&\leq C_0 (t+\delta)^{1/2}\big(1+E_b(t)+E_m(t)\big)(1+E_{vp}(t)).
\end{align*}

For $I_7$, by \eqref{semigroup est 3}, we have
\begin{align*}
	&\|I_7(\tau)\|_{L^2(m)}
	\leq C_0 \int_{\tau_0}^\tau \left\| (s+\delta)^2 U\cdot\nabla\chi_{vp}\omega
	-2(s+\delta)^2\nabla\chi_{vp}\cdot\nabla\omega
	-(s+\delta)^2\Delta\chi_{vp} \omega\right\|_{L^2(m)}d\tau'\\
	&\leq C_0 \int_{\tau_0}^\tau (s+\delta)^2\|\langle\eta\rangle^m(|U\omega|+|\nabla\omega|+|\omega|)\|_{L^2(\operatorname{supp}\nabla\chi_{vp})}d\tau'\\
	&\leq C_0 \int_{\tau_0}^\tau (s+\delta)^{2-\frac{m}{2}}e^{-\frac{\eps_1}{2s}}\big(\|U\|_{L^\infty(\operatorname{supp}\chi_m)}\|e^\Psi \psi\chi_m\omega\|_{L^2}
	+\|e^\Psi \psi\nabla(\chi_m\omega)\|_{L^2}
	+\|e^\Psi \psi\chi_m\omega\|_{L^2}\big)d\tau'\\
	&\leq C_0\int_0^t s^4
	\big(\|U\|_{L^\infty(\operatorname{supp}\chi_m)}\|e^\Psi \psi\chi_m\omega\|_{L^2}
	+\|e^\Psi \psi\nabla(\chi_m\omega)\|_{L^2}
	+\|e^\Psi \psi\chi_m\omega\|_{L^2}\big)ds\\
	&\leq C t^5 E_m(t)\big( E(t)+1\big)
	+C t^4E_m(t),
\end{align*}
where we used Lemma \ref{velocity estimates 3} in the last step.

Summing these estimates and taking a suitable $\gamma_0$ and $p$, we obtain the desired result.
\end{proof}

Next, we give the estimates  second part of $E_{vp}$.

\begin{lemma}\label{lem: estimate of Evp 2}
	There exists $T_0>0$ such that for $0<t\leq T_0$, it holds that
	\begin{align*}
		\sup_{\log\delta\leq \tau'<\log t}\|\nabla_\eta\mathcal W_R(\tau')\|_{L^2(m)}
		\leq& C(t+\delta)^{1/4} +CE(t)^2.
	\end{align*}
\end{lemma}
\begin{proof}
Applying $\nabla_\eta$ to both sides of \eqref{eq: integral eq of W R}, it remains to handle  $\nabla_\eta I_1\sim \nabla_\eta I_7$ term by term.
	
For $\nabla_\eta I_1$, based on \eqref{semigroup est 2}, we utilize \eqref{fact 1 in lemma 4.3} to have
\begin{align*}
	\|\nabla_\eta  I_1(\tau)\|_{L^2(m)}
	&\leq \int_{\tau_0}^\tau \Big\|\nabla_\eta T_\alpha(\tau-\tau') \big\{\alpha\Big(BS_{\mathbb R^2}[(1-\chi_{vp})G](\eta)
  	\\
  	&\qquad\quad-\widetilde{BS_{\mathbb R^2}[(1-\chi_{vp})G]}(\eta+\frac{(0,40)}{e^{\tau'/2}})+\mathcal V^G(\eta+\frac{(0,40)}{e^{\tau'/2}})\Big)\cdot\nabla_\eta(\chi_{vp} G+\mathcal W_R) \big\}\Big\|_{L^2(m)} d\tau'\\
  	&\leq C_0 \int_{\tau_0}^\tau a(\tau-\tau')^{-1/2} \Big\| \Big(BS_{\mathbb R^2}[(1-\chi_{vp})G](\eta)
  	-\widetilde{BS_{\mathbb R^2}[(1-\chi_{vp})G]}(\eta+\frac{(0,40)}{e^{\tau'/2}})\\
  	&\quad\quad
  	+\mathcal V^G(\eta+\frac{(0,40)}{e^{\tau'/2}})\Big)\cdot\nabla_\eta(\chi_{vp} G+\mathcal W_R) \Big\|_{L^2(m)} d\tau'\\
  	&\leq C_0 \int_{-\infty}^\tau a(\tau-\tau')^{-1/2} \Big\| BS_{\mathbb R^2}[(1-\chi_{vp})G](\eta)
  	-\widetilde{BS_{\mathbb R^2}[(1-\chi_{vp})G]}(\eta+\frac{(0,40)}{e^{\tau'/2}})\\
  	&\quad\quad
  	+\mathcal V^G(\eta+\frac{(0,40)}{e^{\tau'/2}}) \Big\|_{L^\infty}
  	\Big\|\nabla_\eta(\chi_{vp} G+\mathcal W_R) \Big\|_{L^2(m)} d\tau'\\
  	&\leq C_0 \int_{-\infty}^\tau 
  	\big(1-e^{-(\tau-\tau')} \big)^{-1/2}e^{\tau'/2}d\tau' 
  	\cdot \big(1+E_{vp}(t)\big)\\
  	&=C_0 \int_0^t (1-\frac{s+\delta}{t+\delta})^{-1/2}(s+\delta)^{-1/2}ds
  	\cdot \big(1+E_{vp}(t)\big)\\
  	&\leq C_0 t^{1/2}\big(1+E_{vp}(t)\big),
\end{align*}
where we perform the change of variables $s+\delta=e^{\tau'}, t+\delta=e^\tau$ in the last line.

For $\nabla_\eta I_5$, by \eqref{semigroup est 2} and \eqref{semigroup est 4} (taking $p=4/3$), we have
\begin{align*}
	\|\nabla_\eta I_5(\tau)\|_{L^2(m)}
	&\leq \int_{\tau-\log 2}^\tau \Big\|\nabla_\eta T_\alpha(\tau-\tau') \big\{ \Big( -\alpha\mathcal V_R+\alpha\widetilde{\mathcal V_R}(\eta+\frac{(0,40)}{e^{\tau'/2}}) \Big)\cdot\nabla_\eta \mathcal W_R \Big\|_{L^2(m)}d\tau'\\
	&\quad+\int_{\tau_0}^{\tau-\log2}
	\Big\|\nabla_\eta T_\alpha(\tau-\tau') 
	\dv \{-\alpha\mathcal V_R\mathcal W_R+\alpha\widetilde{\mathcal V_R}(\eta+\frac{(0,40)}{e^{\tau'/2}},\tau')\mathcal W_R\}\Big\|_{L^2(m)} d\tau'\\
	&\leq C_0 \int_{\tau-\log 2}^\tau a(\tau-\tau')^{-1/2}\|\mathcal V_R\|_{L^\infty}\|\nabla_\eta\mathcal W_R\|_{L^2(m)}d\tau'\\
	&\quad +C_0\int_{-\infty}^{\tau-\log2}
	\frac{e^{-\frac{\tau-\tau'}{2}}}{a(\tau-\tau')^{1/2+1/p}}
	\Big\| -\alpha\mathcal V_R\mathcal W_R+\alpha\widetilde{\mathcal V_R}(\eta+\frac{(0,40)}{e^{\tau'/2}},\tau')\mathcal W_R\Big\|_{L^p(m)}d\tau'\\
	&\leq C_0 \int_{\tau-\log 2}^\tau a(\tau-\tau')^{-1/2}\|\mathcal W_R\|_{L^{4/3}}^{1/2}\|\mathcal W_R\|_{L^4}^{1/2}\|\nabla_\eta\mathcal W_R\|_{L^2(m)}d\tau'\\
	&\quad 
	+C_0\int_{-\infty}^{\tau-\log2}
	\frac{e^{-\frac{\tau-\tau'}{2}}}{a(\tau-\tau')^{1/2+1/p}}
	\|\mathcal V_R\|_{L^{\frac{2p}{2-p}}}\|\mathcal W_R\|_{L^2(m)}d\tau'\\
	&\leq C_0 \int_{\tau-\log 2}^\tau a(\tau-\tau')^{-1/2}\|\mathcal W_R\|_{L^2(m)}^{3/4}\|\nabla_\eta\mathcal W_R\|_{L^2(m)}^{3/2}d\tau'\\
	&\quad +C_0\int_{-\infty}^{\tau-\log2}
	\frac{e^{-\frac{\tau-\tau'}{2}}}{a(\tau-\tau')^{1/2+1/p}}
	\|\mathcal W_R\|_{L^2(m)}^2d\tau'\\
	&\leq C_0\int_{t/2}^t (1-\frac{s+\delta}{t+\delta})^{-1/2}(s+\delta)^{-1}ds \cdot E_{vp}(t)^2
	\\
	&\quad +C_0 \int_0^{t/2} (\frac{s+\delta}{t+\delta})^{1/2} (1-\frac{s+\delta}{t+\delta})^{-5/4}(s+\delta)^{-1}ds \cdot E_{vp}(t)^2\\
	&\leq C_0 E_{vp}(t)^2.
\end{align*}

The remainder follows by the same argument as for $\nabla_\eta I_1$. Thus, we have
\begin{align*}
	&\|\nabla_\eta I_2(\tau)\|_{L^2(m)}
	+\|\nabla_\eta I_3(\tau)\|_{L^2(m)}
	\leq C_0 (t+\delta)^{1/2}(1+E_{vp}(t)),\\
	&\|\nabla_\eta I_4(\tau)\|_{L^2(m)}\leq C_0(t+\delta)^{1/4}E_{vp}(t), 
	\end{align*} 
	 and
	\begin{align*}
	&\|\nabla_\eta I_6(\tau)\|_{L^2(m)}
	\leq C_0 (t+\delta)^{1/2}\big(1+E_b(t)+E_m(t)\big)(1+E_{vp}(t)), \\
	&\|\nabla_\eta I_7(\tau)\|_{L^2(m)}
	\leq C t^3 E_m(t)\big( E(t)+1\big) ,
\end{align*}

Summing these estimates, we obtain the desired result.
\end{proof}

 \section{Energy estimates in the interaction region} \label{section-E_{m}}
 
This section is devoted to the proof of Proposition \ref{prop: uniform estimates for Em(t)}. First, we recall that the energy functional $E^\delta_{m}(t)$ is defined by
    \begin{align*} 
   	E_m(t):=\sup_{0<s<t}\left\|e^\Psi\psi\chi_m\omega(s)\right\|_{L^2\cap L^4}+\left\|e^\Psi\psi \na(\chi_m\omega )\right\|_{L^2(0, t;L^2)}
   	+e^{\frac{5\eps_0}{t}}\sup_{0<s<t}\|(1,x)\omega(s)\|_{H^4(\frac{7}{8}\leq y\leq 4)},
   \end{align*}
  where $\psi(x,y)$ and $\Psi(t,x,y)$ are defined in \eqref{def: psi}. We drop the superscript $\delta$ throughout to simplify the notation.
  
  \medskip
 
\noindent\textbf{Proof of Proposition \ref{prop: uniform estimates for Em(t)}.} Taking $L^2$ inner product with $e^{2\Psi_{t_0}}\chi^2_m\psi^2\omega$ on both sides of \eqref{eq: NS vorticity}, we obtain
		\begin{align}\label{ineq:E-d}
			&\frac{1}{2}\frac{d}{dt}\|e^{\Psi_{t_0} }\psi\chi_m\omega\|^2_{L^2}
			+\frac{40\eps_0\gamma}{t+t_0}\int_{\Gamma(t)}e^{2\Psi_{t_0}}|\chi_m\psi\omega|^2\big(1-\gamma t-\theta(x,y)\big)_+ dxdy\\
			&\qquad+\frac{20\eps_0}{(t+t_0)^2}\|e^{\Psi_{t_0}} \psi\chi_m\omega\big(1-\gamma t-\theta(x,y)\big)_+\|_{L^2}^2
			-\langle\Delta(\chi_m\omega),e^{2\Psi_{t_0}}\chi_m\psi^2\omega\rangle \nonumber\\
			&\quad=-\langle U\cdot\nabla(\chi_m\omega),e^{2\Psi_{t_0}}\chi_m\psi^2\omega\rangle
			+\langle U\cdot(\nabla\chi_m)\omega-2(\nabla\chi_m)\cdot\nabla\omega-(\Delta\chi_m)\omega, e^{2\Psi_{t_0}}\chi_m\psi^2\omega\rangle,\nonumber
		\end{align}
		where we used  the following fact 
		\begin{align*}
			-\pa_t\Psi_{t_0}
			=\frac{40\eps_0\gamma}{t+t_0}\big(1-\gamma t-\theta(x,y)\big)_+ 
			+\frac{20\eps_0}{(t+t_0)^2}\big(1-\gamma t-\theta(x,y)\big)_+^2.
		\end{align*}

\underline{Dissipative term.}		
First, we note that
		\begin{align}\label{nable e^Psi}
			|\nabla(e^{2\Psi_{t_0}}\psi^2)|
			\leq \frac{C_0\eps_0}{t+t_0}\big(1-\gamma t-\theta(x,y)\big)_+\cdot e^{2\Psi_{t_0}}\psi^2  
			+C_0 e^{2\Psi_{t_0}}\psi^2,
		\end{align}
		which implies
		\begin{align*}
			&-\langle\Delta(\chi_m\omega),e^{2\Psi_{t_0}}\chi_m\psi^2\omega\rangle\\
			&=\|e^{\Psi_{t_0}}\psi\nabla(\chi_m\omega)\|^2_{L^2}
			+\langle \nabla(\chi_m\omega), \nabla(e^{2\Psi_{t_0}}\psi^2)\chi_m\omega\rangle\\
			&\geq \|e^{\Psi_{t_0}}\psi\nabla(\chi_m\omega)\|^2_{L^2}
			-C_0\int_{\mathbb R^2_+}e^{2\Psi_{t_0}}\psi^2|\chi_m\omega||\nabla(\chi_m\omega)|dxdy \\
			&\quad-\frac{C_0\eps_0}{t+t_0}\int_{\Gamma(t)}e^{2\Psi_{t_0}}\psi^2\big(1-\gamma t-\theta(x,y)\big)_+|\chi_m\omega||\nabla(\chi_m\omega)|dxdy\\
			&\geq \frac{4}{5}\|e^{\Psi_{t_0}}\psi\nabla(\chi_m\omega)\|^2_{L^2}-\frac{C_0\eps_0^2}{(t+t_0)^2}\|e^{\Psi_{t_0}} \psi\chi_m\omega\big(1-\gamma t-\theta(x,y)\big)_+\|_{L^2}^2
			-C_0\|e^{\Psi_{t_0} }\psi\chi_m\omega\|^2_{L^2}.
		\end{align*}
If we take $\e_0$ small enough and $\gamma$ large enough, the last two terms on the right-hand side of the above inequality can be absorbed by the left-hand side of \eqref{ineq:E-d}.
		
\medskip

\underline{Right-hand side of \eqref{ineq:E-d}.}		
		By Lemma \ref{velocity estimates 3}, we obtain
		\begin{align*}
			|\langle U\cdot\nabla(\chi_m\omega),e^{2\Psi_{t_0}}\chi_m\psi^2\omega\rangle|
			&\leq  \|U\|_{L^\infty(\operatorname{supp}\chi_m)}
			\|e^{\Psi_{t_0}}\psi\nabla(\chi_m\omega)\|_{L^2}
			\|e^{\Psi_{t_0} }\psi\chi_m\omega\|_{L^2} \\
			&\leq \frac{1}{10}\|e^{\Psi_{t_0}}\psi\nabla(\chi_m\omega)\|_{L^2}^2
			+C_0\big( E(t)+1\big)^4,
		\end{align*}
	where we used  \eqref{nable e^Psi} in the last step.
		
		Since $\Psi_{t_0}=0$ on $\operatorname{supp}\nabla\chi_m$ and by integration by parts, we apply Lemma \ref{velocity estimates 3} to obtain
		\begin{align*}
			&|\langle U\cdot\nabla\chi_m\omega-2\nabla\chi_m\cdot\nabla\omega-\Delta\chi_m\omega,e^{2\Psi_{t_0}}\chi_m\psi^2\omega\rangle|\\
			&\leq C_0\big(\|U\|_{L^\infty(\operatorname{supp}\chi_m)}+1\big)
			\Big(\|(1+|x|)\omega\|^2_{L^2(\frac{1}{4}\leq y\leq\frac{3}{8})}
			+\|\omega\|^2_{L^2(3\leq|(x,y)-(0,20)|\leq4)}\Big)\\
			&\leq C_0 \big( E(t)+1\big)
			\Big(\|(1+|x|)\omega\|^2_{L^2(\frac{1}{4}\leq y\leq\frac{3}{8})}
			+\|\omega\|^2_{L^2(3\leq|(x,y)-(0,20)|\leq4)}\Big).
		\end{align*}
Moreover, we have
		\begin{align*}
			\|(1+|x|)\omega\|^2_{L^2(\frac{1}{4}\leq y\leq\frac{3}{8})}
			\leq C_0\|(1,x)\omega(t)\|^2_{Y_2(t)}
			\leq C_0(E_b(t)+1)^2,
		\end{align*}
		and
		\begin{align*}
			&\|\omega\|^2_{L^2(3\leq|(x,y)-(0,20)|\leq4)}
			\leq \frac{\alpha^2}{t+\delta}\int_{\frac{3}{\sqrt{t+\delta}}\leq|\eta|\leq\frac{4}{\sqrt{t+\delta}}}|\mathcal W(\eta,\tau)|^2d\eta\\
			&\leq \frac{C_0}{t+\delta}\int_{\frac{3}{\sqrt{t+\delta}}\leq|\eta|\leq\frac{4}{\sqrt{t+\delta}}}|G(\eta)|^2 d\eta
			+\frac{C_0}{t+\delta}\int_{\frac{3}{\sqrt{t+\delta}}\leq|\eta|\leq\frac{4}{\sqrt{t+\delta}}}|\mathcal W_R(\eta,\tau)|^2d\eta\\
			&\leq C_0+C_0\|\mathcal W_R\|^2_{L^2(m)}
			\leq C_0\big(1+E_{vp}(t)\big)^2.
		\end{align*}
Combing the above estimates, we obtain
		\begin{align*}
			|\langle U\cdot\nabla\chi_m\omega-2\nabla\chi_m\cdot\nabla\omega-\Delta\chi_m\omega,e^{2\Psi_{t_0}}\chi_m\psi^2\omega\rangle|
			&\leq C_0\big( E(t)+1\big)^3.
		\end{align*}
		
		\medskip
		
Armed with these estimates,  integrating from $0$ to $t$ and letting $t_0\rightarrow0$ and $\eps_1$ sufficiently small, we obtain the desired result.
		
By the same argument, we also obtain the estimate for $\sup_{[0,t]}\|e^{\Psi }\psi\chi_m\omega\|_{L^4}$ which we leave to the reader.
\ef

\medskip
For the last term in $E_m(t)$, we adopt the same method in Proposition 3.6 in \cite{HWYZ} to deduce 
\begin{align*}
	e^{\frac{5\eps_0}{t}}\sup_{0<s<t}\|(1,x)\omega(s)\|_{H^4(\frac{7}{8}\leq y\leq 4)}
	\leq Ct^{1/2}\big(1+E(t)\big)^8.
\end{align*}

%
%

  \section{Energy estimates near the boundary}  \label{section-E_{b}}
 As in the previous sections, we drop the superscript $\delta$ throughout to simplify the notation.  We recall that the energy functional $E_b(t)$ is defined by
\begin{align*}
	E_b(t)=\|(1,x)\big(\omega (t)-\omega_c(t)\big)\|_{Y_1(t)\cap Y_2(t)}.
\end{align*}
  And $\omega-\omega_c$ satisfies 
\begin{align} \label{integral eq of omega-omega c}
	(\chi_b\omega-\chi_b\omega_c)_\xi(t,y)
 	=&\int_0^{+\infty}\big(H_\xi(t,y,z)+R_\xi(t,y,z)\big)b_\xi(z)dz\\
 	\nonumber
 	&+\int_0^t\int_0^{+\infty}\big(H_\xi(t-s,y,z)+R_\xi(t-s,y,z)\big)N_\xi(s,z)dzds\\
 	\nonumber
 	&-\int_0^t\big(H_\xi(t-s,y,0)+R_\xi(t-s,y,0)\big)B_\xi(s)ds,
\end{align}
and
\begin{align}\label{integral eq of x omega-omega c}
	( x\cdot(\chi_b \omega-\chi_b \omega_c))_\xi(t,y)
	=&\int_0^{+\infty}\big(H_\xi(t,y,z)+R_\xi(t,y,z)\big)\tilde b_\xi(z)dz\\
 	\nonumber
 	&+\int_0^t\int_0^{+\infty}\big(H_\xi(t-s,y,z)+R_\xi(t-s,y,z)\big)\widetilde N_\xi(s,z)dzds\\
 	\nonumber
 	&-\int_0^t\big(H_\xi(t-s,y,0)+R_\xi(t-s,y,0)\big)\widetilde B_\xi(s)ds,
\end{align}
where
\beno
\widetilde N &=&-\chi_b U\cdot\nabla(x\omega-x\omega_c)
		+\chi_b v(\omega-\omega_c)-2\chi_b\pa_x(\omega-\omega_c)-\chi_b xU\cdot\nabla\omega_c\\
		&&+\chi_b x\pa_x^2\omega_c-x(\chi_b''\omega+2\chi_b'\pa_y\omega)	+x(\chi_b''\omega_c+2\chi_b'\pa_y\omega_c),\\
	\tilde b&=&	xu_0^\delta \chi_b\chi_b',\\
	\widetilde B_\xi	&=&\big(i\pa_\xi(B_\xi)
	-i \text{sgn}\xi(\omega-\omega_c)_\xi\big)|_{y=0}.
\eeno

To simplify the notations, we introduce the following functional space $W_{\mu,s}$, whose norm is defined by 
\begin{align}\label{def: norm N consists of two types}
	\|f(s)\|_{W_{\mu,s}}
	:=\sum_{i+j\leq2}\left\|\pa_x^i(y\pa_y)^j f(s)\right\|_{Y^1_{\mu,s}\cap Y^2_{\mu,s}}
	+e^{\frac{2\eps_0}{s}}\sum_{i+j\leq3}\left\|\|\pa_x^i\pa_y^j f(s)\|_{L^2_x}\right\|_{L^1_y(y\geq1+\mu)}.
\end{align}
 For later convenience, we also introduce
\begin{align}\label{def of mu1}
	\mu_1:=\mu+\frac{1}{2}(\mu_0-\mu-\gamma s), 
\end{align}
which implies the relation $0<\mu<\mu_1<\mu_0-\gamma s$.

\subsection{Estimates of $N, \widetilde N$}
In this subsection, we provide estimates for the force terms $N, \widetilde N$. The following proposition is the main result.
\begin{proposition}\label{prop: estimate of N in new norm}
	For $0<\mu<\mu_0-\gamma s$, it holds that
	\begin{align*}
		\|N(s)\|_{W_{\mu,s}}
		+\|\widetilde N(s)\|_{W_{\mu,s}}
		\leq C_0(\mu_0-\mu-\gamma s)^{-\beta}\big( E(s)+1\big)^2. 
	\end{align*}
\end{proposition}

The proof of Proposition \ref{prop: estimate of N in new norm} follows from Lemmas \ref{est of N1} and Lemma \ref{est of N2}.

First, we deal with the first part of $\|\cdot\|_{W_{\mu,s}}$, which concerns the part near the boundary.
\begin{lemma}\label{est of N1}
	For $0<\mu<\mu_0-\gamma s$, it holds that
	\begin{align*}
		\sum_{i+j\leq2}\left\|\pa_x^i(y\pa_y)^j\big( N(s),\widetilde  N(s)\big)\right\|_{Y^1_{\mu,s}\cap Y^2_{\mu,s}}
		\leq C_0(\mu_0-\mu-\gamma s)^{-\beta}\big( E(s)+1\big)^2. 
	\end{align*}
	
\end{lemma}

\begin{proof}
Here, we only treat $\pa_x^2 N$, as the other cases are similar. By Lemma \ref{product estimate}, we obtain
\begin{align*}
	&\|\pa_x^2 N(s)\|_{Y^1_{\mu,s}\cap Y^2_{\mu,s}}\\
	&\leq  \sum_{i_1+i_2=2}\Big( \|\pa_x^{i_1}u \pa_x^{i_2+1}(\omega-\omega_c) \|_{Y^1_{\mu,s}\cap Y^2_{\mu,s}}
	+ \|\pa_x^{i_1}v \pa_x^{i_2}\pa_y(\omega-\omega_c) \|_{Y^1_{\mu,s}\cap Y^2_{\mu,s}}\\
	&\quad + \|\pa_x^{i_1}u \pa_x^{i_2+1}\omega_c \|_{Y^1_{\mu,s}\cap Y^2_{\mu,s}}
	+ \|\pa_x^{i_1}v \pa_x^{i_2}\pa_y\omega_c \|_{Y^1_{\mu,s}\cap Y^2_{\mu,s}} \Big)+\|\pa_x^4\omega_c\|_{Y^1_{\mu,s}\cap Y^2_{\mu,s}}\\
	&\leq \sum_{i_1+i_2=2}
  \|\sup_{0<y<1+\mu}e^{\eps_0(1+\mu-y)_+|\xi|}|\pa_x^{i_1} (u,\frac{v}{y})_\xi(s,y)| \|_{L^1_\xi} \\ 
		&\qquad\cdot\Big(\|\pa_x^{i_2}(\pa_x,y\pa_y)(\omega-\omega_c)(s)\|_{Y^1_{\mu,s}\cap Y^2_{\mu,s}}+\|\pa_x^{i_2}(\pa_x,y\pa_y)\omega_c\|_{Y^1_{\mu,s}\cap Y^2_{\mu,s}}\Big)
		+\|\pa_x^4\omega_c\|_{Y^1_{\mu,s}\cap Y^2_{\mu,s}}.
\end{align*}
	Using the definition of $Y_k(t)$ and Lemma \ref{velocity estimates 1}, we have
	\begin{align*}
		\|\pa_x^2 N(s)\|_{Y^1_{\mu,s}\cap Y^2_{\mu,s}}
	&\leq C_0\big( E(s)+1\big)
	\big(1+(\mu_0-\mu-\gamma s)^{-\beta}E(s)\big)+C_0\\
	&\leq C_0(\mu_0-\mu-\gamma s)^{-\beta}\big( E(s)+1\big)^2.
	\end{align*}
	where we used  \eqref{boundedness of omega c} in the last step.
\end{proof}

Next, we deal with the second part of $\|\cdot\|_{W_{\mu,s}}$, which concerns the part away from the boundary.

\begin{lemma}
	\label{est of N2}
	It holds that
	\begin{align*}
		e^{\frac{2\eps_0}{s}}\sum_{i+j\leq3}\left\|\left\|\pa_x^i\pa_y^j \big( N(s),\widetilde  N(s)\big)\right\|_{L^2_x}\right\|_{L^1_y(y\geq1)}
		\leq C_0\big( E (s)+1\big)^2.
	\end{align*}
\end{lemma}

\begin{proof}
	A direct computation gives
	\begin{align*}
		&e^{\frac{2\eps_0}{s}}\sum_{i+j\leq3}\left\|\left\|\pa_x^i\pa_y^j N(s)\right\|_{L^2_x}\right\|_{L^1_y(y\geq1)}\\
		&\leq  C_0 e^{\frac{2\eps_0}{s}}\sum_{k=0}^3 \sum_{i+j\leq k}\|\pa_x^i\pa_y^j U(s)\|_{L^\infty(1\leq y\leq3)} \sum_{i+j\leq 4-k}\left\|\left\|\pa_x^i\pa_y^j\omega\right\|_{L^2_x}\right\|_{L^1_y(1\leq y\leq3)}\\
		&\quad+C_0 e^{\frac{2\eps_0}{s}}\sum_{i+j\leq4}\left\|\left\|\pa_x^i\pa_y^j\omega\right\|_{L^2_x}\right\|_{L^1_y(1\leq y\leq3)}
		+C_0 e^{\frac{2\eps_0}{s}}\sum_{i+j\leq4}\left\|\left\|\pa_x^i\pa_y^j\omega_c\right\|_{L^2_x}\right\|_{L^1_y(1\leq y\leq3)}\\
		&:=I_1+I_2+I_3.
	\end{align*}
For $I_1$, we get by  Lemma \ref{velocity estimates 1} that
	\begin{align*}
		I_1&\leq C_0 e^{\frac{2\eps_0}{s}}\sum_{k=0}^3 \big( E (s)+\|\omega(s)\|_{H^{1+k}(\frac{7}{8}\leq y\leq4)}+1\big)\|\omega(s)\|_{H^{4-k}(\frac{7}{8}\leq y\leq4)}\\
		&\leq C_0\big( E (s)+1\big)^2.
	\end{align*}
For $I_2$ and $I_3$, we have
	\begin{align*}
		I_2+I_3
		\leq C_0\big( E (s)+1\big).
	\end{align*}

For $\widetilde N$, the same argument as above yields the same bound.
\end{proof}

\medskip

\subsection{Estimates of $B, \widetilde B$}

We now estimate the boundary terms. The following proposition is the main result.

\begin{proposition}
	\label{prop: estimate of B}
	For $0<\mu<\mu_0-\gamma s$, it holds that
	\begin{align*}
		\sum_{i\leq2}\left\|e^{\eps_0(1+\mu)|\xi|}\xi^i (B_\xi(s),\widetilde B_\xi(s)\big)\right\|_{L^1_\xi\cap L^2_\xi}
		\leq C\big( (\mu_0-\mu-\gamma s)^{-\beta}+s^{-1/2}\big)
		\big(1+ E(s)\big)^2.
	\end{align*}
\end{proposition}

The proof of Proposition \ref{prop: estimate of B} follows from Lemmas \ref{lem: estimate of B1} and Lemma \ref{lem: estimate of B2}.

\begin{lemma}\label{lem: estimate of B1}
	For $0<\mu<\mu_0-\gamma s$, it holds that
	\begin{align*}
		\sum_{i\leq2}\left\|e^{\eps_0(1+\mu)|\xi|}\xi^i B_\xi(s)\right\|_{L^1_\xi\cap L^2_\xi}
		\leq C\big( (\mu_0-\mu-\gamma s)^{-\beta}+s^{-1/2}\big)
		\big(1+ E(s)\big)^2.
	\end{align*}
\end{lemma}

\begin{proof}
	According to the definition of $B$, we utilize Lemma \ref{lem: velocity formula} to obtain
	\begin{align}\label{decomposition of B xi}
		B_\xi(s)&=\big(\pa_y\Delta_D^{-1}(U\cdot\nabla\omega)\big)_\xi|_{y=0}(s)
		-(\pa_y+|\xi|)(\omega_c)_\xi|_{y=0}(s)\\
		\nonumber
		&=-\int_0^{1+\mu}e^{-|\xi|z}(U\cdot\nabla\omega)_\xi(s,z)dz
		-\int_{1+\mu}^{+\infty}e^{-|\xi|z}(U\cdot\nabla\omega)_\xi(s,z)dz\\
		\nonumber
		&\quad-(\pa_y+|\xi|)(\omega_c)_\xi|_{y=0}(s)\\
		&:=I_1+I_2+I_3.	\nonumber
	\end{align}
	
	For $I_1$, the following fact
	\begin{align}\label{transfer 1}
		e^{\eps_0(1+\mu)|\xi|}e^{-|\xi|z}
		\leq e^{\eps_0(1+\mu-z)_+|\xi|}
	\end{align}
	implies
	\begin{align*}
		\left|e^{\eps_0(1+\mu)|\xi|}I_1\right|
		\leq \left\|e^{\eps_0(1+\mu-z)_+|\xi|} (u\pa_x\omega)_\xi(s,z)\right\|_{\mu,s}
		+\left\|e^{\eps_0(1+\mu-z)_+|\xi|} (v\pa_y\omega)_\xi(s,z)\right\|_{\mu,s}.
	\end{align*}
	Thus, we use Lemma \ref{product estimate} and Lemma \ref{velocity estimates 1} to obtain
	\begin{align*}
		&\sum_{i\leq2}\left\|e^{\eps_0(1+\mu)|\xi|}\xi^i I_1\right\|_{L^1_\xi\cap L^2_\xi}\\
		&\leq \sum_{i\leq2}\sum_{j+k\leq i}\left\|\sup_{0<y<1+\mu}e^{\eps_0(1+\mu-y)_+|\xi|}|(\pa_x^j u)_\xi(s,y)|\right\|_{L^1_\xi}
		\|\pa_x^{1+k}(\omega-\omega_c+\omega_c)(s)\|_{Y^1_{\mu,s}\cap Y^2_{\mu,s}}\\
		&\quad+\sum_{i\leq1}\sum_{j+k\leq i}\left\|\sup_{0<y<1+\mu}e^{\eps_0(1+\mu-y)_+|\xi|}\frac{|(\pa_x^j v)_\xi(s,y)|}{y}\right\|_{L^1_\xi}
		\|\pa_x^k(y\pa_y)(\omega-\omega_c+\omega_c)(s)\|_{Y^1_{\mu,s}\cap Y^2_{\mu,s}}\\
		&\leq C_0(\mu_0-\mu-\gamma s)^{-\beta}\big(1+ E(s) \big)^2.
	\end{align*}
	
	For $I_2$, integration by parts yields
	\begin{align*}
		I_2&=-\int_{1+\mu}^{+\infty}e^{-|\xi|z}\big(\dv(U\omega)\big)_\xi(s,z)dz\\
		&=-\int_{1+\mu}^{+\infty}e^{-|\xi|z}\big((i\xi)(u\omega)_\xi(s,z)+|\xi|(v\omega)_\xi(s,z)\big)dz
		+e^{-(1+\mu)|\xi|}(v\omega)_\xi(s,1+\mu)\\
		&=-\int_{1+\mu}^{+\infty}e^{-|\xi|z}\big((i\xi)(u\omega)_\xi(s,z)+|\xi|(v\omega)_\xi(s,z)\big)dz
		+e^{-(1+\mu)|\xi|}\int_0^{1+\mu}\pa_z(v\omega)_\xi(s,z)dz\\
		&=-\int_{1+\mu}^{+\infty}e^{-|\xi|z}\big((i\xi)(u\omega)_\xi(s,z)+|\xi|(v\omega)_\xi(s,z)\big)dz\\
		&\quad-e^{-(1+\mu)|\xi|}\int_0^{1+\mu}(\pa_x u\omega)_\xi(s,z)dz
		+e^{-(1+\mu)|\xi|}\int_0^{1+\mu}(v\pa_z\omega)_\xi(s,z)dz.
	\end{align*}
	Then we have for $i\leq2$
	\begin{align*}
		\left|e^{\eps_0(1+\mu)|\xi|}\xi^i I_2\right|
		&\leq C_0\int_{1+\mu}^{+\infty}e^{-\frac{|\xi|z}{2}}|(U\omega)_\xi(s,z)|dz
		+C_0\left\|e^{\eps_0(1+\mu-z)_+|\xi|}(\pa_x u\omega)_\xi(s,z)\right\|_{\mu,s}\\
		&\quad+C_0\left\|e^{\eps_0(1+\mu-z)_+|\xi|}( v \pa_z\omega)_\xi(s,z)\right\|_{\mu,s},
	\end{align*}
	which implies
	\begin{align*}
		&\sum_{i\leq2}\left\|e^{\eps_0(1+\mu)|\xi|}\xi^i I_2\right\|_{L^1_\xi\cap L^2_\xi}\\
		&\leq C_0 \int_{1+\mu}^{+\infty}\left\| e^{-\frac{|\xi|z}{2}}\right\|_{L^1_\xi\cap L^2_\xi} \|(U\omega)_\xi(s,z)\|_{L^\infty_\xi}dz
		+C_0\|(\pa_x u \omega)\|_{Y^1_{\mu,s}}
		+C_0\|( v \pa_y\omega)\|_{Y^1_{\mu,s}}\\
		&\leq C_0\|U\omega(s)\|_{L^1(y\geq1)}
		+C_0 \left\|\sup_{0<y<1+\mu}e^{\eps_0(1+\mu-y)_+|\xi|}|(\pa_x u)_\xi(s,y)|\right\|_{L^1_\xi}
		\|\omega(s)\|_{Y^1_{\mu,s}\cap Y^2_{\mu,s}}\\
		&\quad+C_0\left\|\sup_{0<y<1+\mu}e^{\eps_0(1+\mu-y)_+|\xi|}\frac{|v_\xi(s,y)|}{y}\right\|_{L^1_\xi}
		\|(y\pa_y)\omega(s)\|_{Y^1_{\mu,s}\cap Y^2_{\mu,s}}\\
		&\leq C_0\|U(s)\|_{L^\infty(y\geq1)}\|\omega(s)\|_{L^1(y\geq1)}
		+C_0\big(1+ E(s) \big)^2.
	\end{align*}
	Obviously, it holds that
	\begin{align*}
		\|\omega(s)\|_{L^1(y\geq1)}
		&\leq C_0\|\chi_{m}\omega(s)\|_{L^1}
		+C_0\|\chi_{vp}\omega(s)\|_{L^1}\\
		&\leq C_0\|e^\Psi \psi\chi_m\omega(s)\|_{L^2}
		+C_0 \|G+\mathcal W_R\|_{L^1}
		\leq C_0 \big( E(s)+1\big),
	\end{align*}
which along with Lemma \ref{velocity estimates 3} implies
\begin{align*}
		\sum_{i\leq2}\left\|e^{\eps_0(1+\mu)|\xi|}\xi^i I_2\right\|_{L^1_\xi\cap L^2_\xi}
		\leq C_0 s^{-1/2}\big( E(s)+1\big)^2.
	\end{align*}
	
	For $I_3$, due to \eqref{boundedness of omega c 2}, we have
	\begin{align*}
		\sum_{i\leq2}\left\|e^{\eps_0(1+\mu)|\xi|}\xi^i I_3\right\|_{L^1_\xi\cap L^2_\xi}
		\leq C_0 s^{-1/2}.
	\end{align*}
	
	Summing these estimates yields the desired result.
\end{proof}

\smallskip

\begin{lemma}\label{lem: estimate of B2}
	For $0<\mu<\mu_0-\gamma s$, it holds that
	\begin{align*}
		\sum_{i\leq2}\left\|e^{\eps_0(1+\mu)|\xi|}\xi^i \widetilde B_\xi(s)\right\|_{L^1_\xi\cap L^2_\xi}
		\leq C\big( (\mu_0-\mu-\gamma s)^{-\beta}+s^{-1/2}\big)
		\big(1+ E(s)\big)^2.
	\end{align*}
\end{lemma}

\begin{proof}
	Due to the relation \eqref{decomposition of B xi}: 
	\beno
	\widetilde B_\xi(s)
		=i\pa_\xi\big(B_\xi(s)\big)
		-isgn\xi(\omega-\omega_c)_\xi|_{y=0}(s),
\eeno
we have
		\begin{align*}
			&\sum_{i\leq2}\left\|e^{\eps_0(1+\mu)|\xi|}\xi^i \widetilde B_\xi(s)\right\|_{L^1_\xi\cap L^2_\xi}\\
			&\leq \sum_{i\leq2}\left\|e^{\eps_0(1+\mu)|\xi|}\xi^i \pa_\xi\big(B_\xi(s)\big)\right\|_{L^1_\xi\cap L^2_\xi}
			+\sum_{i\leq2}\left\|e^{\eps_0(1+\mu)|\xi|}\xi^i (\omega-\omega_c)_\xi|_{y=0}(s)\right\|_{L^1_\xi\cap L^2_\xi}.
		\end{align*}
By the same arguments as in the previous lemma, we have
	\begin{align*}
		\sum_{i\leq2}\left\|e^{\eps_0(1+\mu)|\xi|}\xi^i \pa_\xi\big(B_\xi(s)\big)\right\|_{L^1_\xi\cap L^2_\xi}
		\leq C\big( (\mu_0-\mu-\gamma s)^{-\beta}+s^{-1/2}\big)
	\big(1+ E(s) \big)^2.
	\end{align*}

It remains to handle $ (\omega-\omega_c)_\xi|_{y=0}$. Taking $y=0$ in \eqref{integral eq of omega-omega c} yields
	\begin{align*}
		(\omega-\omega_c)_\xi|_{y=0}(s)
		=&\int_0^{+\infty}\big(H_\xi(s,0,z)+R_\xi(s,0,z)\big)b_\xi(z)dz\\
		&+\int_0^s\int_0^{+\infty}\big(H_\xi(s-\tau,0,z)+R_\xi(s-\tau,0,z)\big)N_\xi(s,z)dzd\tau\\
		&-\int_0^s\big(H_\xi(s-\tau,0,0)+R_\xi(s-\tau,0,0)\big)B_\xi(\tau)d\tau\\
		&:=J_1+J_2+J_3.
	\end{align*}
	
	Based on the expression of $b_\xi$, $H_\xi$ \eqref{def of H xi R xi} and Lemma \ref{prop of R 1} (3), we have
	\begin{align*}
		\sum_{i\leq2}\left\|e^{\eps_0(1+\mu)|\xi|}\xi^i J_1\right\|_{L^1_\xi\cap L^2_\xi}
		\leq C_0.
	\end{align*}
	
	Utilizing \eqref{def of H xi R xi} and Lemma \ref{prop of R 1}, we deduce through a direct computation that
	\begin{align*}
		e^{\eps_0(1+\mu)|\xi|}|\xi|^i|H_\xi(s-\tau,0,z)|
		\leq 
		\left\{
		\begin{aligned}
			&\frac{C|\xi|^i}{(s-\tau)^{1/2}}e^{\eps_0(1+\mu-z)_+|\xi|},\qquad z<1+\mu,\\
			&C,\qquad z\geq1+\mu,
		\end{aligned}
		\right.
	\end{align*}
	and
	\begin{align*}
		e^{\eps_0(1+\mu)|\xi|}|\xi|^i|R_\xi(s-\tau,0,z)|
		\leq 
		\left\{
		\begin{aligned}
			&C\big(\frac{|\xi|^i}{(s-\tau)^{1/2}}+|\xi|^{1+i} \big) e^{\eps_0(1+\mu-z)_+|\xi|},\qquad z<1+\mu,\\
			&C,\qquad z\geq1+\mu.
		\end{aligned}
		\right.
	\end{align*}
	Thus, it holds that
	\begin{align*}
		&\sum_{i\leq2}\left\|e^{\eps_0(1+\mu)|\xi|}\xi^i J_2\right\|_{L^1_\xi\cap L^2_\xi}
		\leq C\int_0^s(s-\tau)^{-1/2}\sum_{i\leq2}\|\pa_x^i N(\tau)\|_{Y^1_{\mu,\tau}\cap Y^2_{\mu,\tau}}d\tau\\
		&\qquad +C\int_0^s \sum_{i\leq3}\|\pa_x^i N(\tau)\|_{Y^1_{\mu,\tau}\cap Y^2_{\mu,\tau}}d\tau
		+C\int_0^s \sum_{i\leq2}
		\left\|\|(\pa_x^i N)_\xi(\tau,z)\|_{L^1_z(z\geq1+\mu)}\right\|_{L^1_\xi\cap L^2_\xi}d\tau\\
		&\leq C\int_0^s \big((s-\tau)^{-1/2}
		+(\mu_0-\mu-\gamma \tau)^{-1}\big)
		\sum_{i\leq2}\|\pa_x^i N(\tau)\|_{Y^1_{\mu_2,\tau}\cap Y^2_{\mu_2,\tau}}d\tau\\
		&\qquad+C\int_0^s\sum_{i\leq3}\left\|\|\pa_x^i N(\tau,z)\|_{L^2_\xi}\right\|_{L^1_z(z\geq1+\mu)}d\tau\\
		&\leq C\int_0^s \big((s-\tau)^{-1/2}
		+(\mu_0-\mu-\gamma \tau)^{-1}\big)
		\|N(\tau)\|_{W_{\mu_2,\tau}}d\tau,
	\end{align*}
	where we take $\mu_2=\mu+\f12(\mu_0-\mu-\gamma \tau)$ and use Lemma \ref{analytic recovery} in the last two steps. By Proposition \ref{prop: estimate of N in new norm} and Lemma \ref{integral computation}, we have
	\begin{align*}
		\sum_{i\leq2}\left\|e^{\eps_0(1+\mu)|\xi|}\xi^i J_2\right\|_{L^1_\xi\cap L^2_\xi}
		&\leq \int_0^s \big((s-\tau)^{-1/2}
		+(\mu_0-\mu-\gamma \tau)^{-1}\big)
		(\mu_0-\mu-\gamma \tau)^{-\beta}
		\big(E(\tau)+1\big)^2d\tau\\
		&\leq \frac{C}{\gamma^{1/2}}(\mu_0-\mu-\gamma s)^{-\beta}
		\big(E(s)+1\big)^2 .
	\end{align*}
	
	For $J_3$, as in $J_2$, we use Lemma \ref{lem: estimate of B1} to obtain
	\begin{align*}
		&\quad\sum_{i\leq2}\left\|e^{\eps_0(1+\mu)|\xi|}\xi^i J_3\right\|_{L^1_\xi\cap L^2_\xi}\\
		&\leq C\int_0^s(s-\tau)^{-1/2}\sum_{i\leq2}\left\|e^{\eps_0(1+\mu)|\xi|}\xi^i B_\xi(\tau)\right\|_{L^1_\xi\cap L^2_\xi}d\tau
		+C\int_0^s \sum_{i\leq3}\left\|e^{\eps_0(1+\mu)|\xi|}\xi^i B_\xi(\tau)\right\|_{L^1_\xi\cap L^2_\xi}d\tau\\
		&\leq C\int_0^s \big((s-\tau)^{-1/2}
		+(\mu_0-\mu-\gamma \tau)^{-1}\big)\sum_{i\leq2}\left\|e^{\eps_0(1+\mu_2)|\xi|}\xi^i B_\xi(\tau)\right\|_{L^1_\xi\cap L^2_\xi}d\tau\\
		&\leq C\int_0^s \big((s-\tau)^{-1/2}
		+(\mu_0-\mu-\gamma \tau)^{-1}\big)
		\big((\mu_0-\mu-\gamma \tau)^{-\beta}+\tau^{-1/2}\big)\big(E(\tau)+1\big)^2 d\tau\\
		&\leq C(\mu_0-\mu-\gamma s)^{-\beta}\big(E(s)+1\big)^2,
	\end{align*}
	where we take $\mu_2=\mu+\f12(\mu_0-\mu-\gamma \tau)$ and use Lemma \ref{analytic recovery} in the last two steps.
	
	Summing these estimates, we obtain the desired result.
	\end{proof}

\subsection{Estimate of $E_b(t)$} 
 Before presenting the uniform estimates for $E_b(t)$, we state some semigroup estimates obtained similarly as  in \cite{HWYZ}.  
 
%

%
%

\begin{lemma}\label{est of HN}
	For $\mu<\mu_0-\gamma t$ and $\mu_1=\mu+\frac{1}{2}(\mu_0-\mu-\gamma s)$, we have
	\begin{align*}
		&\sum_{i+j\leq2}\left\|\pa_x^i(y\pa_y)^j\int_0^t\int_0^{+\infty}\big( H(t-s,y,z)+R(t-s,y,z)\big)N(s,z)dzds\right\|_{Y^1_{\mu,t}\cap Y^2_{\mu,t}}\\
		&\leq C\int_0^t \|N(s)\|_{W_{\mu,s}} ds,
	\end{align*}
	and
	\begin{align*}
		&\sum_{i+j\leq3}\left\|\pa_x^i(y\pa_y)^j\int_0^t\int_0^{+\infty}\big( H(t-s,y,z)+R(t-s,y,z)\big)N(s,z)dzds\right\|_{Y^1_{\mu,t}\cap Y^2_{\mu,t}}\\
		&\leq C\int_0^t\big((\mu_0-\mu-\gamma s)^{-1}+(\mu_0-\mu-\gamma s)^{-\frac{1}{2}}(t-s)^{-\frac{1}{2}}\big) \|N(s)\|_{W_{\mu_1,s}} ds.
	\end{align*}
\end{lemma}

\begin{lemma}\label{est of HB}
	For $\mu<\mu_0-\gamma t$ and $\mu_1=\mu+\frac{1}{2}(\mu_0-\mu-\gamma s)$, we have 
	\begin{align*}
		&\sum_{i+j\leq2}\left\|\pa_x^i(y\pa_y)^j\int_0^t \big( H(t-s,y,0)+R(t-s,y,0)\big)B(s)ds\right\|_{Y^1_{\mu,t}\cap Y^2_{\mu,t}}\\
		&\leq C\int_0^t\sum_{i\leq2}\left\| e^{\eps_0(1+\mu)|\xi|}\xi^i B_\xi(s)\right\|_{L^1_\xi\cap L^2_\xi}ds,
	\end{align*}
	and
	\begin{align*}
		&\quad\sum_{i+j\leq3}\left\|\pa_x^i(y\pa_y)^j\int_0^t \big( H(t-s,y,0)+R(t-s,y,0)\big)B(s)ds\right\|_{Y^1_{\mu,t}\cap Y^2_{\mu,t}}\\
		&\leq C\int_0^t(\mu_0-\mu-\gamma s)^{-1}
		\sum_{i\leq2}\left\| e^{\eps_0(1+\mu_1)|\xi|}\xi^i B_\xi(s)\right\|_{L^1_\xi\cap L^2_\xi}ds.
	\end{align*}
\end{lemma}

\begin{remark}
We stress that Lemma \ref{est of HN} and Lemma \ref{est of HB} remain valid if we replace $N$ by $\widetilde N$ or $B$ by $\widetilde B$.

\end{remark}

Now we are in a position to prove Proposition \ref{prop: uniform estimates for Eb(t)}.\smallskip

\noindent\textbf{Proof of Proposition \ref{prop: uniform estimates for Eb(t)}.} Based on Lemma \ref{est of HN}, Lemma \ref{est of HB}, Lemma \ref{est of HR b} and Proposition \ref{prop: estimate of N in new norm}, Proposition \ref{prop: estimate of B}, Proposition \ref{prop: Sobolev estimate}, we deduce that
\begin{align*}
	&\quad\|\omega(t)-\omega_c(t)\|_{Y_1(t)\cap Y_2(t)}\\
	&\leq Ce^{-\frac{C'}{ t}}
	+C\int_0^t \Big(\|N(s)\|_{W_{\mu,s}}
	+\sum_{i\leq1}\left\|e^{\eps_0(1+\mu)|\xi|}\xi^i B_\xi(s)\right\|\Big)ds\\
	&\quad+C(\mu_0-\mu-\gamma t)^\beta \int_0^t\big((\mu_0-\mu-\gamma s)^{-1}
	+(\mu_0-\mu-\gamma s)^{-1/2}(t-s)^{-1/2}\big)  \|N(s)\|_{W_{\mu_1,s}} ds\\
	&\quad+C(\mu_0-\mu-\gamma t)^\beta \int_0^t (\mu_0-\mu-\gamma s)^{-1}
	\sum_{i\leq2}\left\|e^{\eps_0(1+\mu_1)|\xi|}\xi^i B_\xi(s)\right\|ds\\
	&\leq Ce^{-\frac{C'}{ t}}
	+ \int_0^t \big((\mu_0-\mu-\gamma s)^{-\beta}+s^{-1/2}\big)
	 \big( E(s)+1\big)^2 ds\\
		&\quad+C(\mu_0-\mu-\gamma t)^\beta \int_0^t\big((\mu_0-\mu-\gamma s)^{-1}
	+(\mu_0-\mu-\gamma s)^{-1/2}(t-s)^{-1/2}\big)\\
	&\quad\quad\quad\cdot (\mu_0-\mu-\gamma s)^{-\beta}
	 \big( E(s)+1\big)^2 ds\\
		&\quad+C(\mu_0-\mu-\gamma t)^\beta \int_0^t (\mu_0-\mu-\gamma s)^{-1} \big((\mu_0-\mu-\gamma s)^{-\beta}+s^{-1/2}\big) \big( E(s)+1\big)^2 ds\\
		&\leq C(\gamma^{-1/2}+t^{1/2})  \big( E(t)+1\big)^2,
\end{align*}
where $\mu_1=\mu+\f12(\mu_0-\mu-\gamma s)$ and we used Lemma \ref{integral computation} in the last step.

By the same way, we deduce that $\|x\omega(t)-x\omega_c(t)\|_{Y_1(t)\cap Y_2(t)}$ admits the same bound, which completes the proof.
\ef

\section{Proof of Theorem \ref{main theorem}:  uniqueness part}

In this section, we prove the uniqueness part of Theorem \ref{main theorem}. The main result is as follows.
\begin{proposition}\label{prop: uniqueness part}
Let the initial vorticity $\omega_0$ be given by  \eqref{initial data}. Assume that there exist two solutions $(\omega_1, U_1)$ and $(\omega_2, U_2)$ to \eqref{eq: NS vorticity} such that 
	\begin{align}
		\lim_{t\rightarrow0}E_1(t)= \lim_{t\rightarrow0}E_2(t) =0,
	\end{align}
where $E_i(t)$ is the energy functional defined in \eqref{def: E(t)} with $\delta$ dropped and corresponds to the solution $(\omega_i, U_i)$. Then, it holds that
	$\omega_1=\omega_2$ and $U_1=U_2$ for all $t\geq 0$.
 
\end{proposition}

To prove Proposition \ref{prop: uniqueness part},  we set 
\beno
\overline\omega=\omega_1-\omega_2, \quad \overline U=U_1-U_2.
\eeno
We aim to prove 
\beno
\overline\omega=\overline U=0.
\eeno
Since the argument closely follows the proof of the existence part, we only give a sketch here.  

\subsection{Functional framework}

Since $\omega_1, \omega_2$ both satisfy \eqref{eq: NS near vortex point}, thus $(\overline\omega, \overline U)$ satisfies
\begin{align}\label{eq: NS differ near the vortex point}
	&\pa_t(\chi_{vp}\overline\omega)
	+\{ BS_{\mathbb R^2_+}[\chi_{vp}\omega_1]\cdot\nabla(\chi_{vp}\overline\omega)
	+BS_{\mathbb R^2_+}[\chi_{vp}\overline\omega]\cdot\nabla(\chi_{vp}\omega_2)\}
	-\Delta(\chi_{vp}\overline\omega)\\
	\nonumber
	&=-\{ BS_{\mathbb R^2_+}[(1-\chi_{vp})\omega_1]\cdot\nabla(\chi_{vp}\overline\omega)
	+BS_{\mathbb R^2_+}[(1-\chi_{vp})\overline\omega]\cdot\nabla(\chi_{vp}\omega_2)\}\\
	\nonumber
	&\quad+\{ U_1\cdot\nabla\chi_{vp}\overline\omega
	+\overline U\cdot\nabla\chi_{vp} \omega_2\}
	-2\nabla\chi_{vp}\cdot\nabla\overline\omega
	-\Delta\chi_{vp}\overline\omega.
\end{align}
As in \eqref{def: self-similar}, we introduce self-similar transformation for $i=1,2$
\begin{align}\label{SS transform for omega differ}
\begin{split}
	&\chi_{vp}\omega_i(t,x,y)
	=\frac{\alpha}{t}\mathcal W_i(\eta,\tau),
	\quad \chi_{vp}\overline\omega(t,x,y)
	=\frac{\alpha}{t}\overline{\mathcal W}(\eta,\tau),\\
	&\eta=\frac{(x,y)-(0,20)}{t^{1/2}},\quad
	\tau=\log t.
	\end{split}
\end{align}
Then it obviously holds that
\begin{align*}
	\operatorname{supp}\mathcal W_i, \overline{\mathcal W}\subseteq \{|\eta|\leq 6t^{-1/2}\},\quad
	\lim_{\tau\rightarrow-\infty}\mathcal W_i=\chi_{vp} G,\quad
	\lim_{\tau\rightarrow-\infty}\overline{\mathcal W}=0.
\end{align*}
For the velocity, Lemma \ref{derivation of velocity formula} yields
\begin{align*}
	&BS_{\mathbb R^2_+}[\chi_{vp}\omega_i](t,x,y)
	=\frac{\alpha}{t^{1/2}}
	\{ \mathcal V_i(\eta,\tau)
	-\widetilde{\mathcal V_i}(\eta+\frac{(0,40)}{t^{1/2}},\tau) \},\quad \mathcal V_i:=BS_{\mathbb R^2}[\mathcal W_i], \\
	&BS_{\mathbb R^2_+}[\chi_{vp}\overline\omega](t,x,y)
	=\frac{\alpha}{t^{1/2}}
	\{ \overline{\mathcal V}(\eta,\tau)
	-\widetilde{\overline{\mathcal V}}(\eta+\frac{(0,40)}{t^{1/2}},\tau) \},\quad \overline{\mathcal V}:=BS_{\mathbb R^2}[\overline{\mathcal W}].
\end{align*}
We introduce the following decomposition
\begin{align*}
	&\mathcal W_i:=\chi_{vp} G+\mathcal W_{i,R},\quad
	\overline{\mathcal W}:=\mathcal W_{1,R}-\mathcal W_{2,R},\\
	&\mathcal V_i:=\mathcal V^G-BS_{\mathbb R^2}[(1-\chi_{vp})G]+\mathcal V_{i,R},\quad
	\overline{\mathcal V}:=\mathcal V_{1,R}-\mathcal V_{2,R}.
\end{align*}
Then $\overline{\mathcal W}$ satisfies
\begin{align*}
	&\pa_\tau\overline{\mathcal W}
	+\alpha\mathcal V^G\cdot\nabla_\eta\overline{\mathcal W}
	+\alpha\overline{\mathcal V}\cdot\nabla_\eta G
	-\mathcal L\overline{\mathcal W}
	:= f_1+f_2+f_3+f_4.
\end{align*}
where $f_1 \sim f_4$ are defined by
\begin{align*}
	f_1=&\alpha\{\mathcal V^G(\eta+\frac{(0,40)}{e^{\tau/2}})
	+BS_{\mathbb R^2}[(1-\chi_{vp})G](\eta)
	-\widetilde{BS_{\mathbb R^2}[(1-\chi_{vp})G]}(\eta+\frac{(0,40)}{e^{\tau/2}})\\
	&-\mathcal V_{1,R}(\eta,\tau)
	+\widetilde{\mathcal V_{1,R}}(\eta+\frac{(0,40)}{e^{\tau/2}},\tau)\}\cdot\nabla_\eta \overline{\mathcal W},
\end{align*}
and 
\begin{align*}
	&f_2=-\alpha\{ \overline{\mathcal V}(\eta,\tau)
	-\widetilde{\overline{\mathcal V}}(\eta+\frac{(0,40)}{e^{\tau/2}},\tau) \}
	\cdot\nabla_\eta\{\mathcal W_{2,R}-(1-\chi_{vp})G\},\\
	&f_3=-e^{\tau/2}\{BS_{\mathbb R^2_+}[(1-\chi_{vp})\omega_1]\cdot\nabla_\eta\overline{\mathcal W}
	+BS_{\mathbb R^2_+}[(1-\chi_{vp})\overline\omega]\cdot\nabla_\eta\mathcal W_2 \},\\
&f_4=\frac{t^2}{\alpha}\{U_1\cdot\nabla\chi_{vp}\overline\omega
	+\overline U\cdot\chi_{vp}\omega_2
	-2\nabla\chi_{vp}\cdot\nabla\overline\omega
	-\Delta\chi_{vp}\overline\omega\}.
\end{align*}
Thus, the integral equation of $\overline{\mathcal W}$ becomes
\begin{align*}
	\overline{\mathcal W}(\tau)
	=\int_{-\infty}^\tau T_\alpha(\tau-\tau')(f_1+f_2+f_3+f_4)(\tau')d\tau'.
\end{align*}

We define 
\begin{align}\label{differ Evp(t)}
	\overline E_{vp}(t)
	:=\sup_{\tau'<\log t}\big(\left\|\overline{\mathcal W}(\tau')\right\|_{L^2(m)}
	+\left\|\nabla_\eta\overline{\mathcal W}(\tau')\right\|_{L^2(m)}\big),
\end{align}
and $E_{i,vp}(t)$ by replacing $\overline{\mathcal W}$ by $\mathcal W_{i,R}$.

For the middle region, we define
\begin{align}\label{differ Em(t)}
	\overline E_m(t)
	:=&\sup_{0<s<t}\left\|e^\Psi\psi\chi_m\overline\omega(s)\right\|_{L^2\cap L^4}+  \|e^\Psi\psi\nabla(\chi_m\overline\omega)\|_{L^2(0, t; L^2)}\\
	\nonumber
	&+e^{\frac{5\eps_0}{t}}\sup_{0<s<t}\|(1,x)\overline\omega(s)\|_{H^4(\frac{7}{8}\leq y\leq 4)},
\end{align}
and $E_{i,m}(t)$ by replacing $\overline\omega$ by $\omega_i$.

Near the boundary, we define
\begin{align}\label{differ Eb(t)}
	\overline E_b(t)
	:=\left\|(1,x)\overline\omega\right\|_{Y_1(t)\cap Y_2(t)},\quad
	E_{i,b}(t):=\|(1,x)(\omega_i-\omega_c)\|_{Y_1(t)\cap Y_2(t)},
\end{align}
here $\omega_c$ is the initial layer corrector defined by letting $\delta\rightarrow0^+$ in \eqref{omega c}. Obviously, Lemma \ref{est of omega c} remains valid with the index $\delta$ dropped.

The total energy functionals are defined by
\begin{align}\label{def: differ E(t)}
	\overline{ E}(t):=\overline E_{vp}(t)+\overline E_{m}(t)+\overline E_b(t),\qquad
	\ E_i(t)=E_{i,vp}(t)+E_{i,m}(t)+E_{i,b}(t).
\end{align}
We  aim to show $\overline{ E}(t)=0$ in $0<t\leq T_0$, therefore $\omega_1=\omega_2$ for all $t>0$.

\subsection{Proof of Proposition \ref{prop: uniqueness part}}

To establish the uniqueness part, the key point is to derive the uniform estimates for 
$\overline E(t)$. By arguments similar to those in Proposition \ref{prop: uniform estimates for Evp(t)}, Proposition \ref{prop: uniform estimates for Em(t)} and
Proposition \ref{prop: uniform estimates for Eb(t)}, we obtain the uniform estimates of $\overline E_{vp}(t)$, $\overline E_m(t)$ and $\overline E_b(t)$. Here, we omit the details. 

\begin{proposition}\label{prop: uniform estimate for overline E(t)}
For $t\in(0, T_0)$ with $T_0$ small enough, it holds that
	\begin{align*}
	\overline{ E}(t)
	\leq C \big(t^{3/8}
	+E_1(t)\big)\overline{ E}(t)
	+C t^{1/4}\overline{ E}(t)\big(\overline{ E}(t)+1\big)
	+\frac{C\overline{ E}(t)}{\gamma^{1/2}}
		+Ct^{1/2}\overline{ E}(t)^8.
\end{align*}
\end{proposition}

\medskip

\noindent{\bf Proof of Proposition \ref{prop: uniqueness part}:}
By Proposition \ref{prop: uniform estimate for overline E(t)}, we take $\eps$ in Proposition \ref{prop: uniqueness part} small enough, $t$ small and $\gamma$ large to obtain 
\beno
\overline E(t)=0, \quad t\in[0, T_0],
\eeno
where $T_0$ small enough.  Thus, by a continuous argument, we obtain $\omega_1=\omega_2$ for all time. This completes the uniqueness part.
\ef

\medskip

\smallskip

\appendix

\section{Some technical lemmas}

We enumerate several useful lemmas in the Appendix.\smallskip

The following  lemma is used to treat the loss of derivative and is proved in \cite{HWYZ}.
\begin{lemma}\label{analytic recovery}
	For $\widetilde \mu>\mu\geq0$, we have
	\begin{align*}
		e^{\eps_0(1+\mu-y)_+|\xi|}|(\pa_x f)_\xi(y)|
		&\leq \frac{C}{\widetilde\mu-\mu}
		e^{\eps_0(1+\widetilde\mu-y)_+|\xi|}|f_\xi(y)|,\\
		e^{\eps_0(1+\mu)\frac{y^2}{t}}\left|y\pa_y H_\xi(t-s,y,z)\right|&\leq\frac{C}{\sqrt{(\widetilde\mu-\mu)(t-s)}}
		e^{\eps_0(1+\widetilde\mu)\frac{y^2}{t}}
		\frac{1}{\sqrt{\nu(t-s)}}
		e^{-\frac{(y-z)^2}{5(t-s)}}
		e^{-\xi^2(t-s)}.
	\end{align*}
 
\end{lemma}

We require the following product estimates to treat the nonlinear terms.
\begin{lemma}\label{product estimate}
	For $0<\mu<\mu_0-\gamma s$, we have
	\begin{align*}
		\|fg\|_{Y^1_{\mu,s}\cap Y^2_{\mu,s}}
		\leq \left\|\sup_{0<y<1+\mu}e^{\eps_0(1+\mu-y)_+|\xi|}|f_\xi(s,y)|\right\|_{L^1_\xi}\cdot\|g(s)\|_{Y^1_{\mu,s}\cap Y^2_{\mu,s}}.
	\end{align*}
\end{lemma}

The following lemma  is used to close the uniform boundedness of $\omega$ near the boundary.

\begin{lemma}\label{integral computation}
	For $\frac{1}{2}<\beta<1$, $0<\zeta<1$, $\gamma>0$ and $\mu<\mu_0-\gamma t$, it holds that
	\begin{align*}
		(\mu_0-\mu-\gamma t)^\beta\int_0^t(\mu_0-\mu-\gamma s)^{-1-\beta}ds&
		\leq \frac{C}{\gamma},\\
		(\mu_0-\mu-\gamma t)^\beta\int_0^t(\mu_0-\mu-\gamma s)^{-\frac{1}{2}-\beta}(t-s)^{-\f12}ds&\leq \frac{C}{\gamma^{\f12}},\\
		\sup_{\mu<\mu_0-\gamma t}(\mu_0-\mu-\gamma t)^\zeta \ln\frac{\mu_0-\mu}{\mu_0-\mu-\gamma t}&\leq C(\gamma t)^\zeta,\\
		(\mu_0-\mu-\gamma t)^\beta\int_0^t(\mu_0-\mu-\gamma s)^{-1}s^{-1/2}ds &\leq \frac{C}{\gamma^{\f12}}.
	\end{align*}
	here $C$ is a constant depending on $\mu_0$, $\beta$ and $\zeta$.
\end{lemma}

For the proofs of Lemmas \ref{analytic recovery}, \ref{product estimate}, and  \ref{integral computation}, we refer to \cite{HWYZ}.

\medskip

Next, we establish the relationship between $BS_{\mathbb R^2}[w]$ and $w$ via classical elliptic estimates.

\begin{lemma}\label{lem: velocity weighted est}
(1)
	For $m\in(0,1)$ and $2<q<+\infty$, let $U=BS_{\mathbb R^2}[w]$, there exists $C_{m,q}$ such that
	\begin{align*}
		\|\langle\eta\rangle^{m-\frac{2}{q}}U(\eta)\|_{L^q}
		\leq C_{m,q}\|w\|_{L^2(m)}.
	\end{align*} 
	
	(2) Let $U=BS_{\mathbb R^2}[w]$ or $U=BS_{\mathbb R^2_+}[w]$. Then we have
	\begin{align*}
		\|U\|_{L^{\frac{2p}{2-p}}}
		\leq C_0\|\omega\|_{L^p},\quad for \quad 1<p<2,
	\end{align*}
	\begin{align*}
		\|U\|_{L^{2,\infty}}
		\leq C_0\|\omega\|_M.
	\end{align*}
	and
	\begin{align*}
		\|U\|_{L^\infty}
		\leq C_0 \|\omega\|_{L^{4/3}}^{1/2}\|\omega\|_{L^4}^{1/2}.
	\end{align*}
	Moreover, if a subset $A\subseteq \mathbb R^2(or \ \mathbb R^2_+)$ satisfies $d=\operatorname{dist}(A, \operatorname{supp} \omega)>0$, then it holds that
	\begin{align}
		\|U\|_{L^\infty(A)}\leq C(d)\|\omega\|_{L^1}.
	\end{align}
\end{lemma}

\medskip

The following formulation is to characterize the velocity near the point vortex. 
\begin{lemma}\label{derivation of velocity formula}
	It holds that
	\begin{align}
  	BS_{\mathbb R^2_+}[\chi_{vp}\omega^\delta](t,x,y)
  	=\frac{\alpha}{(t+\delta)^{1/2}}\{\mathcal V^\delta\big(\eta,\tau)-
  	\widetilde{\mathcal V^\delta}(\eta+\frac{(0,40)}{(t+\delta)^{1/2}},\tau)\},
  \end{align}
 where $\mathcal V^\delta:=BS_{\mathbb R^2}[\mathcal W^\delta]$ with  $\mathcal W^\delta \big(\eta,\tau\big)=\frac{t+\delta}{\alpha} \chi_{vp}\omega^\delta$ and $\widetilde {\mathcal V^\delta}(x,y)=(- \mathcal V^\delta_1(x,-y),  \mathcal V^\delta_2(x,-y)).$
\end{lemma}

\begin{proof}
Denote $X=(x_1,x_2) ,Y=(y_1,y_2)$. Based on \eqref{BS law formulation 2} and definition of $\chi_{vp}$, we obtain
	\begin{align*}
		BS_{\mathbb R^2_+}[\chi_{vp}\omega^\delta](t,X)
		&=\frac{1}{2\pi}\int_{\mathbb R^2}\frac{(X-Y)^\perp}{|X-Y|^2}\cdot\frac{\alpha}{t+\delta}\mathcal W^\delta\big(\frac{Y-(0,20)}{(t+\delta)^{1/2}},\tau)dY\\
		&\quad-\frac{1}{2\pi}\int_{\mathbb R^2}\frac{(X-Y^*)^\perp}{|X-Y^*|^2}\cdot\frac{\alpha}{t+\delta}\mathcal W^\delta\big(\frac{Y-(0,20)}{(t+\delta)^{1/2}},\tau)dY\\
		&=\frac{\alpha}{(t+\delta)^{1/2}}\cdot\frac{1}{2\pi}\int_{\mathbb R^2}\frac{\big(\frac{X-(0,20)}{(t+\delta)^{1/2}}-Z\big)^\perp}{\left|\frac{X-(0,20)}{(t+\delta)^{1/2}}-Z\right|^2}\mathcal W^\delta(Z,\tau)dZ\\
		&\quad-\frac{\alpha}{(t+\delta)^{1/2}}\cdot\frac{1}{2\pi}\int_{\mathbb R^2}\frac{\big(\frac{X+(0,20)}{(t+\delta)^{1/2}}-Z^*\big)^\perp}{\left|\frac{X+(0,20)}{(t+\delta)^{1/2}}-Z^*\right|^2}\mathcal W^\delta(Z,\tau)dZ\\
		&=\frac{\alpha}{(t+\delta)^{1/2}}\{\mathcal V^\delta\big(\eta,\tau)-
  	\widetilde{\mathcal V^\delta}(\eta+\frac{(0,40)}{(t+\delta)^{1/2}},\tau)\},
	\end{align*}
	where we used $\eta=\frac{X-(0,20)}{(t+\delta)^{1/2}}$ and the definition of $\widetilde{\mathcal V^\delta}$.
	\end{proof}

\section*{Acknowledgement}

C. Wang is supported by NSF of China under Grant 12471189. Z. Zhang is supported by NSF of China under Grant 12288101.


\begin{thebibliography}{90}
 
 

\bibitem{Ken} K. Abe, {\it The vorticity equations in a half plane with measures as initial data},  Ann.Inst.H.Poincar\'{e} C Anal. Non Lin\'{e}aire, 38(2021), 1055-1094.
 
 
 \bibitem{BA}  M. Ben-Artzi, {\it Global solutions of two-dimensional Navier-Stokes and Euler equations}, Arch. Rational Mech. Anal., 128 (1994), 329-358.
 




\bibitem{Cottet} G. H. Cottet,{\it Equations de Navier-Stokes dans le plan avec tourbillon initial mesure}, C. R. Acad. Sci. Paris Ser. I Math., 303 (1986), 105-108.

\bibitem{Dalibard} A. L. Dalibard and T. Gallay, {\it Viscous evolution of a point vortex in a half-plane}, arXiv:2603.21796.


\bibitem{Gallagher} I. Gallagher and T. Gallay, {\it Uniqueness for the two-dimensional Navier-Stokes equation with a measure as initial vorticity}, Math. Ann., 332(2005), 287-327.


\bibitem{Gallay 1} T. Gallay and C. E. Wayne, {\it Global stability of vortex solutions of the two dimensional Navier-Stokes equation}, Comm. Math. Phys., 255(2005), 97-129.




\bibitem{GMO}Y. Giga, T. Miyakawa, and H. Osada, {\it Two-dimensional Navier-Stokes flow with measures as initial vorticity}, Arch. Rational Mech. Anal., 104 (1988), 223-250.






\bibitem{HWYZ} J. Huang, C. Wang, J. Yue and Z. Zhang, {\it The interaction between rough vortex patch and boundary layer}, arXiv:2412.03198.









\bibitem{Kukavica} I. Kukavica, V. Vicol  and F. Wang, {\it The inviscid limit for the Navier-Stokes quations with data analytic only near the boundary}, Arch. Ration. Mech. Anal., 237(2020), 779-827.





\bibitem{Maekawa} Y. Maekawa, {\it On the inviscid limit problem of the vorticity equations for viscous incompressible flows in the half-plane}, Commun. Pure Appl. Math., 67(2014), 1045-1128.




\bibitem{TT Nguyen} T. T. Nguyen and  T. T. Nguyen, {\it The Inviscid Limit of Navier-Stokes Equations for analytic data on the half-space}, Arch. Ration. Mech. Anal., 230(2018), 1103-1129.








	
\end{thebibliography}
\end{document}